\documentclass[12pt]{amsart}
\usepackage[colorlinks=true, pdfstartview=FitV, linkcolor=blue, citecolor=blue, urlcolor=blue]{hyperref}

\usepackage{amssymb,amsmath, amscd}
\usepackage{times, verbatim}
\usepackage{graphicx}
\usepackage[english]{babel}
 \usepackage[usenames, dvipsnames]{color}
\usepackage{amsmath,amssymb,amsfonts}
\usepackage{enumerate}
\usepackage{anysize}
\marginsize{3cm}{3cm}{3cm}{3cm}
\input xy
\xyoption{all}
\usepackage{pb-diagram}
\usepackage[all]{xy}
\input xy
\xyoption{all}

\DeclareFontFamily{OT1}{rsfs}{}
\DeclareFontShape{OT1}{rsfs}{n}{it}{<-> rsfs10}{}
\DeclareMathAlphabet{\mathscr}{OT1}{rsfs}{n}{it}

\begin{document}
\theoremstyle{plain}

\newtheorem{theorem}{Theorem}[section]
\newtheorem{thm}[equation]{Theorem}
\newtheorem{prop}[equation]{Proposition}
\newtheorem{cor}[equation]{Corollary}
\newtheorem{conj}[equation]{Conjecture}
\newtheorem{lemma}[equation]{Lemma}
\newtheorem{definition}[equation]{Definition}
\newtheorem{question}[equation]{Question}

\theoremstyle{definition}
\newtheorem{conjecture}[theorem]{Conjecture}

\newtheorem{example}[equation]{Example}
\numberwithin{equation}{section}

\newtheorem{remark}[equation]{Remark}

\newcommand{\Hecke}{\mathcal{H}}
\newcommand{\Liea}{\mathfrak{a}}
\newcommand{\Cmg}{C_{\mathrm{mg}}}
\newcommand{\Cinftyumg}{C^{\infty}_{\mathrm{umg}}}
\newcommand{\Cfd}{C_{\mathrm{fd}}}
\newcommand{\Cinftyfd}{C^{\infty}_{\mathrm{ufd}}}
\newcommand{\sspace}{\Gamma \backslash G}
\newcommand{\PP}{\mathcal{P}}
\newcommand{\bfP}{\mathbf{P}}
\newcommand{\bfQ}{\mathbf{Q}}
\newcommand{\Siegel}{\mathfrak{S}}
\newcommand{\g}{\mathfrak{g}}
\newcommand{\A}{\mathbb{A}}
\newcommand{\Q}{\mathbb{Q}}
\newcommand{\Gm}{\mathbb{G}_m}
\newcommand{\Nm}{\mathbb{N}m}
\newcommand{\ii}{\mathfrak{i}}
\newcommand{\II}{\mathfrak{I}}

\newcommand{\kk}{\mathfrak{k}}
\newcommand{\nn}{\mathfrak{n}}
\newcommand{\tF}{\tilde{F}}
\newcommand{\p}{\mathfrak{p}}
\newcommand{\m}{\mathfrak{m}}
\newcommand{\bb}{\mathfrak{b}}
\newcommand{\Ad}{{\rm Ad}\,}
\newcommand{\ttt}{\mathfrak{t}}
\newcommand{\frakt}{\mathfrak{t}}
\newcommand{\U}{\mathcal{U}}
\newcommand{\Z}{\mathbb{Z}}
\newcommand{\bfG}{\mathbf{G}}
\newcommand{\bfT}{\mathbf{T}}
\newcommand{\R}{\mathbb{R}}
\newcommand{\ST}{\mathbb{S}}
\newcommand{\h}{\mathfrak{h}}
\newcommand{\bC}{\mathbb{C}}
\newcommand{\C}{\mathbb{C}}
\newcommand{\F}{\mathbb{F}}
\newcommand{\N}{\mathbb{N}}
\newcommand{\qH}{\mathbb {H}}
\newcommand{\temp}{{\rm temp}}
\newcommand{\Hom}{{\rm Hom}}
\newcommand{\Aut}{{\rm Aut}}
\newcommand{\rk}{{\rm rk}}
\newcommand{\Ext}{{\rm Ext}}
\newcommand{\End}{{\rm End}\,}
\newcommand{\Ind}{{\rm Ind}}
\newcommand{\ind}{{\rm ind}}
\newcommand{\Irr}{{\rm Irr}}
\def\circG{{\,^\circ G}}
\def\M{{\rm M}}
\def\diag{{\rm diag}}
\def\Ad{{\rm Ad}}
\def\As{{\rm As}}
\def\wG{{\widehat G}}
\def\G{{\rm G}}
\def\SL{{\rm SL}}
\def\PSL{{\rm PSL}}
\def\GSp{{\rm GSp}}
\def\PGSp{{\rm PGSp}}
\def\Sp{{\rm Sp}}
\def\St{{\rm St}}
\def\GU{{\rm GU}}
\def\SU{{\rm SU}}
\def\U{{\rm U}}
\def\GO{{\rm GO}}
\def\GL{{\rm GL}}
\def\PGL{{\rm PGL}}
\def\GSO{{\rm GSO}}
\def\GSpin{{\rm GSpin}}
\def\GSp{{\rm GSp}}

\def\Gal{{\rm Gal}}
\def\SO{{\rm SO}}
\def\O{{\rm  O}}
\def\Sym{{\rm Sym}}
\def\sym{{\rm sym}}
\def\St{{\rm St}}
\def\Sp{{\rm Sp}}
\def\tr{{\rm tr\,}}
\def\ad{{\rm ad\, }}
\def\Ad{{\rm Ad\, }}
\def\rank{{\rm rank\,}}

\subjclass{Primary 11F70; Secondary 22E55}

\title{Twisted GGP Problems and Conjectures, GGP IV}

\author{Wee Teck Gan, Benedict H. Gross and Dipendra Prasad}
\thanks{WTG is partially supported by an 
MOE Tier 1 grant R-146-000-320-114.
  DP thanks  Science and Engineering research board of the
  Department of Science and Technology, India for its support
through the JC Bose
National Fellowship of the Govt. of India, project number JBR/2020/000006.
The paper was finalized when the two of us, WTG and DP, were
at the Erwin Schr\"odinger Institute, Vienna in April 2022. We thank ESI for
its excellent program which brought us together there.}

\address{W.T.G.: National University of Singapore,
Singapore 119076.}
\email{matgwt@nus.edu.sg}
\address{B.H.G: Department of Mathematics, University of California San Diego, La Jolla, 92093}\email{gross@math.harvard.edu}
\address{D.P.: Indian Institute of Technology Bombay, Powai, Mumbai-400076} 

\email{prasad.dipendra@gmail.com}
\maketitle

{\hfill \today}
\begin{abstract}{In \cite{GGP1}, we considered a family of restriction problems for classical groups (over local and global fields) and proposed precise answers to these problems using the local and global Langlands correspondence.
These restriction problems were formulated in terms of a pair $W \subset V$
of orthogonal, Hermitian, symplectic, or skew-Hermitian spaces. 
In this paper, we consider a twisted variant of these conjectures in one particular case -- that of a pair of skew-Hermitian spaces $W = V$.}
  \end{abstract}

\tableofcontents

\section{\bf Introduction}
In \cite{GGP1}, we considered a family of restriction problems for classical groups (over local and global fields) and proposed precise answers to these problems using the local and global Langlands correspondence.
These restriction problems were formulated in terms of a pair $W \subset V$
of orthogonal, Hermitian, symplectic, or skew-Hermitian spaces. 
In this paper, we consider a twisted variant of these conjectures in one particular case -- that of a pair of skew-Hermitian spaces $W = V$.

\vskip 5pt

Let $F$ be a nonarchimedean local field and let $E$ be a separable quadratic algebra over $F$ with $\sigma \in \Gal(E/F)$,  the nontrivial element of the Galois group. Let $V$ be a non-degenerate skew-Hermitian space of dimension $n$ over $E$, with pairing $\langle v,w \rangle$. We may choose an orthogonal basis $\{v_1,v_2,\ldots, v_n\}$ of $V$ over $E$ and define the determinant
$$\det(V) = \prod_i \langle v_i,v_i\rangle,$$
in which, each term $\langle v_i,v_i \rangle$ lies in $E_0^{\times}$, where $E_0$ is the $F$-subspace of $E$
consisting of elements of trace $0$ and $E_0^{\times} = E_0 \smallsetminus \{ 0\}$. Since the product of two elements in $E_0^{\times}$ lies in $F^{\times}$, the determinant lies in $E_0^{\times}$ when $n$ is odd and in $F^{\times}$ when $n$ is even. Both $E_0^{\times}$ and  $F^{\times}$
are principal homogeneous spaces for the group $F^{\times}$,
and the orbit spaces $E_0^{\times}/NE^{\times}$ and $F^{\times}/NE^{\times}$
have cardinality $2$ if $E$ is a field, and have cardinality 1 otherwise. The determinant, as an element of one of  these orbit spaces
$E_0^{\times}/NE^{\times}$ or $F^{\times}/NE^{\times}$, 
is independent of the choice of an orthogonal basis, and gives a complete isomorphism invariant of the skew-Hermitian space $V$ over $E$.
\vskip 5pt

The isometry group $\U(V)$ has associated to it the  Weil representation $\omega_{V,\psi, \mu}$
(see \cite[Pg. 47-50]{GGP1}). If $E$ is a field, then
this complex representation of $\U(V)$ depends on a nontrivial additive character $\psi$ of $F$ and a conjugate-symplectic character $\mu$ of $E^{\times}$ (i.e.,  the  restriction of $\mu$ to $F^{\times}$ is the quadratic character $\omega_{E/F}$ associated to $E/F$ by the local class field theory, so that $\omega_{E/F}: F^{\times}/ N(E^{\times}) \cong \{ \pm 1 \}$). For an irreducible representation $\pi_1 \otimes \pi_2$ of $\U(V) \times \U(V)$ with a generic L-parameter, 
we had considered the problem of determining 
\[ \dim  \Hom_{\U(V)}(\pi_1 \otimes \pi_2,  \omega_{V, \psi, \mu}), \]
in \cite{GGP1}.
It is known by the work \cite{S} that this dimension is $0$ or $1$, and the conjecture in \cite{GGP1} (proved in \cite{GI})  determines precisely when this dimension is equal to $1$. 

If $E = F \times F$, $\U(V) \cong \GL_n(F)$, the  Weil representation $\omega_{V,\psi, \mu}$ could be taken to be ${\mathcal S}(F^n)$ with the natural action of $\GL_n(F)$ on it, 
and the resulting Hom space $\Hom_{\GL_n(F)}(\pi_1 \otimes \pi_2, {\mathcal S}(F^n) )$ is the one which intervenes in the local Rankin-Selberg integral for $\GL_n(F) \times \GL_n(F)$. 
\vskip 5pt

Here is the simplest twisted variant of the above question that we consider in this paper.
Instead of considering $\U(V)$ as a subgroup of $\U(V)(F \times F) = \U(V) \times \U(V)$, we consider it as a subgroup of $\U(V)(E) \cong \GL_n(E)$.  
For an irreducible generic representation $\Pi$ of $\GL_n(E)$,  we consider the problem of determining
\[  \dim \Hom_{\U(V)}(\Pi,  \omega_{V, \psi, \mu}). \]
We conjecture that this dimension is equal to $1$ for a unique (up to isomorphism) skew-Hermitian space $V$ of dimension $n$ over $E$, whose determinant is related to a local epsilon factor that we describe now.

Let $M$ be the Langlands parameter of $\Pi$, thus $M$ is an $n$-dimensional representation of the Weil-Deligne group $WD_E$ of $E$. Associated to $M$, let
${}^\sigma \!M ^\vee$ be the conjugate-dual representation of $WD_E$, so that $ M \otimes {}^\sigma \!M ^\vee$ is a conjugate-orthogonal representation of  $WD_E$ of dimension $n^2$. Since $\mu|_{F^\times} = \omega_{E/F}$,  $\mu$
is a conjugate-symplectic character of $E^{\times}$,
and hence $M \otimes {}^\sigma \!M ^\vee \otimes \mu^{-1}$ is a conjugate-symplectic representation of $WD_E$.
In this paper, we conjecture that the skew-Hermitian space $V$ for which $\Hom_{\U(V)}(\Pi,  \omega_{V, \psi, \mu}) \not = 0$
is determined by the identity
$$\mu(\det(V)) = \epsilon(1/2, M \otimes {}^\sigma \!M ^\vee \otimes \mu^{-1},
\psi_E) \cdot \det(M)(-1)^n \cdot \omega_{E/F}(-1)^{n(n-1)/2},$$
where $\psi_E$ is the additive character of $E$ obtained by composing $\psi$ with the trace
from $E$ to $F$. For the other skew-Hermitian
space
$V'$ of rank $n$ over $E$, we conjecture that:
\[  \Hom_{\U(V')}(\Pi,  \omega_{V', \psi, \mu}) = 0. \]
\vskip 5pt

We note that when $n$ is even so that $\det(V) \in F^\times$, $\mu(\det(V)) = \omega_{E/F} (\det(V))= \pm1$, $\mu(\det(V)) = +1$ if and only if the group $\U(V)$ is quasi-split. When $n$ is odd, the group $\U(V)$ is quasi-split for both of the skew-Hermitian spaces, and $\mu(\det(V))$ is a square root of $\omega_{E/F}(-1)$. Likewise, the local root number
$\epsilon(1/2, M \otimes {}^\sigma \!M ^\vee \otimes \mu^{-1}, \psi_E)$
is equal to $\pm 1$ when $n$ is even and is a square root of $\omega_{E/F}(-1)$ if $n$ is odd. 

\vskip 5pt

A related problem that has been studied in the literature is the determination of 
\[  \dim \Hom_{\U(V)}( \Pi, \C).  \]
The third author has proposed precise conjectures about this dimension \cite{P2}.  
Here, we have replaced the trivial representation of $\U(V)$ by a Weil representation, which lies in a one-parameter family (indexed by the characters of $E^1$) of the next smallest representations of $\U(V)$.  In retrospect, this appears quite natural and is simpler than this related problem considered in \cite{P2}. It is also simpler than our original conjecture in the skew-Hermitian case, where we considered $\U(V)$ as a subgroup of $\U(V)(F \times F) = \U(V) \times \U(V)$, whereas
$\U(V)(E)=\GL_n(E)$, a simpler group, in particular, the $L$-packets for $\GL_n(E)$ are singletons.
Note also that for $\Hom_{\U(V)}(\Pi,  \omega_{V, \psi, \mu})$,
we consider $\epsilon$ and $L$-function at $1/2$ of $ M \otimes {}^\sigma \!M ^\vee \otimes \mu^{-1}$
whereas for $\Hom_{\U(V)}( \Pi, \C)$,  one considers the pole at $s=1$ of $ M \otimes {}^\sigma \!M ^\vee$.  
\vskip 5pt
 
The astute reader can no doubt guess by now the general twisted variant of the GGP conjecture we have in mind. 
Beyond the case of $\U(V)$ as a subgroup of $\U(V)(E)$ and $\U(V)(F \times F)$,  we could choose a {\bf different} quadratic extension $K$ of $F$ and consider $\U(V)$ as a subgroup of $\U(V)(K)$, which is the isometry group of the skew-Hermitian space $V \otimes_E L$, with $L = E \otimes K$. Indeed, one could consider an arbitrary pair of \'etale quadratic $F$-algebras $(E,K)$ and formulate a corresponding branching problem. The various possibilities are given in the following table.
\vskip 10pt

 \begin{center}
\begin{tabular}{|c|c|c|c|}
\hline 
$E \backslash K$   &$F \times F$  &  $E$   &  field   \\
\hline 
& & &  \\
$F \times F$ &   Rankin-Selberg  &  Rankin-Selberg    &   Asai  \\
& & & \\
\hline 
& & & \\
field  & GGP &    $\U(V) \subset \GL(V)$  &  $\U(V) \subset \U(V_K)$  \\
& & & \\
\hline
\end{tabular}
\end{center}
\vskip 10pt

\begin{remark}
  We remark that in the case when $E=F \times F$, and $K$ is a separable quadratic extension of $F$
  (corresponding to the first row of the above table), we would be asserting that for any irreducible admissible generic representation $\pi$ of $\GL_n(K)$, and for $\omega$ the Weil representation of $\GL_n(F)$ realized on the Schwartz space ${\mathcal S}(F^n)$, we have,
  \[ \Hom_{\GL_n(F)}[ \pi \otimes \omega, \C] = \C. \]
  The assertion on dimension of the Hom space being $\leq 1$ is part of Theorem B of \cite{S}, and that it is nonzero is the conclusion of the Rankin-Selberg theory. 
  \end{remark}

  The last case in the table above, when $E \ne K$ are two distinct quadratic fields,  is the most complex and will be discussed in \S \ref{S:local-general}.
To provide some evidence for our conjectures, we will prove them  when $n= \dim V \leq 2$ (see \S \ref{S:lowrank} and \S \ref{S:evi}), as well as for unitary principal series representations for general $n$ (see \S \ref{S:UPS}, \S \ref{S:UPS2} and \S \ref{S:UPS3}, especially Corollary \ref{C:ps} and  Theorem \ref{T:ups2}). 
Indeed, when $E = K$, we reduce the conjecture for tempered representations to the case of essentially discrete series representations of $\GL(V)$ (in Corollary \ref{C:temp-to-ds}), and further to the case of supercuspidal representations under a certain hypothesis (in Theorem \ref{sc-to-ds}). In particular, this allows us to prove the conjecture for the Steinberg representation (in Corollary \ref{C:Steinberg}).  As a supplementary result, we show the vanishing of the corresponding higher Ext groups $\Ext^i$ ($i \geq 1$) for tempered representations (in Theorem \ref{Chen}). 
\vskip 5pt

We will also consider the twisted period problems over global fields. As in the GGP conjectures, one  expects that the nonvanishing of the global period integral here too is equivalent to the nonvanishing of a corresponding central L-value, in the absence of local obstructions. For example, when 
$E = K$, the relevant central L-value is $L(1/2, M \times {}^\sigma \!M ^\vee \times \mu^{-1})$.  One can also formulate a refined conjecture in the style of Ichino-Ikeda, which gives a precise formula relating the global period integral to the product of the above central L-value and certain canonical local period integrals. In the global context, it is interesting to note that when $E \ne K$, all possible local scenarios given in the above table will arise. Hence, one of our goals in this paper is to give a uniform formulation of the local conjectures which can be specialized to all the local scenarios in the table.
\vskip 5pt

With the twisted GGP problems and conjectures formulated, one can ask if all the previous work that has been done for the GGP conjectures can be adapted to this twisted setting. 
These include Waldspurger's and Beuzart-Plessis's integral formulae for the branching multiplicity and comparison of  Jacquet-Rallis relative trace formulae, which in the skew-Hermitian case is due to Y.F. Liu \cite{L} and H. Xue \cite{X1,X2}.  To this end, we remark that                    an integral formula for the branching multiplicity is being developed in
the
thesis work of Nhat Hoang Le (a student of the first author), whereas a                                  relative trace formula approach
is being pursued in the thesis work of Danielle Wang (a student of Wei Zhang at MIT).

The work which we needed to do in this paper with the Mackey theory allowed us to deal with certain non-tempered representations too leading us naturally to the non-tempered analogue of the GGP conjectures as in \cite{GGP3}  in  \S \ref{ggp3} in the twisted setting.
 In considering this twisted case, we realized that our original conjectures for non-tempered representations, where we introduced the concept of relevant parameters, needed to be clarified in some cases. This is also done in  \S \ref{ggp3}.

\vskip 10pt

 \section{\bf When $E = K$ is a field}  \label{S:KisE}
In this section, we consider the simpler case $E = K$, which was briefly discussed in the introduction. We shall formulate our conjectures more formally here, in both the local and global setting.
\vskip 5pt

\subsection{\bf Local case.}
We assume first that $F$ is a local field and $E/F$ is a separable quadratic field extension. We will let $E_0$ denote the $F$-subspace of trace $0$ elements in $E$ and let $E_1 \subset E^{\times}$ denote the subgroup of norm 1 elements. Fix a nontrivial additive character $\psi$ of $F$ and let $\sigma \in {\rm Gal}(E/F)$ be the nontrivial automorphism of $E/F$.
\vskip 5pt

For a skew-Hermitian space $V$ over $E$ of dimension $n$, we recall that 
\[  \det(V) \in  \begin{cases}
F^{\times}/ N_{E/F}(E^{\times}), \text{ if $n$ is even;} \\
E_0^{\times}/N_{E/F}(E^{\times}), \text{ if $n$ is odd.} \end{cases}\]
 If $F$ is nonarchimedean, there are precisely two skew-Hermitian spaces of dimension $n$, distinguished by their determinants.  When $F$ is archimedean, there are many more skew-Hermitian spaces, distinguished by their signatures.  
 \vskip 5pt
 
 Without loss of generality, we may assume that all these skew-Hermitian
 spaces (of a given dimension) have the same underlying vector
 space  $V$ over $E$, equipped with non-isomorphic skew-Hermitian
 forms. Thus the unitary groups
 $\U(V) \subset \GL(V)=\Aut_E(V) = \GL_n(E)$
 are  all subgroups of a fixed ambient group $\GL_n(E)$.
\vskip 5pt

For each skew-Hermitian space $V$ over $E$ and a conjugate-symplectic character $\mu$ of $E^{\times}$, we have the associated Weil representation $\omega_{V, \psi, \mu}$ of $\U(V)$. 
Now for an irreducible representation $\Pi$ of $\GL(V) \cong \GL_n(E)$, we consider the Hom space
\[  \Hom_{\U(V)}(\Pi, \omega_{V, \psi, \mu}). \]
Here is our main local conjecture in this case.
\vskip 5pt

\begin{conj}  \label{conj-local}
(i) For any $\Pi \in \Irr(\GL(V))$, 
\[  \dim  \Hom_{\U(V)}(\Pi, \omega_{V, \psi, \mu}) \leq 1.  \]
\vskip 5pt

(ii) If  $\Pi \in \Irr(\GL(V))$ is generic, then
\[   \sum_V  \dim \Hom_{\U(V)}(\Pi, \omega_{V,\psi,\mu})  = 1.  \]
where the sum is over the equivalence classes of skew-Hermitian structures on  $V$.
\vskip 5pt

(iii) For generic $\Pi \in \Irr(\GL(V))$, the unique skew-Hermitian space $V$ which gives a nonzero contribution to the above sum satisfies:
\[  \mu(\det(V)) = \epsilon( 1/2,  \Pi \times {}^\sigma\Pi ^\vee \times \mu^{-1}, \psi_E) \cdot \omega_{\Pi}(-1)^n \cdot \omega_{E/F}(-1)^{n(n-1)/2},\]
where ${}^\sigma \Pi^{\vee}$ is the conjugate-dual representation of $\Pi$ and $\omega_{\Pi}$ is the central character of $\Pi$.
\end{conj}
As noted in the introduction, the ratio of the two sides of (iii) is a priori $\pm 1$.
When $F$ is nonarchimedean,   the condition (iii) in the conjecture uniquely determines the summand with nonzero contribution to the sum in (ii).
When $F = \R$ and $E = \C$,   one needs to be more specific about the $V$ which gives nonzero contribution. We shall consider this archimedean case in greater detail in \S \ref{S:arch}.  Note that if we define the discriminant of $V$ by 
\[  {\rm disc}(V) = (-1)^{n(n-1)/2} \cdot \det(V), \]
then the formula in (iii) can be 
expressed more succinctly as
\[  \mu({\rm disc}(V))   =  \epsilon( 1/2,  \Pi \times {}^\sigma \Pi^\vee \times \mu^{-1}, \psi_E) \cdot \omega_{\Pi}(-1)^n,\] 
taking note of the fact that $\mu(-1) = \omega_{E/F}(-1)$.
We shall provide some evidence for this conjecture in the next two sections, verifying it for $\dim V \leq 2$ and for unitary principal series representations of $\GL(V)$ for $V$ of arbitrary dimension over $E$.

\vskip 5pt

In the above formulation, the conjecture does not require  the local Langlands correspondence, as the local root number in (iii) can be interpreted as the Rankin-Selberg local root number defined by Jacquet-Piatetski-Shapiro-Shalika \cite{JPSS}.

Let $M$ denote the Langlands parameter of $\Pi$, so that $M$ is an $n$-dimensional representation of the Weil-Deligne group $WD_E$ of $E$ with $\det(M)$ corresponding to the central character $\omega_{\Pi}$ under the local class field theory. We have noted in the introduction that $M \otimes {}^\sigma \!M ^\vee \otimes \mu^{-1}$ is a conjugate-symplectic representation of $WD_E$. 
Then Conjecture \ref{conj-local}(iii)
can be written as:
\[  \mu(\det(V)) = \epsilon( 1/2, M \otimes {}^\sigma \!M ^\vee \otimes \mu^{-1}, \psi_E) \cdot \det(M)(-1)^n \cdot \omega_{E/F}(-1)^{n(n-1)/2}. \]
Note that, for $e \in E_0^{\times}$,
\[  \det(M \otimes {}^\sigma \!M ^\vee)(e)  = \det(M)(e)^n / \det(M)(e^{\sigma})^n =  \det(M)(-1)^n,  \]
and
\[   \omega_{E/F}(-1) = \omega_{K/F}(e^2) = (e^2, e^2) \quad \text{(Hilbert symbol)}. \]
Hence the above identity can be expressed as (for $E=K$),
\[  \mu(\det(V)) = \epsilon( 1/2, M \otimes {}^\sigma \!M ^\vee \otimes \mu^{-1}, \psi_E) \cdot \det(M \otimes {}^\sigma \!M ^\vee)(e) \cdot \omega_{K/F}(e^2)^{n(n-1)/2}, \]
and it is this last statement that generalizes well when we deal with the general case
(where $E \ne K$) later.

\vskip 10pt

\subsection{\bf Global case}  \label{SS:global}
Consider now the case when $E/ F$ is a quadratic extension of global fields with adele rings $\A_E$ and $\A_F$.  Fix a nontrivial additive character $\psi$ of $F \backslash \A_F$.
We shall consider all skew-Hermitian structures on  a vector space $V$ of dimension $n$ over $E$.
\vskip 5pt

Let $\Pi \cong \otimes_v \Pi_v$ be a cuspidal automorphic representation of $\GL(V)(\A_F) = \GL(V\otimes_F\A_F) = \GL(V\otimes_E \A_E)$, thus $\Pi_v$
are, in particular,  generic representations for each place $v$ of $E$. For a conjugate-symplectic Hecke character $\mu$ of $\A_E^{\times}$, we may consider the automorphic Weil
representation $\omega_{V,\psi, \mu}$ of $\U(V)(\A_F)$ (see \cite{GGP1}). Now we consider the global period integral
\[  \mathcal{P}_V : \Pi\otimes \overline{\omega_{V,\psi, \mu}} \longrightarrow \C \]
defined by
\[   \mathcal{P}_V(f, \phi) = \int_{[\U(V)]}  f(g) \cdot \overline{\phi(g)} \, dg 
\quad \text{for $f \in \Pi$ and $\phi \in \omega_{V,\psi, \mu}$,} \]
where we have written $[\U(V)]$ for the adelic quotient $\U(V)(F) \backslash \U(V)(\A_F)$ with $dg$
the Tamagawa measure on it. 
\vskip 5pt

Globally, we are interested in characterizing the nonvanishing of this period integral. Our global conjecture is
the following.
\vskip 5pt

\begin{conj}  \label{conj-global}
  In the above setting, in particular for $V$ a
skew-Hermitian space over a global field $E$,
  the global period integral $\mathcal{P}_V$
  is nonzero if and only if the following two conditions hold (denoting $V_v=V \otimes F_v$):
\vskip 5pt

\begin{itemize}
\item[(a)]  For all places $v$ of $F$, $\Hom_{\U(V_v)}( \Pi_v, \omega_{V_v, \psi_v, \mu_v}) \ne 0$. 
\vskip 5pt

\item[(b)] $   L(1/2,   \Pi \times {}^\sigma\Pi^\vee \times \mu^{-1})  \ne 0.$
\end{itemize}
Further, for a cuspidal automorphic representation $\Pi$ of $\GL_n(\A_E)$,
if $   L(1/2,   \Pi \times {}^\sigma\Pi^\vee \times \mu^{-1})  \ne 0$, then there exists
a unique skew-Hermitian space  $V$ of dimension $n$ over $E$
such that the global period integral $\mathcal{P}_V$ is nonzero.
\end{conj}

\vskip 5pt

Observe that if we are given a collection of local skew-Hermitian spaces $\{V_v\}$ for all places  $v$ of $F$
(of a fixed dimension $n \geq 1$)
then 
the adelic skew-Hermitian space $\otimes_v V_v$ is coherent over $F$, i.e. the family of local
skew-Hermitian spaces $ V_v$ comes from a global  skew-Hermitian space $V$,  if and only if 
\[   \prod_v \mu_v\left ( \det(V_v) \right) = 1. \]

Therefore, given a cuspidal automorphic representation $\Pi$ of $\GL_n(\A_E)$,
if the local skew-Hermitian spaces $\{V_v\}$ are the ones for which
$\Hom_{\U(V_v)}( \Pi_v, \omega_{V_v, \psi_v, \mu_v}) \ne 0$ for all places $v$ of $F$, then
part (iii) of our local Conjecture \ref{conj-local} implies that this collection
of local skew-Hermitian spaces $\{V_v\}$ is coherent over $F$ if and only if 
\[  \epsilon(1/2,  \Pi \times {}^\sigma\Pi^\vee \times \mu^{-1}) = 1. \]

Therefore, given a cuspidal automorphic representation $\Pi$ of $\GL_n(\A_E)$ for which the
global period integral on $[U(V)]$ is nonzero (hence
$\Hom_{\U(V_v)}( \Pi_v, \omega_{V_v, \psi_v, \mu_v}) \ne 0$ for all places $v$ of $F$), then
$\epsilon(1/2,  \Pi \times {}^\sigma\Pi^\vee \times \mu^{-1}) = 1$.
Thus a necessary condition for the nonvanishing of  $L(1/2, \Pi \times {}^\sigma\Pi^\vee \times \mu^{-1})$
is satisfied if the global period integral on $[U(V)]$ is nonzero. Conversely,
given a cuspidal automorphic representation $\Pi$ of $\GL_n(\A_E)$ for which $L(1/2, \Pi \times {}^\sigma\Pi^\vee \times \mu^{-1}) \not = 0$,  hence $\epsilon(1/2,  \Pi \times {}^\sigma\Pi^\vee \times \mu^{-1}) = 1$, we have a global
skew-Hermitian space $V$, unique up to isomorphism,   for which Conjecture \ref{conj-global} implies
nonvanishing of period integral on $[\U(V)]$.

\vskip 10pt

\subsection{\bf A refined global conjecture}  \label{SS:global-refined}
Not surprisingly, one expects to be able to refine the above global conjecture to a precise formula relating the the global period integral to the central L-value.
 \vskip 5pt

 For  $\Pi \cong \otimes_v \Pi_v$,
 a cuspidal automorphic representation of $\GL(V)(\A_F) = \GL(V\otimes_F\A_F) = \GL(V\otimes_E \A_E)$, and 
 $\omega_{V, \psi, \mu} \cong \otimes_v  \omega_{V_v, \psi_v, \mu_v}$, the Weil representation of $\U(V)(\A_F)$, $f_v,f'_v \in \Pi_v$ and $\phi_v,\phi'_v \in
 \omega_{V_v, \psi_v, \mu_v}$,  
 we may consider the following integral of matrix coefficients for each place $v$ of $F$:
\begin{equation} \label{E:local-integral}
 \mathcal{I}_v ( f_v, f'_v, \phi_v, \phi'_v) := \int_{\U(V)(F_v)} \langle   g_v \cdot f_v, f'_v \rangle  \cdot \overline{ \langle g_v \cdot \phi_v, \phi'_v\rangle} \,  \, dg_v. \end{equation}
As in \cite{X2}, it is not hard to see that if $\Pi_v$ is tempered, this integral is absolutely convergent, so that it defines a $\U(V_v) \times \U(V_v)$-equivariant linear functional
\[  \mathcal{I}_v :  \Pi_v \otimes \overline{\Pi}_v \otimes \overline{\omega_{V_v, \psi_v, \mu_v}} \otimes \omega_{V_v, \psi_v,\mu_v} \longrightarrow \C. \]
Now one would like to:
\vskip 5pt

\begin{itemize}
\item show that $\mathcal{I}_v$ is nonzero if and only if $\Hom_{\U(V)(F_v)}( \Pi_v , \omega_{V_v, \psi_v,\mu_v}) \ne 0$;

\item compute this integral at almost all places $v$ of $F$  where every data involved is unramified. 
\end{itemize}
Without having done this work, we may nonetheless venture a guess here, in analogy with the original GGP case \cite{X2}. 
\vskip 5pt
\begin{conj}  \label{conj-integral}
Suppose that 
\vskip 5pt
\begin{itemize}
\item $E_v/F_v$ is an unramified quadratic extension of residue characteristic not $2$ and $\psi_v$ has conductor $\mathcal{O}_{F_v}$;
\item $\mu_v$ is unramified; 
\item $V_v$ contains a self-dual lattice $\Lambda_v$ whose stabilizer in $\U(V_v)$ is a hyperspecial  maximal compact subgroup $K_v$, contained in $\tilde{K}_v =  \GL(\Lambda_v) \subset \GL(V_v)$;
\item $dg_v$ is the Haar measure on $\U(V_v)$ which gives $K_v$ volume $1$;
\item $\Pi_v$ is $\tilde{K}_v$-unramified and $f_v = f'_v$ is a $\tilde{K}_v$-spherical  vector of norm 1;
\item $\phi_v = \phi'_v$ is a $K_v$-spherical vector of norm 1
  in the Weil representation $\omega_{V_v, \psi_v, \mu_v}$.
\end{itemize}
Then
\[   \mathcal{I}_v ( f_v, f'_v, \phi_v, \phi'_v)  = \frac{L(1, M_{\GL(V_v)}^{\vee})}{L(1, M^{\vee}_{\U(V_v)})} \cdot  \frac{ L(1/2, \Pi_v \times {}^\sigma\Pi^\vee_v \times \mu_v^{-1})}{L(1, \Pi_v, {\rm Ad})}, \]
where
\[  L(1, M_{\GL(V_v)}^{\vee}) = \prod_{k=1}^n \zeta_{E_v}(k) \quad \text{and} \quad   L(1, M^{\vee}_{\U(V_v)}) = \prod_{k=1}^n  L(k, \omega_{E_v/F_v}^k) \]
are the values at $s=1$ of the L-functions of the dual motives of $\GL(V)$ and $\U(V)$ respectively. (One may observe that the expression for $\mathcal{I}_v ( f_v, f'_v, \phi_v, \phi'_v)$ given above implies, in particular,
that it is nonzero.)
\end{conj}

Given this, it is natural to define a normalized local period integral:
\begin{equation} \label{E:local-integral2}
 \mathcal{I}_v^{\#} = \frac{L(1, M^{\vee}_{\U(V_v)})}{L(1, M_{\GL(V_v)}^{\vee})} \cdot  \frac{L(1, \Pi_v, {\rm Ad})}{L(1/2, \Pi_v \times {}^\sigma\!\Pi^\vee_v \times \mu_v^{-1})} \cdot \mathcal{I}_v. \end{equation}
We also note that if $E_v = F_v \times F_v$, the analog of the above conjecture holds, and has already been considered in the original
formulation of the refined GGP conjecture for skew-Hermitian spaces in \cite{GGP1}. 
\vskip 5pt

Coming back to the global setting, for each of the groups $\GL(V)$ or $\U(V)$, we will fix  a decomposition of the Tamagawa measures $dg = \prod_v dg_v$, so that for almost all $v$, the  local Haar measures $dg_v$ give a hyperspecial maximal compact subgroup volume $1$. We will also fix  a decomposition of the global Petersson inner product as a product of local pairings:
\begin{equation} \label{E:peter}  \langle- , - \rangle_{{\rm Pet}} = \prod_v \langle-, - \rangle_v, \end{equation}
and use these $dg_v$ and $\langle-,-\rangle_v$  in the definition of the local period integrals $\mathcal{I}_v$ introduced above.
We can now state:
\vskip 5pt

\begin{conj}  \label{conj-refined-global}

Given a (tempered) cuspidal automorphic representation $\Pi$ of $\GL(V)$, 
\[  \mathcal{P} \otimes \overline{\mathcal{P}}=  \frac{ L(1/2, \Pi \times {}^\sigma\Pi^\vee \times \mu^{-1})}{L(1, M_{\U(V)}^{\vee})} \cdot 
\left( \frac{L(s, M_{\GL(V)}^{\vee})}{L(s, \Pi, {\rm Ad})} \right)|_{s=1} \cdot \prod_v \mathcal{I}_v^{\#} .\]
as linear functionals on
$\Pi \otimes \overline{\Pi} \otimes \overline{\omega_{V,\psi, \mu}} \otimes  \omega_{V, \psi, \mu}$.  
 \end{conj}

Here, note that $L(s, M_{\GL(V)}^{\vee})$ and $L(s, \Pi, {\rm Ad})$ both have a simple pole at $s=1$, so that their ratio is holomorphic and nonzero at $s=1$.

\vskip 15pt
\subsection{\bf Finite fields}
We conclude this section by highlighting the restriction problem for skew-Hermitian spaces over a finite field $F = \mathbb{F}_q$. In the finite field setting, only the case $E = K$ can occur. In this setting, a naive first guess is that for any irreducible generic representation $\Pi$ of $\GL_n(\mathbb{F}_{q^2})$, 
\[  \dim \Hom_{\U_n(\mathbb{F}_q)} (\Pi, \omega)  =1  \]
where $\omega$ is the Weil representation of $\Sp_{2n}(\mathbb{F}_q)$, restricted to the subgroup $\U_n(\mathbb{F}_q)$. However, an examination of the case $n=1$  shows that  this cannot literally be the case because $\dim  \omega = q$ but $\U_1(\mathbb{F}_q)$ has $q+1$ characters. Indeed, the unique nontrivial quadratic character of $\U_1(\mathbb{F}_q)$ is missing from  $\omega$.  Moreover, experience with the usual GGP problem over finite fields shows that the above branching multiplicity can be larger than 1. 
Nonetheless, the naive expectation should be generically true  for cuspidal Deligne-Lusztig representations and it is an interesting question to quantify the extent of its failure. 
\vskip 5pt

Over finite fields, we can also consider this restriction problem for symplectic groups.  For any irreducible generic representation $\Pi$ of $\Sp_{2n}(\mathbb{F}_{q^2})$, one would thus like to determine
\[  \dim \Hom_{\Sp_{2n}(\mathbb{F}_q)}(\Pi, \omega) .  \]
  It is curious that since the two fold cover of $\Sp_{2n}(E)$ splits over $\Sp_{2n}(F)$, there is no
analogous problem for non-archimedean local fields. Perhaps, one could go to four fold cover
of $\Sp_{2n}(E)$ (if the 4th roots of unity are there in $E$) to study the analogous branching problem?
 \vskip 5pt
 
 A first study of these branching problems over finite fields has been conducted by Nhat Hoang Le. 
\vskip 15pt

\section{\bf Evidence in Low Rank}  \label{S:lowrank}

In this section, we will provide some evidence towards Conjecture \ref{conj-local}
when $n = \dim V \leq 2$.  \vskip 10pt

\subsection{\bf Rank one case}  \label{SS:rk1}
We begin by examining the case when $\dim V  =1$, so that 
$\GL(V) = E^{\times} \supset \U(V) = E_1$, where $E_1$ denotes the subgroup of norm one elements.  Given a character $\chi$ of $E^{\times}$, we are thus interested in understanding $\Hom_{E_1}( \chi, \omega_{V,\psi, \mu})$.  This is addressed by a theorem of
Moen \cite{Mo} and Rogawski \cite{R}:
\vskip 5pt

\begin{thm}  \label{T:moen}
If $\chi$ is a character of $E^{\times}$, then
\[ \dim \Hom_{E_1}( \chi, \omega_{V,\psi, \mu}) \leq 1 \]
and equality holds if and only if
\[  \mu(\det(V)) = \chi(-1) \cdot  \epsilon(1/2, \chi^{\sigma}/\chi \cdot \mu^{-1}, \psi_E). \]  
\end{thm}
This is precisely what Conjecture \ref{conj-local} asserts in the case $\dim V =1$.

\vskip 10pt

\subsection{\bf Rank two case}  \label{SS:rk2}
Suppose now that $\dim V = 2$.  Skew-Hermitian spaces of rank 2 can be described using quaternion $F$-algebras, as we have exploited in \cite{GGP2}. More precisely, 
for a quaternion $F$-algebra $B$, fix an $F$-algebra embedding $i: E \hookrightarrow B$ and write $B = E \oplus E \cdot x$ where $x$ is an element of $B$ such that $x e x^{-1} = e^{\sigma}$. Thus $B$ is a 2-dimensional $E$-vector space (by left multiplication), and we may identify $\GL_E(B)$ with $\GL_2(E)$ with respect to the basis $\{1, x\}$. 
\vskip 5pt

Now fix a trace 0 element $\delta \in E^{\times}$ and set
\[  \langle b_1, b_2 \rangle =  \delta \cdot (\text{projection of $b_1 \cdot \bar{b}_2$ onto $E$)}  \]
Then  $\langle-, -\rangle$ is a skew-Hermitian form on $B$; we shall denote this skew-Hermitian space by $V_B$. The isomorphism class of $V_B$ is independent of $x,\delta$, and $V_B$ is split if and only if $B$ is split. 
\vskip 5pt

The unitary similitude group $\GU(V_B) \subset \GL(V_B) = \GL_2(E)$  can be described by the isomorphism
\[ \iota:  (B^{\times} \times E^{\times}) / \Delta F^{\times} \stackrel{\cong} {\longrightarrow} \GU(V_B) \subset \GL(V_B)\]
 given by  sending $(b,e) \in B^{\times} \times E^{\times}$ to the element of $\GL(V_B)$ whose action on $B$ is:
  \[  (b,e):  y \mapsto e \cdot y  \cdot b^{-1}. \]
The similitude character is:
\[  {\rm sim}(b, e) = N_{E/F}(e) \cdot N_B(b)^{-1}. \]
Hence the unitary group is
\[  \U(V_B) \cong \{ (b, e) \in  (B^{\times} \times E^{\times}) / \Delta F^{\times} = \GU(V_B): N_{E/F}(e) = N_B(b) \} \]
This is contained in the subgroup
\[   \GU(V_B)^+ \cong  \{ (b,e) \in  (B^{\times} \times E^{\times}) / \Delta F^{\times} =\GU(V_B): N_B(b) \in N_{E/F}(E^{\times}) \} \]
which has index 2 in $\GU(V_B)$. Moreover, if $Z = E^{\times}$ denotes the center of $\GL(V_B)$, then 
\[  \GU(V_B)^+ = Z \cdot \U(V_B). \]
Thus, when working with irreducible representations of $\U(V_B)$, there is no difference in working with $\GU(V_B)^+$ instead. 
\vskip 5pt

Let us explicate the Weil representation of $\U(V_B)$ in this framework. The Weil representation $\omega_{\psi,\mu, B}$ is reducible but admits a central character decomposition:
\[  \omega_{\psi, \mu, B} = \bigoplus_{\lambda}  \omega_{\psi,\mu, B}[\lambda] \]
where the sum runs over the characters of $Z(\U(V_B))= E_1$ and  each summand is irreducible or $0$. 
We can describe $\omega_{\psi,\mu, B}[\lambda]$ in terms of the description of $\U(V_B)$ given above. More precisely, suppose that $\lambda = \chi|_{E_1}$ for a  character $\chi$ of $E^{\times}$. Consider the L-parameter 
\[  N = \Ind_{W_E}^{W_F} (\mu \cdot \chi^{-1})  \quad \text{ of $\GL_2(F)$,} \]
and let $\Sigma_{B,N}$ be the associated representation of $B^{\times}$.  This gives a representation 
\[  \Sigma_{B,N} \boxtimes \chi  \quad \text{ of $B^{\times}\times E^{\times}$}  \]
which is trivial on $\Delta F^{\times}$, i.e. a representation of $\GU(V_B)$. This representation of $\GU(V_B)$ decomposes into the sum of two irreducible summands when restricted to $\GU(V_B)^+$. One of these summands is the representation $\omega_{\psi,\mu, B} [\chi|_{E_1}]$ whereas the other is $\omega_{\psi',\mu, B} [\chi|_{E_1}]$, with $\psi'$ in a different $N(E^{\times})$-orbit as $\psi$. 
\vskip 5pt

Now suppose that $\Pi$ is an irreducible generic representation of $\U(V_B \otimes_F E) \cong  \GL(V_B)$ with L-parameter $M$.
The embedding $\U(V_B) \hookrightarrow \U(V_B \otimes_F E)$ is the natural embedding $\U(V_B) \subset \GL(V_B)$.  Pulling $\Pi$ back via $\iota$, and with $\chi := \omega_{\Pi}$, we see that
\begin{align}  
  \Hom_{\U(V_B)}(\Pi, \omega_{\psi,\mu, B})
  &= \Hom_{\U(V_B)}(\Pi, \omega_{\psi,\mu, B}[\chi|_{E_1}])  \notag \\
 &\cong \Hom_{(B^{\times})^+}( \iota^*(\Pi), \omega_{\psi,\mu, B}[\chi|_{E_1}])  \notag \\
 &\cong \Hom_{B^{\times}}( \iota^*(\Pi), \Sigma_{B,N}).  \notag 
 \end{align}
 \vskip 5pt
 
 Now it is important to note that the embedding $\iota: B \hookrightarrow \GL(V_B) = \GL_2(E)$ is not the natural embedding $B^{\times} \hookrightarrow  (B \otimes_F E)^{\times} \cong \GL_2(E)$, but rather differs from it by the outer automorphism $b \mapsto \bar{b}^{-1}$. 
 Indeed, $\iota$ is the inverse on the central $F^{\times}$.
  Taking this into account, we see that the last Hom space above is the space   
  \[     \Hom_{B^{\times}}(\Pi^{\vee} \otimes \Sigma_{B,N}^{\vee}, \C)  \]
 of twisted trilinear forms, where $B^{\times} \hookrightarrow (B \otimes_F E)^{\times}  \cong \GL_2(E)$, with the last isomorphism induced by an $E$-algebra isomorphism $B \otimes_F E  \cong \M_2(E)$. 
 
\vskip 5pt

By a result of the third author \cite{P1}, one has 
 \[   \dim  \Hom_{B^{\times}}(\Pi^{\vee} \otimes \Sigma_{B,N}^{\vee}, \C)   \leq 1  \]
 with equality if and only if
\[     \epsilon(1/2, {\rm As}^+(M^{\vee}) \otimes N^{\vee}, \psi_E) \cdot \omega_{E/F}(-1) = \mu(\det(V_B)), \]
 where $As^+$ is the Asai lift of $M$ from $E$ to $F$. We refer the reader to \S \ref{SS:Asai} below for the definition and properties of $As^+$. 
 Now let us explicate the local root number:
 \begin{align}
  \epsilon(1/2, {\rm As}^+(M^{\vee}) \otimes N^{\vee}, \psi_E) &= \epsilon(1/2, {\rm As}^+(M^{\vee}) \otimes \Ind_E^F (\mu^{-1} \cdot \chi),\psi) \notag \\
  &= \epsilon(1/2, \Ind_E^F ( M^{\vee} \otimes {}^\sigma \!M ^\vee \otimes \mu^{-1} \otimes \chi), \psi) \notag \\
  &=  \epsilon(1/2, \Ind_E^F ( M \otimes {}^\sigma \!M ^\vee \otimes \mu^{-1}),\psi) \notag \\  
& = \epsilon(1/2, M \otimes {}^\sigma \!M ^\vee \otimes \mu^{-1}, \psi_E), \notag
 \end{align}
 where in the second last equality, we have noted that $\chi = \omega_{\Pi} = \det M$, so that $M^{\vee} \otimes \chi \cong M$ (since $\dim M = 2$), and in the last equality, we have
used the fact that epsilon factors are inductive in dimension zero together with the fact that  
$\dim(M \otimes {}^\sigma \!M ^\vee) =4$.

  To conclude, we have shown:

\begin{prop}
  For $\Pi$, an irreducible generic representation of $\U(V_B \otimes_F E) \cong  \GL(V_B)$,
  with L-parameter $M$, a two dimensional representation of $WD_E$,
  \[\Hom_{\U(V_B)}(\Pi, \omega_{\psi,\mu, B}) \ne 0 \iff \epsilon(1/2, M \otimes {}^\sigma \!M ^\vee \otimes \mu^{-1}, \psi_E) \cdot \omega_{E/F}(-1) = \mu(\det(V_B)). \] 
  \end{prop}

 This is precisely what Conjecture \ref{conj-local} says in the case $n =2$.
 \vskip 10pt

 \subsection{\bf Global conjecture: rank 1 case}  \label{SS:global-rank1}
 Finally, we can also verify the global Conjecture  \ref{conj-refined-global} when $\dim_E V = 1$. Let $\chi$ be a Hecke character of $\GL(V)(\A_E) = \A_E^{\times}$, so that we are considering the global period integral
 \[  \mathcal{P}: \C  \chi \otimes  \overline{\omega_{V,\psi,\mu}} \longrightarrow \C \]
 defined by
 \[  \mathcal{P}(\phi) = \int_{[E_1]} \chi(x) \cdot \overline{\phi(x)} \, dx. \]
 Observe that this is simply the (conjugate of the) global theta lifting of $\chi$  for the dual pair 
 \[  \U_1 \times \U_1 = \U(V) \times \U(W),\]
  evaluated at the identity element. Here, $V$ is equipped with its given skew Hermitian structure and $W$ is the rank 1 Hermitian space $\langle 1 \rangle$.
 The nonvanishing of $\mathcal{P}$ is thus equivalent to the nonvanishing of the global theta lift $\Theta_{V,W,\psi,\mu}(\chi)$ of $\chi$.  Moreover, when this global theta lift is nonzero, it is isomorphic to the representation $\chi$ of $\U(W) = E_1$.  Then we have:
 \[   
 \mathcal{P}(\phi_1) \cdot \overline{\mathcal{P}(\phi_2)}  \cdot \mu([E_1]) =  \langle \Theta(\phi_2,\chi), \Theta(\phi_1,\chi) \rangle_{\rm Pet}  \]
 where $\mu([E_1]) =2$ is the Tamagawa measure of $\U(W)$.  Now the Petersson inner product of the global theta lift on the RHS is computed by the Rallis inner product formula. This was first done by Tonghai Yang \cite{Y} and a convenient reference is \cite[Thm. A.4.2]{X2}. One has
 \[   \langle \Theta(\phi_2,\chi), \Theta(\phi_1,\chi) \rangle_{\rm Pet}  = \frac{L(1/2, \chi^{\sigma} \chi^{-1} \cdot\mu^{-1})}{L(1, \omega_{E/F})} \cdot Z^*(\phi_2, \phi_1)  \]
 where 
 \[  Z^*(\phi_2,\phi_1)=   \int^*_{\A_E^1}  \overline{\langle g \phi_1,\phi_2\rangle}\cdot \langle g \chi, \chi \rangle_{\U(V), {\rm Pet}} \, dg  \]
 is the normalized global doubling zeta integral.  Since the Tamagawa measure of $\U(V)$ is $2$, one has
 \[  \langle g \cdot \chi, \chi \rangle_{\U(V), {\rm Pet}}  =  2 \cdot \chi(g)  \]
   so that
\[ Z^*(\phi_1,\phi_2) = 2 \cdot  \int^*_{\A_E^1}  \overline{\langle g \phi_1,\phi_2\rangle}\cdot \chi(g) \, dg  = \prod_v \mathcal{I}_v^{\#}(\chi,\chi, \phi_1,\phi_2), \]
where the local factors $\mathcal{I}_v^{\#}$ are as defined in (\ref{E:local-integral}) and (\ref{E:local-integral2}).
Hence we conclude that
\begin{equation} \label{E:refinedfirst}
   \mathcal{P}(\phi_1)  \cdot \overline{\mathcal{P}(\phi_2)} = \frac{L(1/2, \chi^{\sigma} \chi^{-1} \cdot\mu^{-1})}{L(1, \omega_{E/F})} \cdot \prod_v \mathcal{I}_v^{\#}(\chi,\chi, \phi_1,\phi_2). \end{equation}
This is precisely what Conjecture  \ref{conj-refined-global} says. 
\vskip 5pt

For the case when $\dim_E V = 2$, Conjecture  \ref{conj-refined-global} should reduce to Ichino's formula \cite{I} relating the (twisted) triple product period integral and the  (twisted) triple product L-value. We leave the verification of this to the interested reader. 

\vskip 15pt

\section{\bf Mackey Theory: Restriction of Principal Series} \label{S:UPS}
In this section, we apply Mackey theory to understand the branching of a principal series representation of $\GL(V)$ to the Weil representation of $\U(V)$ over a non-archimedean local field. 
\vskip 5pt

\subsection{\bf Principal series}
Let $V$ be a vector space of dimension $n$ over $E$.
 For a partition  $n = a+b,$ with $0< a \leq b \in \Z$,  let 
 \[  V = V_a  \oplus V_b \]
 with $\dim V_a = a$ and $\dim V_b = b$.  Consider the maximal parabolic subgroup
 \[  P = P_{a,b} = M \cdot N \] 
 of $\GL(V)$ stabilizing $V_a$, with Levi factor  
 \[  M= \GL(V_a) \times \GL(V_b). \]
 Let $\pi = \pi_1 \boxtimes \pi_2$ be a representation of $\GL(V_a)\times \GL(V_b)$
  and consider the (normalized) parabolically induced representation
\begin{equation} \label{E:pi}
 \pi = \pi_1 \times \pi_2 = \Ind_P^{\GL(V)} (\pi_1 \boxtimes \pi_2). \end{equation}
 These are the principal series representations we will consider.
 \vskip 5pt
 
 \subsection{\bf Skew-Hermitian structures}
 Recall that there are two inequivalent skew-Hermitian structures on $V$, distinguished by their determinants in $F^{\times}/NE^{\times}$ or $E_0^{\times}/ NE^{\times}$ (depending on whether $n = \dim V$ is even or odd). 
 For such a class $\delta$, we let $V_{\delta}$ denote the skew-Hermitian structure on $V$ with determinant $\delta$, so that
 $\U(V_{\delta}) \subset \GL(V)$. We will often drop $\delta$ from $V_\delta$
 when a particular skew-Hermitian structure is fixed on $V$. We also let $\rk(V)$ denote the dimension of a maximal isotropic subspace of the skew-Hermitian space $V$; this is sometimes called the Witt index of $V$.
 \vskip 5pt
 
 On the other hand, $V$ with a roman subscript, such as $V_a$, will
 denote either just a vector space over $E$ or a skew-Hermitian space over $E$ of dimension $a$.

 \vskip 5pt
 
 \subsection{\bf Mackey theory}
For a fixed skew-Hermitian space $V = V_{\delta}$,  the goal of this section is to compute
  \[  \Hom_{\U(V)}( \pi_1 \times \pi_2,  \omega_{V, \psi,\mu}), \]
  where  $\omega_{V, \psi , \mu}$ denotes the Weil representation of $\U(V)$ associated to $(\psi,\mu)$.  In fact, we will consider the
 more general 
 \[   \Ext^i_{\U(V)}( \pi_1 \times \pi_2,  \omega_{V, \psi,\mu}). \]
  This will be achieved  by using Mackey theory, which requires the determination   of the orbits of $\U(V)$ on the partial flag variety $\GL(V)/P$. In this analysis, each orbit gives rise to a certain induced representation of $\U(V)$ arising from the restriction of the inducing data to the
stabilizer of a point in the orbit. 
\vskip 5pt

Thus the representation $\pi = \pi_1 \times \pi_2$ when restricted to $\U(V)$ comes equipped with a certain
finite filtration by $\U(V)$-modules in which the open orbits contribute as submodules, and the non-open orbits contribute as
subquotients.

\subsection{\bf Orbits}
The following lemma, whose proof will be omitted,  summarizes the orbit structure of $\U(V)$ on  $\GL(V) / P$ and is a direct consequence of Witt's theorem.
\vskip 5pt

\begin{lemma} \label{orbits}
The orbits of  $\U(V)$ on $\GL(V) / P_{a,b}$ (with $0<a \leq b$)  are
represented by the isometry classes of  $a$-dimensional $E$-subspaces $X \subset V$, which are themselves
parameterized  by the following two invariants:
\begin{enumerate}
\item the dimension $d$ of the kernel     of the skew-Hermitian form on $V$ restricted to $X$,  i.e. $ d=\dim (X \cap X^\perp)$,
  $d \leq \min\{a, \rk(V)\}$, and
\vskip 5pt

  \item the nondegenerate  skew-Hermitian form
    on $X/(X \cap X^\perp)$, which can be arbitrary.
\end{enumerate}   

\vskip 5pt

In particular, with $E$ non-archimedean, one has:
\vskip 5pt
\begin{itemize}
\item For each integer 
\[  0 \leq d  \leq  \min\{a, \rk(V)\},  \] 
there are two orbits $[X]$ of $\U(V)$ on $\GL(V) / P_{a,b}$, with $\dim (X \cap X^{\perp}) = d$, unless $d=a$ (i.e. when
$X/(X \cap X^\perp) = 0$),  in which case there is only one.\vskip 5pt

\item The open orbits  correspond to $d=0$, i.e., 
the isomorphism classes of the two nondegenerate skew-Hermitian subspaces of $V$ of dimension $a$.
\vskip 5pt

\item There is a unique closed orbit  which corresponds to $ d =  \min\{a, \rk(V)\} =a$,
except when $n=2a$ and $V$
does not have an isotropic subspace of dimension $a$,  in which case there are two  closed orbits  corresponding to $d=a-1$.
\end{itemize}
  \end{lemma}
\vskip 5pt

\subsection{\bf Stabilizers}  \label{SS:stabilizers}
 Let $[X]$ be an $\U(V)$-orbit in $\GL(V)/P_{a,b}$, represented by an $E$-subspace  
 $X \subset V$  of dimension $a$ with  $\dim (X \cap X^\perp) =d$.  Let us first determine the stabilizer  $S=S_X$ of $X$ in $\U(V)$.  
 \vskip 5pt
 
 Observe that $S_X$ preserves the flag
\[ 0 \subset  X \cap X^\perp \subset  X           \subset   ( X \cap X^\perp)^\perp = X+X^\perp   \subset  V,\]
and note that
\[ X/( X \cap X^\perp ) \subset  (X+X^\perp) / ( X \cap X^\perp) =: V_{n-2d}\]
are non-degenerate skew-Hermitian spaces of dimension $a-d$ and $n-2d$ respectively.
 Hence, $S_X$ is contained in the maximal parabolic subgroup $Q_d$ of $\U(V)$ stabilizing the isotropic space $X \cap X^{\perp}$. The parabolic subgroup $Q_d =M_d \cdot N_d$  can be depicted in matrix form as:
 $$ Q_d = \left ( \begin{array}{ccc} 
  \GL(X \cap X^{\perp})  & *_1  &  *_d     \\
0 & \U(V_{n-2d}) &  *_2    \\
0  & 0 & *
\end{array}
\right ) ,$$
with Levi factor
\[  M_d = \GL(X \cap X^{\perp})  \times \U(V_{n-2d}) \]
and unipotent radical $N_d$.  The center of $N_d$ is the subgroup
$Z_d$  consisting of matrices  with $*_1=*_2 =0$.

\vskip 5pt

It follows that, as a subgroup of $Q_d$,  $S_X$ has the form:
$$S_X = \left ( \begin{array}{cccc} 
  g  & *_{12}   & *_{13} & *_4     \\
0 & \U_{a-d} & 0  & *_{24}    \\
 0 & 0  & \U_{b-d} & *_{34}  \\
 0 & 0  & 0 & (g^*)^{-1} \\
\end{array}
\right ),$$
where 
\begin{itemize}
\item $g \in \GL(X \cap X^{\perp}) \cong \GL_d(E)$,
\item the entries  $*_{12}$ and $ *_{34}$ are arbitrary matrices with entries in $E$ of appropriate sizes which determines $*_{24}, *_{13}$; 
\item the entry  $*_4$ is an arbitrary skew-Hermitian matrix of size $d \times d$.  
\end{itemize}
Let us highlight certain natural subgroups or quotients of $S_X$:
\vskip 5pt 
\begin{itemize}
\item The unipotent radical $N(S_X)$ of $S_X$ consists of those matrices which have the identity matrix on each diagonal block. Observe that $N(S_X)$ is in fact the unipotent radical $N_d$ of the maximal parabolic subgroup  $Q_d$.  
\vskip 5pt

\item The center $Z(S_X)$ of $N(S_X)$ is the subgroup consisting of elements whose only nonzero entry in the upper triangular blocks is $*_4$, so that $Z(S_X) =Z_d$.
\vskip 5pt

\item The Levi factor $S_X/ N(S_X)$ is isomorphic to  
\[  \GL(X \cap X^{\perp}) \times   \U_{a-d} \times \U_{b-d}. \]
\end{itemize}
\vskip 5pt

\subsection{\bf Modules}  
In what follows, we use $\Ind$
for the usual normalized induction, and $\ind$ for the usual normalized induction with compact support, whereas we will use
$\II nd$ and $\ii nd$ for the corresponding un-normalized induction. Thus, for example,
\[ \pi = \pi_1 \times \pi_2 = \Ind_P^{\GL(V)} (\pi_1 \otimes \pi_2) = \II nd_P^{\GL(V)}   (\pi_1 \otimes \pi_2 \otimes \delta_P^{1/2}).\]

\vskip 5pt
By Mackey theory, the restriction of the principal series representation $\pi = \pi_1 \times \pi_2$ to $\U(V)$ has a finite equivariant filtration indexed by the the $\U(V)$-orbits given in Lemma \ref{orbits}. For each such $\U(V)$-orbit $[X]$, let $\pi_X$ denote the associated $\U(V)$-subquotient of $\pi$.  The following proposition determines  the representation $\pi_X$.  
\vskip 5pt

\begin{prop} \label{P:module}
For a $\U(V)$-orbit $[X]$ on $\GL(V)/P$, with $\dim X \cap X^{\perp} = d$ and stabilizer $S= S_X$, one has:
\[  \pi_X \cong    \ii nd_{S}^{\U(V)} (\pi_1 \otimes \pi_2 \otimes \delta_P^{1/2})|_{S} =  \ind_{S}^{\U(V)} (\pi_1 \otimes \pi_2 \otimes \delta_{P/S}^{1/2})|_{S},\]
where we have written $\delta_{P/S} = \delta_P \delta_S^{-1}$.
\end{prop}

\vskip 5pt
 
We note that the representation $\pi_1 \otimes \pi_2 \otimes \delta_{P/S}^{1/2}$ is non-trivial on the unipotent radical $N(S)$ of $S$, but it
   is trivial on the center  $Z(S)$ of $N(S)$.    \vskip 5pt

\subsection{\bf Branching for $\pi_X$}
We are now ready to consider the branching problem
\[  \Ext^i_{\U(V)}( \pi_1 \times \pi_2, \omega_{V, \psi, \mu}). \]
Since, as a $\U(V)$-module, $\pi_1 \times \pi_2$ has a finite filtration with subquotients $\pi_X$ as given in Proposition \ref{P:module}, it is natural to first consider
\[   \Ext^i_{\U(V)} (\pi_X, \omega_{V, \psi, \mu}). \]
The result of this key computation is given by the following proposition.
\vskip 5pt

 \begin{prop} \label{ps2}
   For an    orbit $[X]$ of $\U(V)$ on $\GL(V)/P$, with $\dim (X \cap X^{\perp}) =d$, corresponding stabilizer $S = S_X$ and associated $\U(V)$-module $\pi_X$, one has:
\[ \Ext^i_{\U(V)} \left( \pi_X, \omega_{V,\psi,\mu}\right)  \]
\[ {\cong}  \, \,\,
   \Ext^i_{S/N(S)}\left( \delta_{P/S}^{1/2} \cdot (\pi_1)_{d,a-d} \otimes (\pi_2)_{b-d,d},  \, \, \delta_{S}^{1/2} \cdot |{\det}_{\GL_d}|^{-1/2}\mu \cdot \omega_{V_{n-2d},\psi,\mu} \right)  \]
where we note:
\vskip 5pt
\begin{itemize}
\item $S/ N(S) \cong \GL_d(E) \times \U_{a-d} \times \U_{b-d}$;
\vskip 5pt

\item $(\pi_1)_{d,a-d}$ denotes the un-normalized Jacquet module of $\pi_1$ with respect to the
$(d,a-d)$ parabolic  subgroup in $\GL(V_a) \cong \GL_a(E)$, regarded as a representation of $\GL_d(E) \times \U_{a-d} \subset \GL_d(E) \times \GL_{a-d}(E)$ by restriction;
\vskip 5pt

\item likewise, $(\pi_2)_{b-d,d}$ is the un-normalized Jacquet module of $\pi_2$ with respect to the $(b-d,d)$-parabolic subgroup in $\GL(V_b) \cong \GL_b(E)$, regarded as a representation of $\U_{b-d} \times \GL_d(E) \subset \GL_{b-d}(E) \times \GL_d(E)$ by restriction and taking contragredient on the $\GL_d(E)$ factor.
\vskip 5pt

\item The characters $\delta_{P/S}$ and $\delta_S$ are trivial on $\U_{a-d} \times \U_{b-d}$ and are given  on $\GL_d(E)$ by:
\[  \delta_{P/S} = |\det|^d \quad \text{and} \quad  \delta_S(g)= |\det|^{n-d}. \] 
\end{itemize}
\vskip 5pt

In particular, for the two open orbits $X$ corresponding to $d=0$,
we have, 
\[ \Ext^i_{\U(V)}[  \pi_X, \omega_{V,\psi,\mu}] {\cong}  \sum_{\begin{array}{c} i=j+k\\ V= V_a\oplus V_b \end{array}} \Ext^j_{\U(V_a)} 
   [\pi_1|_{\U(V_a)}, \omega_{V_a, \psi,\mu}]  \otimes \Ext^k_{\U(V_b)}[ \pi_2|_{\U(V_b)}, \omega_{V_b,\psi,\mu}] ,\]
   where  $X = V_a$ are the isomorphism classes of non-degenerate subspaces of $V$ of dimension $a$ with orthogonal
  complement $X^{\perp} = V_b$.
\end{prop}

\begin{proof}
  For analyzing  $\Ext^i_{\U(V)}[ \pi_X , \omega_{V,\psi, \mu}]$, we will need the following generalities on Ext groups and contragredients
   (cf. \cite{P3} for (a) below):

\vspace{2mm}
(a) For any two smooth representations $U,V$ of a $p$-adic group $G$, we have:
\[ \Ext_G^i[ U,V^\vee] \cong \Ext_G^i[V, U^\vee] .\]

\vspace{2mm}

(b) For $H$ a closed subgroup of a $p$-adic group $G$, and $U$ any smooth representation of $H$ with smooth dual $U^\vee$,
\[ [\ind_H^G U]^\vee \cong \Ind_H^G U^\vee. \]

\vspace{2mm}

(c) For  a nontrivial character $\psi: F \rightarrow \C^\times$, with the associated Weil representation   $\omega_{V,\psi, \mu}$  of $\U(V)$, we have
\[\omega_{V,  \psi, \mu}^\vee  \cong  \omega_{V, \psi^-,\mu^{-1}},\]
where $\psi^-(x) = \psi(-x)$. 


We shall also use the following three lemmas whose proofs are left to the reader.
\vskip 5pt

 \vskip 5pt

\begin{lemma} \label{ext}
  Let $U \subset H$ be $p$-adic groups with $U$ a normal subgroup of $H$ which is a union of compact open subgroups. Let $\pi_1,\pi_2$ be two smooth representations of $H$ such that $U$ acts trivially on $\pi_2$. Then
  \[ \Ext^i_H[\pi_1,\pi_2] \cong \Ext^i_{H/U}[ \pi_{1, U}, \pi_2].\]
\end{lemma}

 \begin{lemma} \label{central} Let $G$ be any $p$-adic group, $Z\subset G$ a closed central subgroup. If $\pi_1$ and $\pi_2$ are two
    smooth representations of $G$  on which $Z$ operates by different characters, then
    \[ \Ext_G^i[\pi_1,\pi_2] = 0 {\rm ~~for ~~all~~} i \geq 0.\] 
  \end{lemma}

\vskip 5pt

\begin{lemma} \label{Weil-Jac}
  Let $V=V_n$ be a skew-Hermitian space over $E$ of dimension $n$, and let
  $\omega_{V,\psi,\mu}$ be a Weil representation of $\U(V)$. Let $Q_d = M_d N_d$ be
  the maximal parabolic subgroup of $\U(V)$ stabilizing a $d$-dimensional isotropic space
  (see \S \ref{SS:stabilizers}), so that $M_d \cong \GL_d(E) \times \U(V_{n-2d})$. Then for $Z_d$, the center of $N_d$, we have:
 \[  (\omega_{V,\psi,\mu})_{Z_d}  = (\omega_{V,\psi, \mu})_{N_d} \cong  (\mu  \cdot |-|^{1/2} \circ \det) \otimes\omega_{{V_{n-2d}},\psi,\mu} , \]
 as $M_d$-modules.
  \end{lemma}

 \vspace{2mm}

With these preliminaries in place,  we now compute:
\begin{alignat}{2}
    &\hskip 10pt \Ext^i_{\U(V)}[  \pi_X, \omega_{V,\psi, \mu}] &  \notag \\
  \cong  & \hskip 10pt \Ext^i_{\U(V)}[  \ind_{S}^{\U(V)}  (\pi_1 \otimes \pi_2 \otimes \delta_{P/S}^{1/2}), \omega_{V,\psi, \mu}], &\text{  (by Prop. \ref{P:module})} \notag \\
   \cong  &\hskip 10pt \Ext^i_{\U(V)}[ \omega_{V,\psi^-, \mu^{-1}}, \Ind_{S}^{\U(V)}  (\pi_1 \otimes \pi_2 \otimes \delta_{P/S}^{1/2})^\vee ],  & \text{ (by (a),(b) and (c))} \notag \\
   \cong  &\hskip 10pt \Ext^i_{S}[ \delta_{S}^{-1/2}\omega_{V,\psi^-, \mu^{-1}},
    (\pi_1 \otimes \pi_2 \otimes \delta_{P/S}^{1/2})^\vee ],  & \text{  (by Frobenius reciprocity)} \notag \\
 \cong  &\hskip 10pt \Ext^i_{S/Z(S)}[ \delta_{S}^{-1/2}(\omega_{V,\psi^-, \mu^{-1}})_{Z(S)},
    (\pi_1 \otimes \pi_2 \otimes \delta_{P/S}^{1/2})^\vee ], & \text{   (by  Lemma \ref{ext})} \notag \\
   \cong  &\hskip 10pt   \Ext^i_{S/Z(S)}[ \delta_{S}^{-1/2} \cdot \mu^{-1} \cdot  |\det|^{1/2} \cdot \omega_{{V_{n-2d}},\psi^-, \mu^{-1}},
    (\pi_1 \otimes \pi_2 \otimes \delta_{P/S}^{1/2})^\vee ]    & \text{   (by Lemma \ref{Weil-Jac})} \notag \\
     \cong &\hskip 10pt    \Ext^i_{S/Z(S)}[ \pi_1 \otimes \pi_2 \otimes \delta_{P/S}^{1/2}, \,\, \delta_{S}^{1/2} |\det|^{-1/2}\cdot \mu \cdot \omega_{V_{n-2d},\psi,\mu}]   
      & \text{ \hskip 5pt  (by (a))} \notag \\
 \cong &\hskip 10pt  \Ext^i_{S/N(S)}[ (\pi_1)_{d,a-d} \otimes (\pi_2)_{b-d,d} \otimes \delta_{P/S}^{1/2},\,\,  \delta_{S}^{1/2} |\det|^{-1/2} \cdot \mu \cdot \omega_{V_{n-2d},\psi,\mu}]
 &   \notag 
\end{alignat}
\vskip 10pt

\noindent where $ (\pi_1)_{d,a-d}$ denotes the un-normalized Jacquet module of $\pi_1$ with respect to the
$(d,a-d)$ parabolic  subgroup in $\GL_a(E)$; similarly for $(\pi_2)_{b-d,d}$.
Here we have applied Lemma \ref{ext} (taking $U = N(S)/Z(S)$)  for the last isomorphism for which 
it is important to  note that
$N(S)/Z(S)$  maps isomorphically to the product of the unipotent radicals of the $(d,a-d)$-parabolic subgroup  of $\GL_a(E)$ and the $(b-d,d)$-parabolic subgroup of $\GL_b(E)$.
 


 \vskip 10pt
 
For the final assertion in the proposition regarding the open orbits corresponding to $d=0$, it suffices to observe that for the direct sum
of nondegenerate skew-Hermitian spaces $V = V_a \oplus V_b$, we have the tensor product decomposition of their Weil representations:
\[\omega_{V,\psi,\mu}  \cong \omega_{V_a,\psi,\mu} \otimes \omega_{V_b,\psi,\mu},\]
as representations of $\U(V_a) \times  \U(V_b) \subset \U(V)$. Thus the final assertion is a direct consequence of the Kunneth theorem \cite[Theorem 3.1]{P3}, completing  the proof of  Proposition \ref{ps2}.
\end{proof}

\vskip 5pt
\subsection{\bf Branching for $\pi_1 \times \pi_2$}
We can now assemble the results of Proposition \ref{ps2} for the various $\U(V)$-orbits $[X]$ on $\GL(V)/ P$ to understand the branching problem
\[   \Ext^i_{\U(V)}( \pi_1 \times \pi_2, \omega_{V, \psi, \mu}). \]
The result is most definitive when only the open orbits have nonzero contribution. The following theorem, which is the main result of this section, 
gives a simple sufficient condition (temperedness of $\pi_1$ and $\pi_2$) for this to happen.
\vskip 10pt

\begin{thm} \label{ps1}
  Suppose that $\pi_1$ and $\pi_2$ are tempered representations of $\GL(V_a)$ and $\GL(V_b)$
  (with unitary central characters) , so that  $\pi = \pi_1 \times \pi_2$ is a tempered principal series representation
  of $\GL(V)$ for $V=V_a+ V_b$.
 If $[X]$ is a non-open orbit of $\U(V)$ on $\GL(V)/P$, then for $\pi_X$, the subquotient of $\pi$ supported on the
 orbit $[X]$, we have:
  \[\Ext^i_{\U(V)}[ \pi_X, \omega_{V,\psi,\mu}] = 0,\]
  for all $i \geq 0$.
\vskip 10pt

As a consequence, for all $i \geq 0$, one has:
   \begin{equation} \label{E:Ext1}
   \bigoplus_{\delta} \Ext^i_{\U(V_{\delta})}[ \pi, \omega_{V_{\delta},\psi, \mu}]    \end{equation}
   \[  \cong
    \bigoplus_{i=j+k} \left(\bigoplus_{\delta'} \Ext^j_{\U(V_{a,\delta'})} [ \pi_1,
     \omega_{V_{a, \delta'},\psi, \mu}] \right) \otimes
  \left(\bigoplus_{\delta''} \Ext^k_{\U(V_{b,\delta''})}[ \pi_2, \omega_{V_{b,\delta''},\psi, \mu}] \right)\]
  where the sums over $\delta$, $\delta'$ and $\delta''$ run over $F^{\times}/NE^{\times}$ or $E_0^{\times}/ NE^{\times}$ according to the parity of $n, a,b$, respectively.
  \vskip 5pt
  
   Hence, for $i=0$, one has:
     \vskip 5pt
     
     \begin{equation} \label{E:Ext1.5}   
      \Hom_{\U(V_{\delta})}[ \pi, \omega_{V_{\delta},\psi}]
   \cong \end{equation}
\[  \bigoplus_{(\delta', \delta''): V_{a,\delta'} \oplus V_{b,\delta''} \cong V_{\delta} } \Hom_{\U(V_{a,\delta'})}[ \pi_1, \omega_{V_{a,\delta'},\psi}] \otimes
  \Hom_{\U(V_{b,\delta''})} [ \pi_2, \omega_{V_{b,\delta''},\psi}] . \]
  In particular,
    
\begin{equation} \label{E:Ext2}
  \bigoplus_{\delta} \Hom_{\U(V_{\delta})}[ \pi, \omega_{V_{\delta},\psi}]
   \cong \end{equation}
\[  \left(\bigoplus_{\delta'} \Hom_{\U(V_{a,\delta'})}[ \pi_1, \omega_{V_{a,\delta'},\psi}] \right)\otimes
  \left(\bigoplus_{\delta''} \Hom_{\U(V_{b,\delta''})} [ \pi_2, \omega_{V_{b,\delta''},\psi}] \right). \]

   \end{thm}

\vskip 5pt

 \begin{proof}  
 By Proposition \ref{ps2},  $\Ext^i_{\U(V)}[ \pi_X, \omega_{V,\psi,\mu}]$ is equal to
  \[  
   \Ext^i_{S/N(S)}[ \delta_{P/S}^{1/2} \cdot (\pi_1)_{d,a-d} \otimes (\pi_2)_{b-d,d} , \, \,  \delta_{S}^{1/2} \cdot |\det|^{-1/2} \mu\cdot \omega_{V_{n-2d},\psi,\mu}] . \]  
By Casselman's criterion, the normalized Jacquet module of a tempered representation with respect to a parabolic $Q = L \cdot U$ of $\GL_n$ is a tempered representation of the Levi factor $L$  up to multiplication by a character $\delta_{Q}^{\alpha/2}$ for some real number $\alpha \geq 0$. 
Hence, since $\pi_1$ is tempered, 
\[
    (\pi_1)_{d,a-d}  =  \delta_{P_{d,a-d}}^{(1+\alpha)/2} \times ({\rm  a ~~tempered ~~ representation~~ of~~ } \GL_d(E) \times \GL_{a-d}(E)).  \]
Moreover, for  $(g,h) \in \GL_d(E) \times \U_{a-d}(F)$, 
  \[  \delta_{P_{d,a-d}}(g,h)^{(1+\alpha)/2} =  |\det(g)|^{(a-d+\epsilon)/2 }, \quad \text{with  $\epsilon=\alpha \cdot (a-d)$.} \]
  Similarly, 
  \[     (\pi_2)_{b-d,d}  =  \delta_{P_{b-d,d}}^{(1+\alpha')/2} \times ({\rm  a ~~tempered ~~ representation ~~of~~ } \GL_{b-d}(E) \times \GL_{d}(E)) \]
  and  for $(h, (g^*) ^{-1} ) \in  \U_{b-d}(F) \times \GL_d(E)$, 
  \[ \delta_{P_{b-d,d}}(h, (g^*)^{-1})^{(1+\alpha')/2}  =  |\det(g)|^{(b-d+\epsilon')/2}, \quad \text{with $\epsilon' = \alpha' \cdot (b-d)$.}  \]
 Summarizing, we have:
\begin{enumerate}
\item the representation
  \[ A = \delta_{P/S}^{1/2} \cdot   (\pi_1)_{d,a-d} \otimes (\pi_2)_{b-d,d} \]
   is the twist of  a unitary representation of $\GL_d(E) \times \U(V_{a-d}) \times \U(V_{b-d})$   by the character
  \[  |\det |^{(a-d+\epsilon)/2}  \cdot |\det|^{(b-d+\epsilon')/2} |\det|^{d/2} = |\det|^{(n-d+\epsilon+\epsilon')/2} \]
    of $\GL_d(E)$. 
\vskip 5pt

\item the representation
\[  B= \delta_{S}^{1/2} \cdot |\det|^{-1/2} \mu\cdot \omega_{V_{n-2d},\psi,\mu} \]
is the twist of a unitary representation of $\GL_d(E) \times \U(V_{n-2d})$  by  the character 
 \[  |\det|^{(n-d)/2} \cdot  |\det |^{-1/2} = |\det|^{(n-d-1)/2} \]
of $\GL_d(E)$. 
\end{enumerate}
\vskip 5pt

\noindent Thus, when $d > 0$,  the actions of the center of $\GL_d(E)$ in 
\[  S/N(S)=\GL_d(E) \times  \U(V_{a-d}) \times \U(V_{b-d}) \]
on the two representations $A$ and $B$ are different.
Therefore, by Lemma \ref{central},
\[ \Ext^i_{S/N(S)}[A,B] = 0, \]
for all $i \geq 0$ (as long as $d\not = 0$). This completes the proof
that for a non-open $\U(V)$-orbit $[X] \subset \GL(V)/P$, the associated subquotient $\pi_X$ of the $\U(V)$-module $\pi$  satisfies
   \[\Ext^i_{\U(V)}[ \pi_X, \omega_{V,\psi}] = 0 \quad \text{for all $i \geq 0$.} \]
 \vskip 5pt
 
 As a consequence of the vanishing of $\Ext^i$, $i \geq 0$, for all non-open orbits, we deduce that 
  \[\Ext^i_{\U(V)}[ \pi, \omega_{V,\psi,\mu}] = \bigoplus_X \Ext^i_{\U(V)}[ \pi_X, \omega_{V,\psi,\mu}],\]
  where now $X$ runs over the two open orbits of $\U(V)$ on $\GL(V)/P$. Since Proposition \ref{ps2} calculates 
$\Ext^i_{\U(V)}[ \pi_X, \omega_{V,\psi}]$ for the open orbits, 
 the proof of Theorem \ref {ps1} is completed. \end{proof}

\vskip 10pt

\section{\bf Application to Conjecture \ref{conj-local}.}  \label{S:UPS2}

In this section, we deduce the implications of  Theorem  \ref{ps1}  for  Conjecture \ref{conj-local}.
Indeed, we shall see how Theorem \ref{ps1} allows us to reduce  Conjecture \ref{conj-local} for tempered representations 
to the case of discrete series representations of $\GL_n(E)$. This allows us to
prove  Conjecture \ref{conj-local} for unitary principal series induced from the Borel subgroup. We also investigate if the Mackey theory argument allows one to further
reduce  Conjecture \ref{conj-local} to the case of supercuspidal representations. As we shall see, we will fall slightly short of that,  but we  will at least be able to prove Conjecture \ref{conj-local}  for the Steinberg representation of $\GL_n(E)$. 
 
\vskip 5pt

\vskip 5pt

\subsection{\bf Inductive argument}
Let us first record the following consequence of Theorem \ref{ps1}.
\vskip 5pt

\begin{cor} \label{ds} 
Let $V = V_a \oplus V_b$ be a direct sum of nondegenerate skew-Hermitian spaces, and let $\pi_1$ and $\pi_2$ be irreducible tempered representations of $\GL(V_a)$ and $\GL(V_b)$ respectively. Then we have the following.
\vskip 5pt

\noindent (i)  If Conjecture \ref{conj-local} holds
  for $\pi_1$  and $\pi_2$,   then  Conjecture \ref{conj-local} holds for the unitary principal series representation $\pi_1 \times \pi_2$ of $\GL(V)$. 
  
  \vskip 5pt
  
  \noindent (ii) If
  \[  
  \Ext^i_{\U(V_a)}(\pi_1, \omega_{V_a, \psi,\mu}) = \Ext^i_{\U(V_b)}(\pi_2, \omega_{V_b,\psi,\mu}) = 0 \quad \text{ for all $i>0$,}  \]
  then
     \[\Ext^i_{\U(V)}(\pi_1 \times \pi_2 , \omega_{V,\psi, \mu}) = 0  \quad \text{ for all $i>0$.}\]
  \end{cor}

\begin{proof}
The vanishing statement in (ii) follows from Theorem \ref{ps1}, especially (\ref{E:Ext1}). 
Likewise,  (\ref{E:Ext2}) imply that Conjecture   \ref{conj-local}(i) and (ii)   for $\pi_1 \times \pi_2$ follows from the corresponding statements for $\pi_!$ and $\pi_2$. 
Thus it remains to verify that   the unique 
  skew-Hermitian space $V_{\delta}$ for which $\Hom_{\U(V_{\delta})}(\pi_1 \times \pi_2 , \omega_{V_{\delta}, \psi, \mu})$ is nonzero
  is as predicted by  Conjecture \ref{conj-local}(iii).
  \vskip 5pt
  
  Assume then that 
  \[ \Hom_{\U(V)}(\pi_1 \times \pi_2 , \omega_{V, \psi, \mu}) \ne 0. \]
 By  (\ref{E:Ext1.5}),
 \[  \Hom_{\U(V_a)}(\pi_1 , \omega_{V_a, \psi, \mu}) \ne 0 \quad \text{and} \quad \Hom_{\U(V_b)}(\pi_2 , \omega_{V_b, \psi, \mu}) \ne 0 \]
 for a unique pair of skew-Hermitian spaces $V_a$ and $V_b$ satisfying $V_a \oplus V_b \cong V$. 
 As we have assumed that  Conjecture \ref{conj-local}(iii) holds for $\pi_1$ and $\pi_2$, we have: 
\[  \mu(\det(V_a)) = \epsilon( 1/2,  \pi_1 \times {}^\sigma\pi_1 ^\vee \times \mu^{-1}, \psi_E) \cdot \omega_{\pi_1}(-1)^a \cdot \omega_{E/F}(-1)^{a(a-1)/2},\]  
and 
  \[  \mu(\det(V_b)) = \epsilon( 1/2,  \pi_2 \times {}^\sigma\pi_2 ^\vee \times \mu^{-1}, \psi_E) \cdot \omega_{\pi_2}(-1)^b \cdot \omega_{E/F}(-1)^{b(b-1)/2}.\] 
  This implies that for $\pi= \pi_1 \times \pi_2$,
\[  \mu(\det(V)) = \epsilon( 1/2,  \pi \times {}^\sigma\pi ^\vee \times \mu^{-1}, \psi_E) \cdot \omega_{\pi}(-1)^n \cdot \omega_{E/F}(-1)^{n(n-1)/2},\]  
 using the facts that $\mu ={}^\sigma\mu ^{-1}$, $\mu(-1)= \omega_{E/F}(-1)$ and  
\[ \epsilon (\Pi + {}^\sigma\Pi ^\vee, \psi_E) = \det \Pi(-1)\]
for any representation $\Pi$ of $\GL_m(E)$.
This completes the proof of the corollary.\end {proof}

\vskip 10pt
\subsection{\bf Reduction to discrete series case}
 Corollary \ref{ds} allows one to reduce Conjecture \ref{conj-local} for tempered representations to the case of discrete series representations.
\vskip 5pt
\begin{cor}  \label{C:temp-to-ds}
If Conjecture \ref{conj-local} holds for all (unitary) discrete series representations of all $\GL(V)$, then it holds for all (unitary)  tempered representations of all $\GL(V)$.
\end{cor}

\vskip 5pt
\begin{proof}
This follows from Corollary \ref{ds} and the fact that any tempered non-discrete-series representation of $\GL(V)$ is irreducibly and unitarily induced from a discrete series representation of a proper parabolic subgroup. 
\end{proof}
\vskip 10pt

\subsection{\bf Borel-Principal series}
In addition, by applying Corollary \ref{ds} inductively, we deduce:
\vskip 5pt

\begin{cor}  \label{C:ps}
  Conjecture \ref{conj-local} holds for all irreducible unitary principal series representations  of $\GL(V)$
  induced from a unitary character of a Borel subgroup.
  \end{cor}

\begin{proof}
Using Corollary \ref{ds}, this follows by induction on $\dim V$. The base case, with $\dim V = 1$, is 
  Theorem \ref{T:moen} due to Moen and Rogawski.  
  \end{proof}
  \vskip 5pt
  
  \subsection{\bf An alternative proof} 
  We now give another proof of Corollary \ref{C:ps}  as it  brings out an interesting structure of the open orbits of $\U(V)$ on $\GL(V)/B$.
\vskip 5pt

  Suppose that
\[    \Pi = \Ind_B^{\GL(V)} (\chi_1 \otimes  \cdots \otimes \chi_n)  \quad \text{(normalized induction),}\]
so that its L-parameter is 
\[  M = \bigoplus_i \chi_i . \]
On restriction to $\U(V)$, Theorem \ref{ps1} inductively implies that only the open $\U(V)$-orbits on the flag variety $\GL(V)/B$ will contribute
to the Hom space $\Hom_{\U(V)}[ \Pi, \omega_{V,\psi}]$.  Moreover, using Lemma \ref{orbits} inductively, 
the open orbits can be described as follows.  Given an ordered collection  
\[   \mathcal{L} = \{L_1,\cdots , L_n \} \]
 of  nondegenerate
orthogonal lines in $V$, the $\U(V)$-orbit of the flag
\[  \mathfrak{F}_{\mathcal{L}}: L_1 \subset L_1 \oplus L_2 \subset \cdots  \]
is an open orbit, and the stabilizer of $\mathfrak{F}_{\mathcal{L}}$ in $\U(V)$ is the subgroup 
\[  \U(\mathcal{L}):= \prod_i \U(L_i). \]
Moreover, all open orbits are given by such an ordered collection $\{ L_i \}$ of isomorphism classes of
(nondegenerate) skew-Hermitian $E$-spaces of dimension 1, subject to the condition that $\oplus_i L_i \cong V$ as skew-Hermitian $E$-spaces; we say that such an $\mathcal{L}$ is $V$-relevant. There are thus $2^{n-1}$ open orbits, indexed by $V$-relevant $\mathcal{L}$'s. This can also be gleaned from a Galois cohomological argument: having fixed an open orbit over $F$ with stabilizer $\U(\mathcal{L})$ in $\U(V)$ and noting that there is exactly one open orbit over $\overline{F}$, the number of open $\U(V)$-orbits is given by:
\[  {\rm Ker}\left( H^1(F,  \U(\mathcal{L})) \rightarrow H^1(F, \U(V))\right) = {\rm Ker} \left( (F^{\times}/ N(E^{\times}))^n \rightarrow F^{\times}/ N(E^{\times}) \right). \] 
\vskip 5pt

Hence, by an inductive application of Theorem \ref{ps1},
we have
\[  \Hom_{\U(V)}(\Pi, \omega_{V,\psi,\mu}) \cong \bigoplus_{\mathcal{L}}   \Hom_{\U(V)} (\ind_{\U(\mathcal{L})}^{\U(V)} ( \boxtimes_i \chi_i ), \omega_{V,\psi,\mu}) \]    
where the sum runs over $V$-relevant $\mathcal{L}$'s.
By Frobenius reciprocity,  and the fact that
\[  \omega_{V,\psi,  \mu} |_{\U(\mathcal{L})} \cong  \bigotimes_i \omega_{L_i,\psi, \mu}, \]
one deduces that
\[  \Hom_{\U(V)}(\Pi, \omega_{V,\psi,\mu}) \cong  \bigoplus_{\mathcal{L}} \bigotimes_i  \Hom_{\U(L_i)}( \chi_i,   \omega_{L_i,\psi, \mu}).  \]
Now by Theorem \ref{T:moen} (the theorem of Moen and Rogawski),  
\[  \Hom_{\U(L_i)}( \chi_i,   \omega_{L_i,\psi, \mu}) \ne 0  \Longleftrightarrow \epsilon(1/2, \chi_i/ \chi_i^{\sigma} \cdot \mu^{-1}, \psi_E) \cdot \chi_i(-1) = \mu(\det(L_i)). \]
Hence, at most one term in the sum over $\mathcal{L}$ has nonzero contribution, and this unique $\mathcal{L}$ exists if and only if
\[  \mu(\det(V)) = \prod_i   \epsilon(1/2, \chi_i/ \chi_i^{\sigma} \cdot \mu^{-1}, \psi_E) \cdot \chi_i(-1)  \]
To prove Conjecture \ref{conj-local}, we explicate:
\begin{align}
 &\epsilon(1/2, M \otimes {}^\sigma \!M ^\vee \cdot \mu^{-1}, \psi_E) \notag \\
  = &\prod_i \epsilon(1/2, \chi_i/\chi_i^{\sigma} \cdot \mu^{-1},\psi_E) \cdot  \prod_{i<j} 
\epsilon\left( 1/2, (\chi_i/ \chi_j^{\sigma} + \chi_j/ \chi_i^{\sigma}) \cdot \mu^{-1}, \psi_E \right).  \notag 
\end{align}
For $i < j$, observe that
\begin{eqnarray*}  \epsilon\left( 1/2, (\chi_i/ \chi_j^{\sigma} + \chi_j/ \chi_i^{\sigma}) \cdot \mu^{-1}, \psi_E \right)  & = & \epsilon ( 1/2, \chi_i/ \chi_j^{\sigma}\cdot \mu^{-1}, \psi_E)  \epsilon ( 1/2, \chi^\sigma_j/ \chi_i\cdot \mu^{-1}, \psi_E) \\
  & =& \chi_i(-1) \cdot \chi_j(-1) \cdot \omega_{E/F}(-1), 
  \end{eqnarray*}
where we have used the following standard properties of the epsilon factor:
\begin{enumerate}
\item $\epsilon(1/2, W, \psi_E)  \cdot \epsilon(1/2, W^\vee, \psi_E)  = \det(W)(-1)$,
\item  $\epsilon(1/2, W, \psi_E) = \epsilon(1/2, W^\sigma, \psi_E)$.
  \end{enumerate}

  It follows that:
  \[  \prod_{i<j}
  \epsilon\left( 1/2, (\chi_i/ \chi_j^{\sigma} + \chi_j/ \chi_i^{\sigma}) \cdot \mu^{-1}, \psi_E \right)   = \det(M)(-1)^{n-1} \cdot  \omega_{E/F}(-1)^{n(n-1)/2}. \]
Putting everything together, we see that $\Hom_{\U(V)}(\Pi, \omega_{V,\psi,\mu})   \ne 0 $ if and only if
\[ \mu(\det(V)) =  \epsilon(1/2, M \otimes {}^\sigma \!M ^\vee \cdot \mu^{-1}, \psi_E) \cdot \det(M)(-1)^n \cdot  \omega_{E/F}(-1)^{n(n-1)/2}, \]
as desired. 

\vskip 5pt
\vskip 5pt

\subsection{\bf Reduction to supercuspidal case}  \label{SS:red-sc}
We have seen 
that  Conjecture \ref{conj-local} for tempered representations can be reduced to the case of discrete series representations by a Mackey theory argument. 
In the rest of the section, we investigate if the same argument can be used to reduce  Conjecture \ref{conj-local} for discrete series representations to the case of supercuspidal representations. It will turn out that this can be done under a certain hypothesis. While we cannot prove this hypothesis in general, it can be shown in some situations. This will allow us to prove  Conjecture \ref{conj-local} for the Steinberg representation, for example. 
 \vskip 5pt
 
 Let us first set up some notations and formulate the relevant hypothesis. Suppose that $\pi$ is a supercuspidal representation (with unitary central character) of $\GL(V) = \GL_m(E)$. 
 The parabolically induced representation 
  \[  \pi |\det|^{\frac{n-1}{2}} \times \pi|\det|^{\frac{n-3}{2}} \times ....\times \pi|\det|^{-\frac{n-1}{2}} \]
  of $\GL(V^{\oplus n}) \cong \GL_{mn}(E)$ is a standard module and thus has a unique irreducible quotient $\Sp(\pi,n)$, which is often called a Speh representation and is nontempered (if $n>1$).  
  This parabolically induced representation also has a unique irreducible submodule $\St(\pi,n)$;  this is the ``generalized Steinberg'' representation, which is a discrete series representation. All the irreducible (unitary) discrete series representations of general linear groups are of the form $\St(\pi, n)$. The supercuspidal ones are precisely those with $n=1$. 
  \vskip 5pt
   If
  \[ \phi_\pi: W_E\rightarrow \GL_m(\C) \]
  is the L-parameter of $\pi$, then the L-parameter  of
  $\St(\pi,n)$ is the representation $\phi_{\pi} \otimes \Sym^{n-1}(\C^2)$ of the Weil-Deligne group
  $WD_E= W_E \times \SL_2(\C)$. We also write  $[n]=\Sym^{n-1}(\C^2)$ for the unique irreducible
  $n$-dimensional representation
  of $\SL_2(\C)$, and write $\phi_{\pi}[n]$  for the L-parameter  of $\St(\pi,n)$.
\vskip 5pt

To deal with the generalised Steinberg representations, we will need to make
an assumption.
In this,  $V,W$ are  the two isomorphism calasses of skew-Hermitian spaces over $E$ of dimension $m$. 
Then we make the following assumption:

\[({\bf Assumption}) \hspace{1cm}  \begin{cases} \Hom_{\U(V+ V )}[\Sp(\pi,2),  \omega_{V+V,\psi, \mu}]  =  0, \\
  \Hom_{\U(V+ W )}[\Sp(\pi,2),  \omega_{V+W,\psi, \mu}]  =  0. \end{cases}\]
\vskip 5pt

\noindent We remark that this assumption is a case of the nontempered twisted GGP conjecture formulated below in  Conjecture \ref{nontemperedggp}.
\vskip 5pt

With this assumption formulated, our result is:
\vskip 5pt

\begin{thm} \label{sc-to-ds}
Let $\pi$ be a supercuspidal representation of $\GL(V)$ with $\dim V = m$. If $\pi$ satisfies 
    Conjecture \ref{conj-local} and the above (Assumption), then   Conjecture \ref{conj-local} holds for  the discrete series representations $\St(\pi, n)$ for all $n \geq 1$. 
\vskip 5pt

More precisely:

\vskip 5pt

(a) Suppose that $V$ and $W$ are the two isomorphism classes of skew-Hermitian spaces over $E$ of dimension $m$ and 
\begin{eqnarray*} \Hom_{\U(V)}(\pi,  \omega_{V,\psi, \mu}) & \cong &  \C, \\
\Hom_{\U(W)}(\pi,  \omega_{W,\psi, \mu}) & = & 0. \end{eqnarray*}
Then, under (Assumption),  one has:  
\begin{eqnarray*} \Hom_{\U(V^n)}(\St(\pi,n),  \omega_{V^n,\psi, \mu}) & \cong & \C, \\
\Hom_{\U(W+V^{n-1})}(\St(\pi,n),  \omega_{W+V^{n-1},\psi, \mu}) & = & 0 \end{eqnarray*}
for the two isomorphism classes of skew-Hermitian spaces $V^n, W+V^{n-1}$ 
of dimension $mn$ over $E$, 
\vskip 5pt

(b) If
\[  \mu(\det(V)) \stackrel{(1)}{=} \epsilon( 1/2,  \phi_{\pi} \times {}^\sigma\phi_{\pi} ^\vee \times \mu^{-1}, \psi_E) \cdot \omega_{\pi}(-1)^m \cdot \omega_{E/F}(-1)^{m(m-1)/2},\]
  then for the skew-Hermitian space $V^{ n}$,
 \begin{eqnarray*}  \mu(\det(V^{ n})) & = & \mu(\det(V)^n) \\
   & \stackrel{(2)}{=}& \epsilon( 1/2,  \phi_{\pi}[n] \times {}^\sigma\phi_{\pi} ^\vee[n] \times \mu^{-1}, \psi_E) \cdot \omega_{\pi}(-1)^{nm} \cdot \omega_{E/F}(-1)^{mn(mn-1)/2}.
\end{eqnarray*}
\end{thm}

 \vskip 5pt    

\begin{proof}
The first assertion of the Theorem (concerning the truth of Conjecture \ref{conj-local}) is an immediate consequence of statements (a) and (b). We shall prove these two statements in turn, starting with the simpler (b).

\vskip 5pt

\noindent {\bf Proof of Theorem \ref{sc-to-ds}(b):}
Recall  from \cite{T} that for an irreducible representation  $\lambda \otimes [n]$ of $WD_E=W_E \times \SL_2(\C)$, one has
  $$  \epsilon( \lambda \otimes [n]) = \epsilon(\lambda)^n \cdot \det (-F, \lambda^I)^{n-1}, $$
  where $\lambda^I$ denotes the subspace of $\lambda$ fixed by the inertia group $I$  
  and $F$ denotes the Frobenius element of $W_E/I$.
 \vskip 5pt
 
 On the other hand, by the  Clebsch-Gordon theorem,
  \[[n] \otimes [n] = [2n-1] \oplus [2n-3] \oplus \cdots \oplus  [1].\]
  In particular, only odd integers $(2d+1)$  appear in this decomposition.  It is easy to see  that
  in the expression 
  \[  \epsilon( \lambda \otimes [2d+1]) = \epsilon(\lambda)^{2d+1} \cdot \det (-F, \lambda^I)^{2d},  \]
the factor $\det (-F, \lambda^I)^{2d}$  is trivial for $\lambda$  a conjugate selfdual representation of $W_E$. Hence we find that 
\[   \epsilon( \lambda \otimes [2d+1]) = \epsilon(\lambda)^{2d+1}  \]
 for $\lambda$ a conjugate selfdual representation
 of $W_E$. These considerations applied to the conjugate selfdual representation $\lambda$ of $W_E$ associated to
 $  \pi \times {}^\sigma\pi ^\vee \times \mu^{-1}$ allows one to prove the identity
$(2)$ from the identity $(1)$; we leave the simple and pleasant computation to the reader.

 \vskip 5pt

 \noindent {\bf Proof of Theorem \ref{sc-to-ds}(a):}
 The proof of (a) depends on some intermediate results contained in the following series of lemmas.

  \vskip 5pt

\begin{lemma} \label{two-ps}
Let $\pi$ be a unitary supercuspidal representation of $\GL(V) \cong \GL_m(E)$. 
  \vskip 5pt
  
  \noindent (i)   One has a short exact sequence $\GL_{mn}(E)$-representations:
  \[ 0 \rightarrow K_n \rightarrow
  \nu^{-(n-1)/2}\pi \times \nu^{1/2} \St(\pi, n-1) \rightarrow \St(\pi, n) \rightarrow 0, \]   
  with $K_n$ an irreducible representation of $\GL_{mn}(E)$.
  \vskip 5pt
  
  \noindent (ii) The irreducible representation $K_n$ of $\GL_{mn}(E)$ arising in the exact sequence above,  sits in the following
  short exact sequence:
  
    \[ 0 \rightarrow L_{n-1} \rightarrow
  \nu^{-(n-2)/2}\Sp(\pi,2) \times \nu \St(\pi,n-2)
  \rightarrow K_n \rightarrow 0. \]

\end{lemma}
\begin{proof}
(i) The fact that the discrete series representation $\St(\pi, n)$
of $\GL_{mn}(E)$  appears as a quotient of the principal series $\nu^{-(n-1)/2}\pi \times \nu^{1/2} \St(\pi, n-1)$
is clear, since  $\St(\pi, n)$ is a quotient of the principal series representation
\[ \nu^{-(n-1)/2}\pi \times \nu^{-(n-3)/2} \pi \times \cdots \times \nu^{(n-1)/2}\pi .\]
It is well-known from Zelevinski \cite{Ze}  that the principal series $\nu^{-(n-1)/2}\pi \times \nu^{1/2} \St(\pi, n-1)$ has length
2, so that $K_n$ is irreducible and (i) is proved.
\vskip 5pt

\noindent (ii) Since  $K_n$ is irreducible,  it suffices to prove that there is a nonzero $\GL_{mn}(E)$-equivariant homomorphism from the principal series $\nu^{-(n-2)/2}\Sp(\pi,2) \times \nu \St(\pi,n-2)$ to $K_n$.
By the second adjointness theorem, this boils down to proving that the normalized Jacquet functor of $K_n$ for the opposite parabolic
of the maximal standard parabolic
with Levi $\GL_{2m}(E) \times  \GL_{(n-2)m}(E)$ contains the irreducible representation  $\nu^{-(n-2)/2}\Sp(\pi,2) \otimes \nu \St(\pi,n-2)$
of $\GL_{2n}(E) \times \GL_{(m-2)n}(E)$ as a submodule. We leave this simple calculation  to the reader.
  \end{proof}
\vskip 5pt

Next, we apply Proposition \ref{ps2} to the two principal series representations appearing in Lemma  \ref{two-ps}.
We do not perform the explicit calculation here, but simply summarize the results:

\begin{lemma} \label{para}
  Let $P$ denote the maximal parabolic subgroup of $\GL(V^n)$ from which the  principal series representation considered
  below is induced. Then we have:
\vskip 5pt

\begin{itemize}
\item[(i)]   For any  non-open $\U(V^{n})$-orbit $[X] \subset \GL(V^{ n})/P$, the associated subquotient $\pi_X$ of the $\U(V^{n})$-module
$\nu^{-(n-2)/2}\Sp(\pi,2)  \times \nu \St(\pi,n-2)$,   satisfies
   \[\Ext^i_{\U(V^{ n})}[ \pi_X, \omega_{V^{ n},\psi}] = 0,\]
   for all $i \geq 0$. If (Assumption) holds for $\pi$, then the above result holds for the open orbits as well, and therefore:
   \[\Hom_{\U(V^{ n})}[\nu^{-(n-2)/2}\Sp(\pi,2)  \times \nu \St(\pi,n-2), \omega_{V^{ n},\psi}] = 0.\]
\vskip 10pt

\item[(ii)]   For any  non-open $\U(V^{n})$-orbit $[X] \subset \GL(V^{ n})/P$, the associated subquotient $\pi_X$ of the $\U(V^{n})$-module
$\nu^{-(n-1)/2}\pi \times \nu^{1/2} \St(\pi, n-1)$,   satisfies
   \[\Ext^i_{\U(V^n)}[ \pi_X, \omega_{V^{ n},\psi}] = 0,\]
   for all $i \geq 0$.

   \vskip 10pt

\item[(iii)]   For any  non-open $\U(V^{n})$-orbit $[X] \subset \GL(V^{ n})/P$, the associated subquotient $\pi_X$ of the $\U(V^{n})$-module
$ \nu^{1/2} \St(\pi, n-1) \times \nu^{-(n-1)/2}\pi $,   satisfies
   \[\Ext^i_{\U(V^n)}[ \pi_X, \omega_{V^{ n},\psi}] = 0,\]
   for all $i \geq 0$.

  \end{itemize}
\end{lemma}
\vskip 5pt

With the above two lemmas at hand, let us now return to the proof of Theorem  \ref{sc-to-ds}(a).  By Lemma \ref{two-ps}(ii), combined with Lemma \ref{para}(i), we deduce that:
\[ \Hom_{\U(V^{ n})}[K_n,\omega_{V^{ n},\psi}] = 0.\]
Therefore by  Lemma \ref{two-ps}(i),
\[ \Hom_{\U(V^{ n})}(\St(\pi, n) , \omega_{V^{ n},\psi}) 
\cong \Hom_{\U(V^{ n})}(\nu^{-(n-1)/2}\pi \times \nu^{1/2} \St(\pi, n-1) , \omega_{V^{ n},\psi}).\]

\noindent Furthermore,  from  Lemma \ref{para}(ii),
\[ \Hom_{\U(V^{ n})}(\nu^{-(n-1)/2}\pi \times \nu^{1/2} \St(\pi, n-1) , \omega_{V^{ n},\psi})\]
is contributed by the submodule of the principal series representation
\[ \nu^{-(n-1)/2}\pi \times \nu^{1/2} \St(\pi, n-1) \]
 supported on the open orbits.
\vskip 5pt

Observe that the open orbits of the action of
$\U(V^n)$ on $\GL_{mn}(E)/P_{m,m(n-1)}$ are parametrized by the isomorphism classes of the skew-Hermitian subspaces of $V^n$ of dimension  $m=\dim (V)$. Thus, there are exactly two orbits,   represented by the skew-Hermitian spaces $V$ and $W$, with stabilizer in $\U(V^n)$ being $\U(V) \times \U(V^{n-1})$ and $\U(W) \times \U(W+V^{n-2})$. Therefore,
\[ \Hom_{\U(V^{ n})}(\St(\pi, n) , \omega_{V^{ n},\psi}) \cong  \Hom_{\U(V^{ n})}(\nu^{-(n-1)/2}\pi \times \nu^{1/2} \St(\pi, n-1) , \omega_{V^{ n},\psi})\ \]
 is the sum $A + B$ of the contributions coming from these two open orbits orbits, with 
  \begin{eqnarray*}
  A & = &   \Hom_{\U(V)}(\pi,  \omega_{V,\psi}) \otimes
  \Hom_{\U(V^{n-1})}(\St(\pi, n-1) , \omega_{V^{(n-1)},\psi}), \\
  B & = &  \Hom_{\U(W)}(\pi,  \omega_{W,\psi}) \otimes
  \Hom_{\U(W+V^{n-2})}(\St(\pi, n-1) , \omega_{W+V^{(n-2)},\psi}).
 \end{eqnarray*}
  Now as $\Hom_{\U(W)}(\pi,  \omega_{W,\psi})=0$, we conclude that $B = 0$. This completes the proof of part (a) of Theorem
  \ref{sc-to-ds},  and hence the proof of Theorem \ref{sc-to-ds} is completed.
\end{proof}
\vskip 5pt

\subsection{\bf The Steinberg representation}
At the moment, we do not know how to prove the (Assumption) that 
\[ \Hom_{\U(V^{ 2})}[\Sp(\pi,2),  \omega_{V^{ 2},\psi}] = \Hom_{\U(V+W)}[\Sp(\pi,2),  \omega_{V+W,\psi}] = 0, \]
except when $m=\dim V = \dim W =1$, i.e., when $\pi$ is a character of $E^\times$.  In this case, the (Assumption)  is equivalent to saying  that the Weil representation of $\U(2)$
does not contain any one dimensional representation of $\U(2)$ where $\U(2)$ is either of the two
unitary groups in two variables. The next lemma establishes this.
\vskip 5pt

\begin{lemma} \label{character}
 For  $V$ be a skew-Hermitian space over $E$, a nonarchimedean local field,  of dimension $d \geq 2$,   one has
  \[ \Hom_{\U(V)}[ \chi,  \omega_{V,\psi, \mu}] = 0\]
  for any 1-dimensional character $\chi$ of $\U(V)$.
\end{lemma}
  
\begin{proof} If $V = V_1\oplus V_2$, a direct sum of skew-Hermitian spaces, one knows that as representations
  of $\U(V_1) \times \U(V_2) \subset \U(V)$,
  \[\omega_{V,\psi, \mu}=  \omega_{V_1,\psi,\mu} \otimes  \omega_{V_2,\psi,\mu}.\]
  Therefore, the proof of the Lemma reduces to the case of $d=2$.
  \vskip 5pt
  
  When $d =2$, we have seen in the discussion in \S \ref{SS:rk2} that $\omega_{V, \psi,\mu}$ is a direct sum of irreducible summands (with different central characters), each of which has dihedral L-parameters. Hence, $\omega_{V, \psi,\mu}$  does not contain 1-dimensional characters of $\U(V)$. (This uses  the requirement for $E$ to be  a nonarchimedean local field!)
    \end{proof}

\vskip 5pt
As a consequence, we obtain:
\vskip 5pt

\begin{cor}  \label{C:Steinberg}
The Steinberg representation $\St$ of $\GL(V)$ satisfies Conjecture \ref{conj-local}.
\end{cor}
\vskip 10pt

\subsection{\bf Ext vanishing}
The results of this section also prove the following theorem on the vanishing of Ext groups for tempered representations.
\vskip 5pt

\begin{thm} \label{Chen}
Let $F$ be a nonarchimedean local field and $E$ a separable quadratic algebra over $F$. Let $V$ be a skew-Hermitian space over $E$, with corresponding unitary group $\U(V) \subset \GL(V)$. For  any irreducible tempered representation $\Pi$ of $\GL(V)$  and any Weil representation  $\omega_{V, \psi, \mu}$  of $\U(V)$,
\[   \Ext^i_{\U(V)}( \Pi, \omega_{V, \psi, \mu}) = 0 {\rm ~~~for ~~~all~~~} i \geq 1. \]
\end{thm}

\begin{proof}\footnote{The authors thank Rui Chen of Zhejiang university for his help with this proof; Rui Chen has used similar ideas --- dimension shifting as here, cf. \cite{Ch},
  to prove theorems about vanishing of Ext groups in many situations involving the GGP branching.}
  As any irreducible tempered representation of $\GL(V)$ is parabolically induced from an irreducible
  discrete series representation of a Levi subgroup, by Corollary \ref{ds}, it suffices to prove this theorem for
  the discrete series representations $\Pi= \St(\pi,n)$ of $\GL_{mn}(E)$
  where $\pi$ is a cuspidal representation of $\GL_m(E)$. We prove this by an induction on the integer $n$. 
  \vskip 5pt

  The base case $n=1$ is  clear since a supercuspidal representation $\pi$ of $\GL_m(E)$  is a projective representation when restricted to any subgroup $H \subset \GL_m(E)$ for which the intersection of $H$ with the center of $\GL_m(E)$ is compact. In particular,  this applies to $H=\U_m(E)$.
 \vskip 5pt
 
 For the inductive step, let us assume that the theorem holds good for   $\Pi= \St(\pi,n-1)$, so that our goal is to prove it for  $\Pi= \St(\pi,n)$.
Recall the exact sequence from Lemma \ref{two-ps}:
  \begin{equation} \label{E:tag1}  
  0 \rightarrow K_n \rightarrow
  \nu^{-(n-1)/2}\pi \times \nu^{1/2} \St(\pi, n-1) \rightarrow \St(\pi, n) \rightarrow 0, \end{equation}   
  with $K_n$ an irreducible representation of $\GL_{mn}(E)$.
Observe  that:
  \begin{equation} \label{E:tag1.5}
    \Ext^i_{\U(V)}(  \nu^{-(n-1)/2}\pi \times \nu^{1/2} \St(\pi, n-1) , \omega_{V, \psi, \mu}) = 0 {\rm ~~~for ~~~all~~~} i \geq 1. \end{equation}
  Indeed,  by Lemma \ref{para}(ii), one knows the vanishing of $\Ext^i, i \geq 0$ for the subquotient of the principal series representation $\nu^{-(n-1)/2}\pi \times \nu^{1/2} \St(\pi, n-1)$  of $\GL_{mn}(E)$ supported on a non-open orbit. For the open orbits,
  the vanishing  of $\Ext^i, i \geq 1$ is consequence  of the induction hypothesis and the Kunneth theorem.
\vskip 5pt

  Equipped with this vanishing of   $\Ext^i_{\U(V)}(  \nu^{-(n-1)/2}\pi \times \nu^{1/2} \St(\pi, n-1) , \omega_{V, \psi, \mu})$  for all $i \geq 1, $ the usual long exact sequence of Ext groups associated to the short exact sequence of modules in (\ref{E:tag1}) gives us isomorphisms:

  \begin{equation} \label{E:tag2}  \Ext^{i+1}_{\U(V)}( \St(\pi, n) , \omega_{V, \psi, \mu}) \cong  \Ext^i_{\U(V)}( K_n , \omega_{V, \psi, \mu}) {\rm ~~~for ~~~all~~~} i \geq 1. 
  \end{equation}
\vskip 5pt

  Next,  we use the pinned outer automorphism $\phi$ on $\GL_{mn}(E)$ which is a conjugate of the automorphism $g \rightarrow {}^tg^{-1}$ by a (longest) 
Weyl group element. The outer automorphism $\phi$ 
  takes standard parabolic subgroups to standard parabolic subgroups, and in particular takes the parabolic
  $P_{m,m(n-1)}$  to $P_{m(n-1),m}$. Moreover, for an element  $(g_1,g_2)\in \GL_m(E) \times \GL_{m(n-1)}(E)$ in the Levi subgroup  $\GL_m(E) \times \GL_{m(n-1)}(E)$ of   $P_{m,m(n-1)}$, one has
  \[  \phi(g_1, g_2)  =  ({}^tg_2^{-1}, {}^tg_1^{-1}) \in \GL_{m(n-1)}(E) \times \GL_m(E). \]
\vskip 5pt

  Applying $\phi$ to  the exact sequence (\ref{E:tag1}) above,
    we obtain:
  \[ 0 \rightarrow K^\phi_n \rightarrow
     [\nu^{-(n-1)/2}\pi \times \nu^{1/2} \St(\pi, n-1)]^\phi \rightarrow \St(\pi, n)^\phi \rightarrow 0. \]
\noindent
 By transport of structure, 
    \[  [\nu^{-(n-1)/2}\pi \times \nu^{1/2} \St(\pi, n-1)]^\phi \cong   [\nu^{1/2}\St(\pi, n-1)]^\phi  \times [\nu^{-(n-1)/2}\pi]^\phi.\]
Now by a well-known theorem of Gelfand-Kazhdan, the action of $\phi$ on any irreducible
   representation  of $\GL_{mn}(E)$ is just the contragredient. Thus, we obtain the exact sequence:
\[ 0 \rightarrow K^\vee_n \rightarrow
   \nu^{-1/2} \St(\pi, n-1)^\vee \times (\nu^{(n-1)/2}\pi^\vee)
   \rightarrow \St(\pi, n)^\vee \rightarrow 0. \]
\vskip 5pt

 Taking the contragredient of this exact sequence, we get:

\begin{equation} \label{E:tag3}
0 \rightarrow \St(\pi,n)  \rightarrow
   \nu^{1/2}\St(\pi, n-1) \times (\nu^{-(n-1)/2}\pi)
   \rightarrow K_n \rightarrow 0. 
   \end{equation}
   Once again, we have:
\begin{equation} \label{E:tag3.5}
   \Ext^i_{\U(V)}(  \nu^{1/2} \St(\pi, n-1) \times \nu^{-(n-1)/2}\pi, \omega_{V, \psi, \mu}) = 0 {\rm ~~~for ~~~all~~~} i \geq 1. \end{equation}
As for (\ref{E:tag1.5}), this  follows by   Lemma \ref{para}(iii), which gives the vanishing of $\Ext^i$ ($i \geq 0$) for the subquotient of the principal series $\nu^{1/2} \St(\pi, n-1) \times \nu^{-(n-1)/2}\pi$  supported on a non-open orbit, and the vanishing  of $\Ext^i$ ($i \geq 1$) for the open orbits is a consequence of the induction hypothesis and the Kunneth theorem. 
\vskip 5pt

  Equipped with this vanishing of   $\Ext^i_{\U(V)}(
 \nu^{1/2} \St(\pi, n-1) \times \nu^{-(n-1)/2}\pi, \omega_{V, \psi, \mu})$  for all $i \geq 1, $ the usual long exact sequence of Ext groups associated to the short exact sequence of modules in (\ref{E:tag3}) gives us isomorphisms:

 \begin{equation} \label{E:tag4}
    \Ext^{i}_{\U(V)}( \St(\pi, n) , \omega_{V, \psi, \mu}) \cong  \Ext^{i+1}_{\U(V)}( K_n , \omega_{V, \psi, \mu}) {\rm ~~~for ~~~all~~~} i \geq 1. 
    \end{equation}
Using the isomorphisms in (\ref{E:tag2}) and (\ref{E:tag4}), we get:
  \[   \Ext^{i}_{\U(V)}( \St(\pi, n) , \omega_{V, \psi, \mu}) \cong    \Ext^{i+2}_{\U(V)}( \St(\pi, n) , \omega_{V, \psi, \mu}) {\rm ~~~for ~~~all~~~} i \geq 1.\]

Now for any two smooth representations $\pi_1,\pi_2$ of a reductive $p$-adic group $G(F)$, one has
\[  \Ext^i_G[\pi_1,\pi_2]=0  \quad \text{ for any $i> $ the $F$-rank of $G$.} \]
Hence, we deduce by (\ref{E:tag4}) that  
  \[   \Ext^{i}_{\U(V)}( \St(\pi, n) , \omega_{V, \psi, \mu}) = 0 {\rm ~~~for ~~~all~~~} i \geq 1,\]
  completing the proof of the theorem.
\end{proof}

\section{\bf Archimedean Case}  \label{S:arch}
In this section, we consider the archimedean case, so that $\GL(V) = \GL_n(\C)$. As mentioned before, the local conjecture does not determine the unique skew-Hermitian space $V$ which has nonzero contribution.  In this section, we shall explain how the conjecture can be refined in the archimedean case to give a definitive answer. 
\vskip 5pt

Recall that Hermitian forms over $\C$ are classified by their signatures $(p, q)$. Since skew-Hermitian forms can be obtained from Hermitian ones by scaling-by $i$, we shall likewise say that a skew-Hermitian space has signature $(p,q)$ if it has $p$ many $i$'s and $q$ many $(-i)$'s as its eigenvalues. We will denote the corresponding space as $V_{p,q}$ and its isometry group as $\U(V_{p,q}) = \U_{p,q}$. In particular, in rank $1$, the two skew-Hermitian forms are classified by their determinant, which is $i$ or $-i$. 

\vskip 5pt

An irreducible generic representation $\Pi$ of $\GL_n(\C)$ is an irreducible  principal series representation:
\[    \Pi = \Ind_{B(\C)}^{\GL_n(\C)} (\chi_1 \otimes \cdots \otimes \chi_n)  \quad \text{(normalized induction)}\]
where the $\chi_j$'s are characters of $\C^{\times}$. We may write $\chi_j$ as:
\[  \chi_j(z)  = |z|^{r_j} \cdot \left( \overline{z}/ z \right)^{k_j/2}  \]
where $k_j \in \Z$.   
\vskip 5pt

As in the previous section, we may consider the restriction of the representation $\Pi$ of $\GL_n(\C)$
to a subgroup $\U(V_{p,q}) = \U_{p,q} \subset \GL_n(\C)$ by Mackey theory.
The open $\U(V_{p,q})$-orbits on the flag variety  $\GL_n(\C)/B$ 
are associated,  as in the $p$-adic case, to the ordered collection
of orthogonal (nondegenerate) skew-Hermitian lines $\mathcal{L} = \{ L_1,\cdots ,L_n\}$,
with $\oplus_j L_j \cong V_{p,q}$ as skew-Hermitian spaces.
This means that $p$ of the lines $L_i$'s have determinant  $i$ and the rest have determinant $-i$;
we shall call such $\mathcal{L}$'s to be $V_{p,q}$-relevant.
In particular, the number of open $\U(V_{p,q})$-orbits is $\binom{n}{p}$.
\vskip 5pt

If we assume that the analog of Theorem \ref{ps1} holds in the archimedean case, then the proof of Corollary \ref{C:ps} gives:
\[  \Hom_{\U(V)}(\Pi, \omega_{V, \psi,\mu}) \cong  \bigoplus_{\mathcal{L} }  \bigotimes_j  \Hom_{\U(L_j)}( \chi_j,   \omega_{L_j,\psi, \mu}),  \tag{*}\]
where the sum is taken  over those $\mathcal{L}$ which are $V_{p,q}$-relevant.
For each $i$, one may apply Theorem \ref{T:moen} \cite{Mo, R}:
\[ \Hom_{\U(L_j)}( \chi_j,   \omega_{L_j,\psi, \mu}) \ne 0 \Longleftrightarrow \epsilon(1/2, \chi_j/ \overline{\chi_j} \cdot \mu^{-1}, \psi_E) \cdot \chi_j(-1) = \mu(\det(L_j)),  \]
which shows that at most one $\mathcal{L}$ can have a nonzero contribution to the sum in (*).
Now let us explicate this local root number condition. 
\vskip 5pt

The conjugate-symplectic character $\mu$ of $\C^{\times}$ has the form
\[ \mu(z)  =\left( \frac{\bar{z}}{z} \right)^{\alpha} \quad \text{  with $\alpha \in \frac{1}{2} \Z \setminus \Z$.} \]
Observe that
\[  \mu(i) =  i^{-2 \alpha}. \]
Then writing $\chi$ in place of  $\chi_j$ for simplicity,
\[  \chi/ \overline{\chi} \cdot \mu^{-1} : z \mapsto  \left( \frac{\bar{z}}{z} \right)^{k  - \alpha}. \]
Hence,  if  $\psi$ is the additive character of $\R$ given by
\[   \psi(x)  =   e^{2 \pi i x}, \]
then by \cite[3.2.5]{T} (see also \cite[Prop. 2.1]{GGP2})
\[  \epsilon(1/2, \chi/ \overline{\chi} \cdot \mu^{-1}, \psi({\rm Tr})) =   
  {\rm sign}(k - \alpha) \cdot  i^{2k - 2 \alpha}  = {\rm sign}(k - \alpha) \cdot (-1)^k \cdot i^{-2 \alpha}. \]
 Hence, we conclude that
 \begin{align}
   \Hom_{\U(L_j)}( \chi_j,   \omega_{L_j, \psi, \mu}) \ne 0 &\Longleftrightarrow   \mu(\det(L_j)) = {\rm sign}(k_j - \alpha) \cdot i^{-2 \alpha}, \notag \\
   &\Longleftrightarrow \det(L_j)  = {\rm sign}(k_j - \alpha) \cdot i. \notag 
   \end{align}
 For this to hold with $\mathcal{L}$ being $V_{p,q}$-relevant, we need
 \[  \#\{ j: k_j > \alpha\} = p \quad \text{and} \quad \# \{j: k_j < \alpha   \} = q = n-p. \]
  \vskip 5pt
  
  Hence, our refinement of Conjecture \ref{conj-local} in the archimedean case is:
  \vskip 5pt
  
  \begin{conj}
  Assume that $E/F = \C / \R$. Let
  \[    \Pi = \Ind_{B(\C)}^{\GL_n(\C)} (\chi_1 \otimes  \cdots \otimes \chi_n)  \]
  be an irreducible generic principal series representation of $\GL_n(\C)$ with
  \[ \chi_j(z)  = |z|^{r_j} \cdot \left( \overline{z}/ z \right)^{k_j/2}, \quad  k_j \in \Z, \]
  and let
  \[  \mu(z)  =\left( \frac{\bar{z}}{z} \right)^{\alpha} \quad \text{  with $\alpha \in \frac{1}{2} \Z \setminus \Z$.} \]
  Then for $\psi(x) = e^{2 \pi ix}$,
\[  \Hom_{\U(V_{p,q})}(\Pi, \omega_{V_{p,q},\psi,\mu}) \ne 0 \]
if and only if
\[   \#\{ j: k_j > \alpha\} = p \quad \text{and} \quad \# \{j : k_j< \alpha   \} = q = n-p. \]
\end{conj}
\vskip 5pt

We have essentially proved this conjecture by our open-orbit analysis above, under the hypothesis that Theorem \ref{ps1} holds in the archimedean case.
We leave the analysis of non-open orbits and the resulting extension problems to more capable hands.

\vskip 15pt

\section{\bf The Conjecture for A-Parameters} \label{ggp3}
In this section, we shall extend Conjecture \ref{conj-local}  beyond the setting of generic representations to the setting of nontempered representations of Arthur type, analogous to what we did in \cite{GGP3} for the classical GGP conjectures. We begin with a brief recollection of this nontempered  conjecture from  \cite{GGP3}.
\vskip 5pt

\subsection{\bf Nontempered GGP and Relevance} 
In \cite{GGP3}, we considered the problem of determining
\[  \dim \Hom_{\GL_n(F)} ( \pi_M, \pi_N) = \text{$0$ or $1$ }  \]
where $\pi_M$  and $\pi_N$ are  respectively  irreducible  representations of $\GL_{n+1}(F)$ and $\GL_n(F)$   of Arthur type, with associated A-parameters
\[  M_A = \bigoplus_{i=1}^k   M_i \boxtimes \Sym^{d_i-1}(\bC^2) \quad \text{and} \quad    N_A = \bigoplus_{i=1}^l   N_i \boxtimes \Sym^{e_i-1}(\bC^2). \]
Here, $M_i$ and $N_i$ are irreducible  bounded admissible representations of the Weil-Deligne group $WD_F$ and $\Sym^{d-1}(\bC^2)$ is the $d$-dimensional irreducible representation of $\SL_2(\bC)$ (the Arthur $\SL_2(\bC)$), so that $M_A$ and $N_A$ are representations of $WD_F \times \SL_2(\bC)$ of dimension $n+1$ and $n$ respectively. 
The associated A-packets are singletons, containing the irreducible unitary principal series representations:
\[  \pi_M  = \times_{i=1}^r \Sp( \pi_{M_i}, d_i) \quad \text{and} \quad \pi_N = \times_{i=1}^{l} \Sp(\pi_{N_i}, e_i), \]
where $\pi_{M_i}$ refers to the irreducible representation of the appropriate $\GL$ with L-parameter $M_i$ and $\Sp(\pi_{M_i}, d_i)$ denotes the associated Speh representation (as introduced in \S \ref{SS:red-sc}).
\vskip 5pt

\noindent{\bf Remark:} We take this opportunity to correct a misnomer in \cite[\S 5]{GGP3}. In the first paragraph of \cite[Pg. 2312]{GGP3}, the representation with A-parameter $M_i \otimes \Sym^{d_i} (\bC^2)$ was denoted by ${\rm Speh}(\pi_{M_i}, d_i)$. Though just a naming convention,  it is more customary to denote this representation by ${\rm Speh}(\pi_{M_i}, d_i+1)$. We have followed the latter convention here. 
\vskip 5pt

Now  the main conjecture in \cite{GGP3} (for the general linear groups) is that 
  \[  \dim \Hom_{\GL_n(F)} ( \pi_M, \pi_N) = 1 \]
 if and only if the pair $(M_A, N_A)$ is a {\em relevant pair of A-parameters}. This conjecture has now been proven by K.Y. Chan \cite{C}. Our goal here is to recall the key notion of ``relevance" and make a couple of remarks about it, especially in the context of classical groups.
 \vskip 5pt
 
 \begin{definition} \label{D:relevance}
 Given two A-parameters (of arbitrary dimensions) of $\GL$-groups
 \[  M_A = \bigoplus_{i=0}^d   M_i \boxtimes \Sym^{i}(\bC^2) \quad \text{and} \quad    N_A = \bigoplus_{i=0}^d   N_i \boxtimes \Sym^{i}(\bC^2), \]
we say that $(M_A, N_A)$ is a relevant pair if 
we have a decomposition of the respective representations of $WD_F$ as
\[  M_i = M_i^+ + M_i^- \quad \text{and} \quad N_i = N_i^+ + N_i^- \]
 with the property that
\[  M_i^+  = N_{i+1}^- \quad \text{for $i \geq 0$ and} \quad M_i^- = N_{i-1}^+ \quad \text{for $i \geq 1$.} \]
 \end{definition}
\noindent  This combinatorial definition has a more geometric interpretation which was discussed in \cite[\S 4]{GGP3}.
 
 \vskip 5pt
 
\subsection{\bf Relevance for classical groups}
 We  make a few remarks on the relevance condition for classical groups, clarifying \cite{GGP3}. 
 \vskip 10pt
 
 \begin{itemize}
 \item  The first point is minor but worth noting.  The typical GGP conjecture (in the context of $\GL_n \times \GL_{n+1}$ say) is formulated as the branching problem of determining
 \[  \dim \Hom_{\GL_n(F)^{\Delta}} (\pi_M \otimes \pi_N, \C), \quad \text{rather than } \quad   \dim \Hom_{\GL_n(F)} ( \pi_M, \pi_N).  \]
 When formulated in this way, the nontempered GGP conjecture would then say that
 \[ \dim \Hom_{\GL_n(F)^{\Delta}} (\pi_M \otimes \pi_N, \C) = 1 \]
 if and only if $(M_A, N_A^{\vee})$ is relevant, where $N_A^{\vee}$ is the dual representation of $N_A$.

  \vskip 5pt
  
  \item Secondly, in \cite[\S 6] {GGP3}, we formulated the nontempered GGP conjecture for the classical groups, asserting that the same ``relevance" condition plays a crucial role. 
  We take this opportunity to explicate the relevance notion here. 
  \vskip 5pt
  
  For classical groups, the branching problem concerns the determination of  
  \[  \dim \Hom_H(\pi, \nu), \]
  where $\pi$ is an irreducible representation of 
\[  G=  G_1 \times G_2   = \U_n \times \U_m \quad \text{(say),}  \]
with $n\geq m$,   
\[  H= \U_m\ltimes N  \subset G  \]
is a subgroup with unipotent radical $N$ and $\nu$ is a certain small representation of $H$. More precisely,  $\nu$ is a 1-dimensional character if $n \not \equiv m \bmod 2$; this case is referred to as the {\em Bessel  case for Hermitian spaces}.  On the other hand, the case when  $n \equiv m \bmod 2$ is referred to as the {\em Fourier-Jacobi 
  case for skew-Hermitian spaces}; in this case,   $\nu$ is a Weil representation.

\vskip 5pt

The A-parameters for classical groups are likewise finite dimensional representations of
$WD_E \times \SL_2(\bC)$, where $WD_E$ 
is the Weil-Deligne group of $E$,  with appropriate (conjugate)-duality conditions. 
Suppose we are given A-parameters 
 \[  M_A = \bigoplus_{i=0}^d  M_i \boxtimes \Sym^{i}(\bC^2), \]
\[  N_A = \bigoplus_{i=0}^d   N_i \boxtimes \Sym^{i}(\bC^2), \]
with $M_i$ and $N_i$ satisfying  appropriate (conjugate-)duality conditions. 
We can now  summarize the relevance conditions required in each case as follows.
\vskip 5pt

\begin{itemize}
\item[(i)]  Orthogonal and symplectic groups  (both Bessel and Fourier-Jacobi models): an A-parameter $M_A \boxtimes N_A$ of 
  $G_1 \times G_2$ is relevant if and only if $M_A= M_A^{\vee}$ and $N_A= N_A^{\vee}$ form a relevant pair in the sense of Definition \ref{D:relevance} for $\GL_m \times \GL_n$.
  \vskip 5pt
  
  \item[(ii)]  Hermitian case  (Bessel models): an A-parameter $M_A \boxtimes N_A$ of 
    $G_1 \times G_2$ is relevant if and only if $M_A^\vee$ and $N_A$ form a  relevant pair in the sense of \ref{D:relevance}  for $\GL_m\times \GL_n$.
\vskip 5pt

  \item[(iii)]  Skew-Hermitian case (Fourier-Jacobi model): in this case, the definition of the Weil representation $\nu$ requires an extra piece of data, namely a character 
  \[  \mu: E^\times \rightarrow \C^\times \quad \text{with  $\mu|_{F^\times} = \omega_{E/F}$.} \]
   An A-parameter $M_A \boxtimes N_A$ of 
    $G_1 \times G_2$ is relevant if and only if $\mu \cdot M_A^\vee$ and $N_A$ are relevant in the sense of \ref{D:relevance}  for
    $\GL_m \times  \GL_n$.

  \end{itemize}
\end{itemize}

\subsection{\bf Nontempered Twisted GGP}
 We shall now formulate the extension of the nontempered GGP conjecture of \cite{GGP3} to the twisted setting considered in this paper.
 Hence, with $E/F$ a quadratic extension, suppose we have a representation $\pi_M$ of $\GL(V)$ with associated A-parameter  $M_A$. 
 The  notion of relevance is not immediately obvious in this setting, as in contrast to the situations discussed above, we do not have a pair of A-parameters but only a single one. Nonetheless, we have:
 \vskip 5pt

\begin{conj}  \label{nontemperedggp}
Let $V$ be an
  $n$-dimensional $E/F$-skew-Hermitian space.
Let $\pi$ be an irreducible admissible representation of $\GL(V)$ with an A-parameter
(which   is an $n$-dimensional representation of $WD_E \times \SL_2(\bC)$)  of the form
\[  M_A = \bigoplus_{i=1}^r   M_i \boxtimes \Sym^{d_i}(\bC^2), \]
where $M_i$ is an irreducible $m_i$-dimensional tempered representation of $WD_E$. If
\[ \Hom_{\U(V)}[  \pi, \omega_{V,\psi,\mu}] \not = 0,\]
 then  $M_A$ is a sum of a tempered A-parameter (i.e. with $\SL_2(\bC)$ acting trivially) and 
 summands of the form
\[N_i \boxtimes \Sym^{d_i}(\bC^2) \oplus \mu\cdot N_i^\sigma \boxtimes \Sym^{d_i-1}(\bC^2),\]
  where $d_i \geq 1$, and the $N_i$ are tempered representations of $WD_E$ with $N^\sigma_i$ their
  conjugate under the action of $\Gal(E/F)$. Equivalently, the parameters $ M_A$ and $\mu \cdot M_A^\sigma$ should be relevant in the sense of \cite{GGP3}.
\vskip 5pt

  Conversely, if the parameters $M_A$ and $\mu \cdot M_A^\sigma$ are relevant in the sense of \cite{GGP3}, then 
\[ \Hom_{\U(V)}[  \pi, \omega_{V,\psi,\mu}] = \C\]
 for exactly one skew-Hermitian space $V$, namely the one determined as in  Conjecture \ref{conj-local}(iii).
\end{conj}

\vskip 5pt

We leave it to the reader to verify that when $E = F \times F$ is split, so that $V = V_1 \times V_2$, the relevance condition in Conjecture \ref{nontemperedggp}
reduces to the one formulated earlier for a representation $\pi_M = \pi_1 \otimes \pi_2$ of $\GL(V) = \GL(V_1) \times \GL(V_2)$. 
\vskip 15pt

\subsection{\bf Degenerate principal series}
The reader may wonder how we are led to the above conjecture. In fact, we are led to the conjecture by considering the branching problem for degenerate principal series representations. Recall that in the previous three sections, we have appealed to Mackey theory computations to study the twisted branching problem for tempered principal series representations and generalized Steinberg representations. As much of the material there is of a general nature, it is natural to apply them to the analogous restriction problem for degenerate principal series representations. The result is given in the following proposition. Note that the degenerate principal series considered below are of Arthur type. Hence, the proposition serves as a motivation and check for Conjecture \ref{nontemperedggp}.
\vskip 10pt

\begin{prop} \label{relevance}
  Let
  \begin{itemize}
  \item   $n = a+b$,  with $0 <a \leq b \in \Z$;
  \item  $\chi_1,\chi_2: E^\times \rightarrow \C^\times$ be two unitary characters;
  \item $V = V_a \oplus V_b$ be an $n$-dimensional $E/F$-skew-Hermitian space, with $\dim V_a = a$;
  \item $P = P_{a,b}$ the maximal parabolic subgroup of $\GL(V)$ stabilizing $V_a$, with Levi factor $\GL(V_a) \times \GL(V_b)$;
  \item $\pi = \chi_1 \times \chi_2$ be the degenerate principal series representation
of $\GL(V)$ induced from the corresponding 1-dimensional character $(\chi_1 \circ \det_{V_a}) \otimes (\chi_2 \circ \det_{V_b})$ of  $P_{a,b}$.
\end{itemize}
If
  \[ \Hom_{\U(V)}[  \pi, \omega_{V,\psi,\mu}] \not = 0,\]
  then the following holds:
  \begin{itemize}
  \item[(i)]  $b=a+1$ and,
  \item[(ii)]  $\chi_1 = \chi_2^\sigma \cdot \mu$ where $\sigma$ is the Galois involution of $E/F$.
    \end{itemize}
\vskip 5pt

\noindent Conversely,   if $b=a+1$ and $\chi_1 = \chi_2^\sigma  \cdot \mu$, then
  there is exactly one skew-Hermitian structure on  $V$ such that 
\[ \Hom_{\U(V)}[  \pi, \omega_{V,\psi,\mu}]  = \C,\]
and for the other skew-Hermitian space $V'$,
\[ \Hom_{\U(V')}[  \pi, \omega_{V',\psi,\mu}]  = 0.\]
  \end{prop}

\begin{proof}
We shall apply the results from Mackey theory obtained in Proposition \ref{ps2}.
 Recall that the orbits for  the action of $\U(V)$ on $X= \GL_n(E)/P_{a,b}$  are given by Lemma \ref{orbits}. For an $a$-dimensional
  subspace $X \subset V$ with $\dim(X \cap X^\perp)=d$ with the corresponding subquotient 
  $\pi_X$ of  $\pi$,     Proposition  \ref{ps2} says that:
\[ \Ext^i_{\U(V)}[\pi_X, \omega_{V,\psi,\mu}]  \]
\[  {\cong} \, \,
 \Ext^i_{Q/N(Q)}[ (\pi_1)_{d,a-d} \otimes (\pi_2)_{b-d,d} \otimes \delta_{P/Q}^{1/2},  \delta_{Q}^{1/2} \cdot |\det|^{-1/2} \mu\cdot \omega_{V_{n-2d},\psi,\mu}]  \]
where $Q/N(Q)=\GL_d(E) \times \U_{a-d}\times \U_{b-d}$ and the other notations are as given there.
Applying this to $\pi_1=\chi_1$ and  $\pi_2=\chi_2$, we deduce:
 \[ \Ext^i_{\U(V)}[  \pi_X, \omega_{V,\psi,\mu}]   \]
 \[   \cong   \Ext^i_{Q/N(Q)}[ \chi_1 \cdot (\chi_2^\sigma)^{-1} |\det|^{d/2}, |\det|^{(n-d)/2} |\det|^{-1/2}\mu \cdot \omega_{V_{n-2d},\psi,\mu}].  \]
 We shall now study when this Ext group can be nonzero.
 \vskip 5pt
 
 Consider first the case when $[X]$ is an open orbit (so that $d = 0$) and $i=0$. In this case
 each of the $|\det|$ factors which
 refers to $\GL_d(E)$,       are trivial for $d=0$, hence
   it follows by Lemma \ref{character} that 
 \[  \Hom_{\U(V)}[  \pi_X, \omega_{V,\psi,\mu}]   =  \Hom_{\U(V_a)}( \chi_1,  \omega_{V_a, \psi,\mu}) \otimes \Hom_{\U(V_b)}(\chi_2, \omega_{V_b, \psi,\mu}) = 0. \]
On the other hand, when $d >  0$, it follows by Lemma \ref{central} (on matching powers of $|\det|$ for the two arguments) that a necessary condition for the nonvanishing of the above Ext group is:
  \[ 2d+1=n. \]
  Since $d\leq a \leq b < n=(a+b)$, this implies that we must have:
  \[  d=a \quad \text{and} \quad  b=a+1, \]
 which means that $[X]$ is  the unique closed orbit of $\U(V)$ on $\GL_n(E)/P_{a,b}$. In particular, $\pi_X$ is a quotient of $\pi$. 
 
\vskip 5pt
 
With $a$, $b$ and $d$ related as above, the Ext group in question is:
 \[
\Ext^i_{\U(V)}[  \pi_X, \omega_{V,\psi,\mu}]   \cong   \Ext^i_{\GL_a(E) \times \U_1} [ \chi_1 \cdot (\chi_2^\sigma)^{-1}, \mu \cdot \omega_{V_{1},\psi,\mu}],\]
where   $\omega_{V_{1},\psi,\mu}$ is a Weil representation of $\U(V_1)= \U_1$ for the 1-dimensional skew-Hermitian space $V_1$ with ${\rm disc}(V_1) = {\rm disc}(V)$, and we are regarding  $\omega_{V_{1},\psi,\mu}$  as a representation of $\GL_a(E) \times \U_1$.
Thus, we see that
\[  \chi_1 = \chi_2^\sigma \cdot \mu \]
 is a necessary condition for the nonvanishing of this Ext group.   When this condition holds, the above Ext group becomes $\Ext^i_{\U(V_1)}( 1, \omega_{V_1, \psi, \mu})$ and this vanishes if $i > 0$ (since $\U(V_1)$ is compact). 
\vskip 5pt

We have thus shown that 
\[  \Hom_{\U(V)}( \chi_1 \times \chi_2, \omega_{V, \psi,\mu})  = \Hom_{\U(V)}( \pi_X, \omega_{V, \psi,\mu}).  \]
for $[X]$ the unique closed $\U(V)$-orbit on $\GL(V) / P_{a,b}$, and a necessary condition for the nonvanishing of this Hom space is
\[  b = a+1 \quad \text{and} \quad \chi_1 = \chi_2^\sigma \cdot \mu. \]
In other words, we have proved the first assertion of the Proposition. 
\vskip 5pt

For the converse, since $[X]$ is the closed orbit of $\U(V)$ on
 $\GL(V) / P_{a,b}$, 
we have seen that when the above conditions hold,  one has
\[  \Hom_{\U(V)}( \chi_1 \times \chi_2, \omega_{V, \psi,\mu})  =   \Hom_{\U(V)}( \pi_X, \omega_{V, \psi,\mu}) \cong \Hom_{\U(V_1)}( 1, \omega_{V_1, \psi, \mu}).\]
One is thus reduced to the $n=1$ case of Conjecture \ref{conj-local} which is known.

The proof of the Proposition is now complete.
\end{proof}
\vskip 5pt

\begin{remark}
  The proof above also  proves that for the degenerate principal series representation
  $\pi= \chi_1\times \chi_2$ of $\GL_n(E)$,
 with $ b = a+1$ and $\chi_1 = \chi_2^\sigma \cdot \mu$,
  \[  \Ext^i_{\U(V)}[  \pi|_{\GL_n(E)}, \omega_{V,\psi,\mu}]   \cong
  \sum_{i=j+k} \Ext^j_{\U(V_a)}( \chi_1,  \omega_{V_a, \psi,\mu}) \otimes \Ext^k_{\U(V_b)}(\chi_2, \omega_{V_b, \psi,\mu})  \]
\end{remark}

\section{\bf When $E \ne K$: Local Case}  \label{S:local-general}
In this section, we consider  the general twisted variant of the GGP problem, where $E \ne K$ are two distinct quadratic extensions of a local field $F$, In particular,  $F$ is necessarily nonarchimedean and we fix a nontrivial additive character $\psi$ of $F$.   This case is considerably more intricate and like the GGP problem, we will need to make use of  the local Langlands correspondence for unitary groups to formulate our conjectural answers. 
\vskip 5pt

\subsection{\bf Biquadratic extension}
Let  $L = E \otimes_F K$, so that $L$ is a biquadratic extension of $F$.  We thus have the picture:
 \begin{equation*} 
\begin{gathered} 
\xymatrix{ & L=E \otimes K  \ar@{-}[ld]_{\sigma} \ar@{-}[rd]^{\tau}& \\ 
K \ar@{-}[rd]_{\tau} & & E \ar@{-}[ld]^{\sigma} \\   
& F &   } 
\end{gathered} 
\end{equation*} 
In particular, we have set:
\[  {\rm Gal}(E/F)\cong {\rm Gal}(L/K)  = \langle \sigma \rangle \quad \text{and} \quad {\rm Gal}(K/F) \cong {\rm Gal}(L/E) = \langle \tau \rangle. \]
 Observe that the biquadratic field $L$ contains a third quadratic subfield $E'$ which is the fixed field of $\sigma \cdot \tau$. This field $E'$ will play a role later on.
    
\vskip 5pt

\subsection{\bf Skew-Hermitian spaces.}  \label{SS:skew}
We consider the two isomorphism classes of skew-Hermitian spaces $V$ and $V'$ over $E$ of dimension $n$, and make the following
observation:
\vskip 5pt

\begin{lemma} \label{quasi-split}
The two skew-Hermitian spaces $V_K = V \otimes_F K$ and $V'_K = V' \otimes_F K$  are isomorphic over $L$. When $n$ is even, $V_K \cong V'_K$ is the maximally split skew-Hermitian space; whereas when $n $ is odd, $V_K \cong V'_K$ is characterized as the unique skew-Hermitian space whose determinant can be represented by elements of $E^{\times}_0$. In either case, $\U(V_K) \cong \U(V'_K)$ is a quasi-split group.
\end{lemma}

\vskip 5pt

\begin{proof}
It suffices to show that $\det V$ and $\det V'$ belong to the same $N_{L/K}(L^{\times})$-coset,  when viewed as elements of $K^{\times}$. Since $\det V$ and $\det V'$ belong to the same $F^{\times}$-coset, it suffices to observe that $F^{\times} \subset  N_{L/K}(L^{\times})$. Indeed, since $L$ is a biquadratic extension of $F$, $\omega_{L/K} = \omega_{E/F} \circ N_{K/F}$. Hence
\[  \omega_{L/K}(F^{\times}) = \omega_{E/F} \left( N_{K/F}(F^{\times}) \right) = \omega_{E/F}(F^{\times 2}) =1. \qedhere \]
\end{proof}
\vskip 10pt

\noindent In view of the lemma, we may regard $\U(V)$ and $\U(V')$ as subgroups of a fixed $\U(V_K) = \U(V'_K)$.

\vskip 5pt
\subsection{\bf Local Langlands correspondence}  \label{SS:LLC}
 Now we recall  the local Langlands correspondence for $\U(V_K)$. 
An L-parameter for $\U(V_K)$ is a conjugate-dual $n$-dimensional semisimple representation $M$ of the Weil-Deligne group $WD_L = W_L \times \SL_2(\C)$ of sign $(-1)^{n-1}$. We have studied such conjugate-dual representations in some detail in \cite{GGP1} and described their associated component groups $A_M$. More precisely, we may write
\[  M = \oplus_{i\in I}  V_i \otimes M_i \oplus P  \oplus {}^\sigma P^\vee  \]
with $M_i$ distinct conjugate-dual representations of sign $(-1)^{n-1}$ , $V_i$ its multiplicity space and $P$ contains all the irreducible summands which are either non-conjugate-dual or conjugate-dual of sign $(-1)^n$, with  ${}^\sigma P^\vee$ its conjugate-dual. As we discussed in \cite[\S 4]{GGP1}, the centralizer group of the L-parameter is of the form
\[   C_M = \prod_{i \in I} \O(V_i)  \times (\text{a connected reductive group}). \]
Hence the component group of $C_M$ is an elementary abelian 2-group
\[  A_M = \prod_{i \in I} \Z/2\Z \cdot a_i, \] 
  equipped with a canonical basis indexed by $I$. The element $-1_M$ gives rise to the element
  \[  \sum_{i \in I} \dim(V_i) \cdot a_i  \in A_M, \]
 which generates  a subgroup of order $\leq 2$ in $A_M$. Now
the local Langlands correspondence for $\U(V_K)$ gives a partition 
\[ {\rm Irr}(\U(V_K))   = \bigsqcup_M  \Pi_M,  \]
of ${\rm Irr}(\U(V_K))$ into the disjoint union of finite subsets, the L-packets, with the sum running over L-parameters of $\U(V_K)$.  Moreover since we are at the moment
concerned only with the quasi-split group $\U(V_K)$, for each parameter $M$ of $\U(V_K)$, one has a bijection
\[ J:  \Pi_M   \longleftrightarrow  {\rm Irr}(A_M / \langle -1_M \rangle). \]
Here the bijection $J$ is canonical when $n$ is odd and depends on the choice of an equivalence class of Whittaker datum for $\U(V_K)$ when $n$ is even. In that case, we have seen in \cite{GGP1} that the equivalence classes of Whittaker data are parameterized by additive characters of $K$ modulo the translation action of $N_{L/K}(L^\times)$. We shall use the Whittaker datum associated to $\psi_K = \psi \circ {\rm Tr}_{K/F}$. 
\vskip 10pt

Recall that an L-parameter $M$ is generic if the adjoint L-factor $L(s, M, {\rm Ad})$ is holomorphic at $s =1$. In that case, the L-packet $\Pi_M$ contains a unique  representation which is generic with respect to the Whittaker datum associated to $\psi \circ {\rm Tr}_{L/K}$. This representation corresponds to the trivial character of $A_M$ under the bijection $J$.
\vskip 10pt

\subsection{\bf Asai factors}  \label{SS:Asai}
We recall from \cite{GGP1} the notion of Asai L-factors and $\epsilon$-factors associated to a representation $M$ of $WD_L$ relative to the quadratic extension $L/E$. 
If  $\tau$ denotes the nontrivial element of $\Aut(L/E) \cong \Aut(K/F)$, the representation $M \otimes M^{\tau}$ is $\tau$-invariant and hence we have a decomposition 
\[  {\rm Ind}_{WD_L}^{WD_E} (M \otimes M^{\tau})  = {\rm As}_{L/E}^+(M) \oplus \As_{L/E}^-(M)   \]
of $WD_E$-modules, with ${\rm As}_{L/E}^{\pm}(M) \cong M \otimes M^{\tau}$ as $WD_L$-modules. On ${\rm As}_{L/E}^+(M)$, an element $s \in W_E \setminus W_L$ acts by $v \otimes w \mapsto w \otimes s^2 \cdot  v$, whereas on ${\rm As}_{L/E}^-(M)$, this action is twisted by the nontrivial character of $W_E/W_L$ (see \cite[Pg. 26-27]{GGP1}), thus 
${\rm As}_{L/E}^-(M)=  \As_{L/E}^+(M) \cdot \omega_{L/E}$.
\vskip 5pt

We record here some useful properties of the functor ${\rm As}^{\pm}_{L/E}$. Later we will deal
exclusively with ${\rm As}^{+}_{L/E}$, dropping the sign $+$. 
\vskip 5pt

\begin{lemma}  \label{L:Asai}
One has:
\vskip 5pt

\begin{itemize}
 \item[(a)]  If $M = \oplus_i M_i$, then
\[  {\rm As}_{L/E}^{\epsilon}(M) = \bigoplus_i   {\rm As}_{L/E}^{\epsilon}(M_i) \oplus    \bigoplus_{i < j}  {\rm Ind}_L^E (M_i \otimes M_j^{\tau}).\]
\vskip 5pt

\item[(b)] $\As^{\epsilon}_{L/E}(M)^{\vee} \cong \As^{\epsilon}_{L/E}(M^{\vee})$, where $M^{\vee}$ denotes the dual of $M$. 
\vskip 5pt

\item[(c)]
  $\As^{\epsilon}_{L/E}(M_1\otimes M_2) \cong \As^{\epsilon}_{L/E}(M_1) \otimes \As^{\epsilon}_{L/E}(M_2)$. 
\vskip 5pt

\item[(d)] If $\dim M =1$, in which case $M$ is treated  as a character of $WD_L^{ab} = L^{\times}$,
  $\As^+_{L/E}(M)$ is the restriction of $M$ from $L^\times$ to $E^{\times}$.
\vskip 5pt

\item[(e)] As a character of $WD^{ab}_E \cong E^{\times}$, 
\[  \det(\As^+_{L/E}(M)) = \As^+(\det(M))^n  \cdot \omega_{L/E}^{n(n-1)/2} = \det(M)|^n_{E^{\times}}\cdot \omega_{L/E}^{n(n-1)/2}   \]
where $n = \dim M$. 

\vskip 5pt
\item[(f)] If $M$ is an L-parameter of $\U(V_K)$ and hence is conjugate-dual (with respect to $L/K$) of sign $(-1)^{n-1}$, then  ${\rm As}^{\pm}_{L/E}(M)$ is necessarily conjugate-orthogonal relative to $E/F$.
\end{itemize}
\end{lemma}
\vskip 5pt

\vskip 10pt

\subsection{\bf Conjectures}  \label{SS:conj-local-general}
We come now to the restriction problem to be studied.    For each of the two skew-Hermitian spaces $V$ over $E$, we have the Weil representation $\omega_{V, \psi,\mu}$, where $\mu$ is a conjugate-symplectic character of $E^{\times}$. Then we are interested in determining
\[ m_V(\pi, \mu) := \dim \Hom_{\U(V)} (\pi, \omega_{V,\psi,\mu})  \quad \text{  for $\pi \in {\rm Irr}(\U(V_K))$.} \]
Here is our main local conjecture for arbitrary separable quadratic extensions $E,K$ of $F$, subsuming the earlier Conjecture
\ref{conj-local} (for the case $E=K$):
\vskip 5pt

\begin{conj}  \label{conj-local-general}
\begin{itemize}
\item[(i)] For each $\pi \in {\rm Irr}(\U(V_K))$, 
 \[  m_V(\pi, \mu) = \dim \Hom_{\U(V)}(\pi, \omega_{V,\psi, \mu})  \leq 1.  \]
\vskip 5pt

\item[(ii)] Let $M$ be a generic L-parameter of $\U(V_K)$ with associated L-packet $\Pi_M \subset {\rm Irr}(\U(V_K))$. Then
\[  \sum_V  \sum_{\pi \in \Pi_M} m_V(\pi, \mu)  =1 \]
where the first sum runs over the two skew-Hermitian spaces over $E$ of dimension $n$ and the second runs over the L-packet $\Pi_M$.  
\vskip 5pt

\item[(iii)]  The unique $V_0$ which has nonzero contribution to the sum in (ii) is characterized by
\[  \mu(\det(V_0)) = \epsilon( 1/2,   {\rm As}_{L/E}(M) \otimes \mu^{-1}, \psi_E) \cdot \det ({\rm As}_{L/E}(M))(e) \cdot  \omega_{K/F}(e^2)^{n(n-1)/2}   \]
where $e$ is any nonzero trace 0 element of $E$, so that $E = F(e)$.
\vskip 5pt

\item[(iv)] The unique $\pi \in \Pi_M$ which has nonzero contribution to the sum in (ii) corresponds via the bijection $J$ to the character of the local component group $A_M = \prod_{i \in I} \Z/2\Z \cdot a_i$  given by:
  \begin{eqnarray*}  \chi(a_i) & = &
    \epsilon( 1/2, \Ind_L^E( M_i^{\tau} \otimes (M/M_i)) \cdot \mu^{-1}, \psi_{E,e}) \\
&  = & \epsilon(1/2, [{\rm As}(M_i) + {\rm As}(M) + {\rm As}(M/M_i)] \cdot   \mu^{-1}, \psi_{E,e}), \end{eqnarray*}
\vskip 5pt
\noindent  where $\psi_{E,e}$ is the additive character of $E/F$ defined by $\psi_{E,e}(x) = \psi(Tr(ex))$. 
 \end{itemize}
\end{conj}

We make a few remarks on the above conjecture:
\vskip 5pt

\begin{itemize}
\item[(a)]  In (iii), the proposed expression for $\mu(\det(V_0))$ is independent of the choice of the trace 0 element $e$. Moreover, using property (d) in \S \ref{SS:Asai} and the fact that $\omega_{L/E}(e) = \omega_{K/F}(N_{E/F}(e)) = \omega_{K/F}(-e^2)$,   the equation in (iii) can be explicated as:
\[ \mu(\det(V_0)) = \epsilon( 1/2,   {\rm As}_{L/E}(M) \otimes \mu^{-1}, \psi_E)  \cdot  \det(M) (e)^n \cdot \omega_{K/F}(-1)^{n(n-1)/2}. \]
Though this may be more compact, our original expression has the advantage that it can be specialized to all possible situations for the pair $(E,K)$, as we shall explain below.
\vskip 5pt

\item[(b)]  In (iii), observe that if $E = F(e)$ and $K = F(k)$ with $k \in K^{\times}$ a trace zero element, then 
\[  \omega_{K/F}(e^2)  = (k^2, e^2)_F.  \]
In particular, we see that this term only appears when $K$ and $E$ are both fields (as we are assuming in the conjecture). 

\item[(c)]  The distinguished character $\chi$ in (iv) is indeed trivial on the image of $-1_M$ in $A_M$. Moreover, it is independent of the choice of the trace $0$ element $e$. This follows from the fact that $({\rm As}(M_i) + {\rm As}(M) + {\rm As}(M/M_i)) \cdot \mu^{-1}$ is an even-dimensional conjugate-symplectic representation of $WD_E$ and hence its determinant is conjugate-orthogonal.  
\vskip 5pt


\vskip 5pt

\item[(d)]  For the skew-Hermitian case considered in  \cite{GGP1}, we had defined a distinguished character $\chi$ of the local component group which gives the unique representation in the L-packet with nonzero branching multiplicity. This distinguished character automatically picks out the skew-Hermitian space $V_0$ over $E$ which supports the nonzero multiplicity, so that (iii) is a consequence of (iv) in the original GGP setting. 
In the case here, the distinguished character $\chi$ in (iv) gives a representation of $\U(V_K)$, but does not specify the $E$-space $V_0$. This is why the condition (iii) is needed.  
\end{itemize}

\vskip 10pt

\subsection{\bf Specializations}  \label{SS:special}
Though we are assuming that $E \ne K$ are distinct quadratic fields  in this section, the formulas in Conjecture \ref{conj-local-general}(iii) and (iv) make sense for general $(E, K)$. For this, we need to explain how the L-parameter of $\Pi \in \Irr(\U(V_K))$ gives rise to a representation of $WD_L$ and how to interpret  the Asai lift relative to $L/E$  in the various situations. 

 \vskip 5pt

\begin{itemize}
\item $E = K$ is a field: this is the  setting of \S \ref{S:KisE}. In this case, $L = E \otimes K$ is isomorphic to $E \times E = K \times K$.  Note however that  the embeddings of $K$ and $E$ into $L$ are different. The embedding of $K$ into $L$ is the diagonal embedding $x \mapsto (x,x)$, whereas that of $E$ into $L$ is $x \mapsto (x, x^{\sigma})$, where $\Aut(E/F) = \langle \sigma \rangle$. We interpret the Weil-Deligne group of $L$ as $WD_L = WD_K \times WD_K = WD_E \times WD_E$.
\vskip 5pt

Now given an irreducible representation $\Pi$ of $\U(V_K) = \GL(V)$, its L-parameter $M$ is an $n$-dimensional  representation of  $W_K = W_E$ and this gives rise to the pair  $(M, M^{\vee})$ which we interpret as a representation of $WD_L$.  Now the nontrivial element of $\Aut(L/E)$ acts on $L = E \times E$ via 
$(x,y) \mapsto (y^{\sigma}, x^{\sigma})$. Thus its induced action on the representations of $WD_L$ is $(M , M^{\vee}) \mapsto ({}^\sigma \!M ^\vee, M^{\sigma})$ (the switch, followed by the action of $\sigma$). We interpret the Asai lift as the tensor product representation $M \otimes {}^\sigma \!M ^\vee$ of $WD_E$.  
With these interpretations, the formula in Conjecture \ref{conj-local-general}(iii) specializes to
that in Conjecture \ref{conj-local}(iii), in view of the remark (a) in \S \ref{SS:conj-local-general} above.
\vskip 5pt
The issue addressed by Conjecture \ref{conj-local-general}(iv) is not relevant in this case since the L-packet of $\U(V_K) = \GL(V)$ is a singleton. However, we note that with the above interpretations, the RHS of the formula there is equal to $1$.

\vskip 10pt

\vskip 5pt
\item $E$ is a field and $K = F \times F$, so that $L = K \otimes E = E \times E$ and $WD_L = WD_E \times WD_E$: this is the original GGP situation. Then $\U(V_K) \cong \U(V) \times \U(V)$ and an irreducible representation of $\U(V_K)$ is of the form $\pi_1 \boxtimes \pi_2$ with $\pi_i \in \Irr(\U(V))$. The L-parameters of $\pi_1$ and $\pi_2$ are conjugate-dual representations $M_1$ and $M_2$ of $WD_E$ of sign $(-1)^{n-1}$, giving  a representation $(M_1, M_2)$ of $WD_L$. Now since $E$ is embedded diagonally in $L = E \times E$, the nontrivial automorphism of $L/E$ is the switch of the two factors of $E$ in $L$. The Asai lift of $M_1 \boxtimes M_2$ from $L$ to $E$ is interpreted as the internal tensor product $M_1 \otimes M_2$.  With these interpretations, the formula in Conjecture \ref{conj-local-general} (iii) reads:
\[  \mu(\det(V_0)) = \epsilon(1/2, M_1 \otimes M_2 \otimes \mu^{-1}, \psi_E) \cdot \det(M_1 \otimes M_2)(e). \]
We leave it to the reader to verify that this reduces to the relevant conjecture in \cite{GGP1}. 
\vskip 5pt

\item Compared to the other cases, a peculiarity of the original GGP situation is that $\U(V_K)$ and $\U(V'_K)$ are not isomorphic when $V$ and $V'$ are the two distinct skew-Hermitian spaces over $E$. Hence,  one needs to choose and fix  a quasi-split $\U(V_K)$ to formulate the LLC, before one can consider Conjecture \ref{conj-local-general}(iv). When $\dim V$ is even, this choice is unique, but when $\dim V$ is odd, this amounts to choosing a trace zero element $e_0 \in E^{\times}$ (the determinant of the distinguished $V$).  Moreover, it is no longer the case that the character given in Conjecture \ref{conj-local-general}(iv) is independent of $e$ when $\dim V$ is odd (though it is still the case when $\dim V$ is even). 
Thus, in Conjecture \ref{conj-local-general}(iv), one needs to use the distinguished $e_0$ in the definition of the character $\chi$ when $\dim V$ is odd. With this caveat, we leave it to the reader to verify that the formula for the character $\chi$  in (iv) specializes to the one we had in \cite{GGP1}.
\vskip 5pt

 \item $E = F \times F$ and $K$ is a field, so that $L = K \times K$. Here, $\U(V) = \GL(V)$ and $\U(V_K) = \GL(V_K)$. Given an irreducible generic representation $\Pi$ of $\GL(V_K)$, 
and a conjugate-dual character $\mu = (\nu, \nu^{-1})$ of $E^{\times} / F^{\times} = (F^{\times} \times F^{\times})/F^{\times}$, the  multiplicity $\dim \Hom_{\GL(V)}(\Pi, \omega_{V,\psi,\mu})$ should be always nonzero.  So we expect the proposed identity in Conjecture \ref{conj-local-general}(iii) to always hold, after appropriate interpretations.
\vskip 5pt

Now the L-parameter of $\Pi$
 is an $n$-dimensional representation $M$ of $WD_K$. This gives rise to the pair $(M, M^{\vee})$ which we regard as a representation of $WD_L = WD_K \times WD_K$. 
The nontrivial automorphism of $L/E$ is the component wise action of $\tau \in \Aut(K/F)$  on $L = K \times K$, so the Asai lift from $L$ to $E$ is the pair $(\As_{K/F}(M), \As_{K/F}(M^{\vee}))$, regarded as a representation of $WD_E = WD_F \times WD_F$. In this case, 
\begin{align}
 &\epsilon(1/2, \As_{L/E}(M, M^{\vee}) \cdot\mu^{-1}, \psi_E)  \notag \\
 = &\epsilon(1/2, \As_{K/F}(M)\cdot \nu^{-1}, \psi) \cdot \epsilon(1/2, \As_{K/F}(M)^{\vee} \cdot \nu, \psi). \notag \\
 = &\det(\As_{K/F}(M))(-1) \cdot  \nu(-1)^n  \notag
 \end{align}
 Moreover, an element $ e \in E = F \times F$ of trace $0$ is of the form $(a, -a)$ for $a \in F^{\times}$. Hence,
 \begin{align}
   \det(\As_{L/E}(M,M^{\vee})) (e) &= \det(\As_{K/F}(M))(a) \cdot \det(\As_{K/F}(M))(-a)^{-1}  \notag \\
   &= \det(\As_{K/F}(M))(-1) \notag  
     \end{align}
   and
   \[  \omega_{K/F}(e^2) = \omega_{K/F}(a^2) =1. \]   
  Thus, the RHS of the formula in (iii) is $\nu(-1)^n$, which is equal to the LHS. 
\end{itemize}
There is also the case where $E = K = F \times F$, which we will leave to the reader.  The main reason for formulating Conjecture \ref{conj-local-general} in  a uniform way which allows for specialization to the various cases is that in the global setting to be considered in Section \ref{S:global-general},  any one of these local scenarios will arise.

\section{\bf Low Rank Evidences: $E \ne K$}  \label{S:evi}

Just as for Conjecture \ref{conj-local}, we provide here some evidences for Conjecture \ref{conj-local-general} in low rank cases. In particular, we shall show:
\vskip 5pt

\begin{thm}  \label{T:low12}
Conjecture \ref{conj-local-general} holds when $dim V \leq 2$.
\end{thm}
The rest of this section is devoted to the verification of the theorem. 
 \vskip 10pt

\subsection{\bf Rank 1 case}
Assume first that $V$ is a skew-Hermitian space of dimension 1, so that $\U(V) = E_1\subset \U(V_K) = L_1$, where $L_1$ denotes the subgroup of elements $x \in L^{\times}$ with $N_{L/K}(x) =1$. 
Given a character $\chi$ of $L_1$,  choose an extension $\tilde{\chi}$ of $\chi$ to $L^{\times}$.  Then  
the L-parameter of $\chi$ is the  1-dimensional conjugate-orthogonal representation $M =   \tilde{\chi}/ \tilde{\chi}^{\sigma}$ of $W_L$. 
By the theorem of Moen and Rogawski, we know that
\[  \Hom_{E_1}(\chi,  \omega_{V, \psi, \mu}) \ne 0 \Longleftrightarrow \epsilon(1/2, ( \tilde{\chi}/ \tilde{\chi}^{\sigma})|_{E^{\times}}  \otimes \mu^{-1}, \psi_E) \cdot \chi(-1) = \mu(\det(V)). \]
The local root number above can be written as
\[  \epsilon(1/2, {\rm As}_{L/E}(M) \cdot \mu^{-1}, \psi_E),\]
whereas
\[  \det({\rm As}(M))(e) =  \tilde{\chi}(e) / \tilde{\chi}(e^{\sigma}) = \chi(-1). \]
This shows Conjecture \ref{conj-local-general} when $n= \dim V =1$.
\vskip 10pt

\subsection{\bf Rank 2 case}
Suppose now that $\dim V =2$. In this case, we need to verify the independent statements  (iii) and (iv)  of Conjecture \ref{conj-local-general}.
As we have noted before, $V = V_B$ is associated with a quaternion $F$-algebra $B$, with 
\[  \GU(V_B)\cong (B^{\times} \times E^{\times} ) / \Delta F^{\times}. \]
The embedding $\GU(V_B) \hookrightarrow \GU(V_{B,K})$ is the natural embedding
\[  (B^{\times} \times E^{\times}) / \Delta F^{\times} \hookrightarrow ((B \otimes_F K)^{\times} \times L^{\times} ) / \Delta K^{\times}, \]
with $B \otimes_F K \cong M_2(K)$. 

\vskip 5pt

A generic L-packet of $\U(V_K)$ is thus given by an irreducible representation
\[  \Pi \boxtimes  \chi \quad \text{  of $\GL_2(K) \times L^{\times}$,} \]
with $\omega_{\Pi}  \cdot \chi|_{K^{\times}}  =1$.  If $P$ is the L-parameter of $\Pi$, then the L-parameter of the corresponding L-packet of $\U(V_K)$ is the conjugate-symplectic (relative to $L/K$) representation
\[  M =P|_{WD_L} \otimes \chi \]
of $WD_L$. On the other hand, the Weil representation $\omega_{\psi,\mu, B}[ \chi|_{E^{\times}}]$ of $\U(V_B)$ is an irreducible summand of the representation
\[  \Sigma_{B,N} \otimes \chi
\]
of $B^{\times} \times E^{\times}$ restricted to $(B^{\times})^+ \times E^{\times}$, where  as in \S \ref{SS:rk2}, $\Sigma_{B,N}$ has L-parameter 
\[  N = \Ind_E^F (\mu \cdot\chi|_{E^\times}^{-1} ). \]
The corresponding L-parameter of $\U(V_B)$ is the conjugate-symplectic (relative to $E/F$) representation 
\[  N|_{WD_E} \otimes  \chi|_{E^{\times}}. \]
 \vskip 5pt
 
 Now we consider the sum
 \begin{equation} \label{E:sum}  \sum_{\pi \in \Pi_M}  \dim \Hom_{\U(V_B)}( \pi, \omega_{\psi,\mu, B}). \end{equation}
Via the above identifications, one sees that this sum is simply
\[  \dim \Hom_{(B^{\times})^+} (\Pi,  \omega_{\psi,\mu, B}) = \dim \Hom_{B^{\times}}(  \Pi, \Sigma_{B,N}).\]
In other words, we are reduced to a twisted trilinear form problem as in \S \ref{SS:rk2}.
 Hence, by a result of the third author, cf. \cite{P1},  this dimension is at most $1$ and is nonzero if and only if
\begin{equation} \label{E:tri}
  \epsilon(1/2, {\rm As}_{K/F}(P) \otimes \Ind_E^F (\mu^{-1} \cdot \chi|_{E^{\times}}), \psi) \cdot \omega_{K/F}(-1)  = \mu(\det(V_B)). \end{equation}
 Now the local root number can be explicated as:
 \begin{align}
   \epsilon(1/2, {\rm As}_{K/F}(P) \otimes \Ind_E^F (\mu^{-1} \cdot  \chi|_{E^{\times}}), \psi)  &=
   \epsilon(1/2, \Ind_E^F ({\rm As}_{K/F}(P)|_{WD_E} \cdot   \chi|_{E^{\times}} \cdot \mu^{-1}),\psi), \notag \\
  &= \epsilon(1/2, \Ind_E^F ({\rm As}_{L/E}( P|_{WD_L} \cdot \chi) \cdot \mu^{-1}), \psi), \notag  \\
  &= \epsilon(1/2, {\rm As}_{L/E}(P|_{WD_L} \cdot \chi) \cdot \mu^{-1}, \psi_E), \notag \\
  &= \epsilon(1/2, {\rm As}_{L/E}(M) \cdot \mu^{-1},\psi_E). \notag 
  \end{align}
  In the above, we have used the facts that
  \[ {\rm As}_{K/F}(P) |_{WD_E} \cong {\rm As}_{L/E} \left( P|_{WD_L} \right) \]
  and
  \[  {\rm As}_{L/E}(P|_{WD_L}) \otimes \chi |_{E^{\times}}  = {\rm As}_{L/E}( P|_{WD_L} \otimes \chi) = {\rm As}_{L/E}(M). \]
On the other hand, with $n =2$, 
\[  \det({\rm As}_{L/E}(M))(e)^n \cdot \omega_{K/F}(e^2)^{n(n-1)/2} =  \det(M)(e)^2 \cdot \omega_{L/E}(e) \cdot \omega_{K/F}(e^2) = \omega_{K/F}(-1)  \]
since $\det(M)$ is conjugate-orthogonal and hence trivial on $e^2 \in F^{\times}$, and $\omega_{L/E}(e) = \omega_{K/F}(-e^2)$.
Hence, the equality (\ref{E:tri}) is precisely the statement of  Conjecture \ref{conj-local-general}(iii).
\vskip 5pt

We now come to Conjecture \ref{conj-local-general}(iv).   Continuing with the analysis above, let us fix $V = V_B$ such that (\ref{E:tri}) holds, so that the sum in (\ref{E:sum}) is equal to $1$, and we need to determine which element in the L-packet $\Pi_M$ has nonzero contribution. 
Now the members of the L-packet are given by the restriction of $\Pi$ to $\GL_2(K)^+$. If this restriction is irreducible, then we leave it to the readers to convince themselves   that Conjecture \ref{conj-local-general}(iv) holds. Let us examine the more intricate case when this restriction is the sum of two irreducible summands, i.e. when $\Pi$ is dihedral with respect to $L/K$.
Thus we see that the problem at hand  is a refined version of the twisted trilinear form problem, relative to the embedding $\GL_2(F) \subset \GL_2(K)^+$.
\vskip 5pt

Since $\Pi$ is dihedral with respect to $L/K$,  $P|_{WD_L}$ is reducible and so is $M = P|_{WD_L} \cdot \chi$. 
To understand the L-packet, we shall return to the setting of unitary groups, as $\Pi_M$ can be constructed via theta lifting from rank 1 skew-Hermitian spaces.
\vskip 10pt

\subsection{\bf Unitary theta lifts}
Let $M = M_1 + M_2$ be an L-parameter of $\U(V_B)(K)$
with $M_1$ and $M_2$ conjugate-symplectic characters of $W_L$. The L-packet $\Pi_M$ has 2 representations of 
$\U(V_B)(K)$, which we may denote 
by $\pi^{+}$ and $\pi^{-}$ (these are $\pi^{++}, \pi^{--}$ of \cite{GGP2}),
so that $\pi^{+}$ is generic with respect to the Whittaker datum determined by $\psi_K = \psi \circ {\rm Tr}_{K/F}$. Note that by Lemma \ref{quasi-split},
$\U(V_B)(K)$ is always the quasi-split unitary group in two variables, so the representations
on the anisotropic form of $\U(V_B)(K)$ does not arise in our considerations.  
We shall explain how these representations $\pi^{\pm}$
can be constructed as theta lifts from $\U_1$. 
\vskip 5pt

Let $W^{\pm}$ be the two rank 1 Hermitian spaces over $L$ with $\omega_{L/K}({\rm disc}(W^{\pm})) = \pm 1$. In particular, the Hermitian form on $W^+$ is  $(x,y)\mapsto x \cdot y^{\tau}$, with ${\rm Gal}(L/K) = \langle \tau \rangle$.  Then for $\epsilon = \pm 1$, $\U(W^{\epsilon})   \times  \U(V_K) $
is a reductive dual pair where $V_K = V_B \otimes_F K$.
Likewise, we may consider the rank 2 Hermitian space 
\[   W^{\epsilon}_E := {\rm Res}_{L/E}(W^{\epsilon}) \quad \text{with Hermitian form ${\rm Tr}_{L/E}( -, -)_{W^{\epsilon}}$.} \]
This rank 2 Hermitian space over $E$ has discriminant
\[  {\rm disc}(W^{\epsilon}_E) =   N_{K/F}( k \cdot {\rm disc}(W^{\epsilon})) \in F^{\times}/ N(E^{\times}) \]
where $k \in K^{\times}$ is any trace 0 element; we leave the verification of this to the reader. 
Then $\U(W^\epsilon_E) \times \U(V)$ is a reductive dual pair, and we have the seesaw diagram:
\[
 \xymatrix{
  \U(W^{\epsilon}_E)   \ar@{-}[dr] \ar@{-}[d] & \U(V_K)  \ar@{-}[d] \\
  \U(W^{\epsilon})   \ar@{-}[ur] & \U(V).}
\]
To consider the theta correspondences for these two dual pairs, we need to select splitting characters in each case, and to obtain a seesaw identity from the seesaw diagram, we need to select these two sets of splitting characters compatibly. With the goal of obtaining the L-packet $\Pi_M$ of $\U(V_K)$ as theta lifts from $\U(W^{\pm})$, we shall select these splitting characters as follows:
\vskip 5pt

\begin{itemize}
\item recall that $M_1$ is a conjugate symplectic character of $L^{\times}$ relative to $L/K$. Then its restriction $M_1|_{E^{\times}}$ is a conjugate-orthogonal character of $E^{\times}$ relative to $E/F$ (because $F^{\times} \subset N_{L/K}(L^{\times})$).

\item For the equal rank dual pair $\U(V) \times \U(W_E)$ over $F$, we use the pair of splitting characters $(M_1|_{E^{\times}}, M_1|_{E^{\times}})$, and the additive character $\psi$ of $F$.
\vskip 5pt

\item For the almost equal rank dual pair $\U(V_K) \times \U(W^{\epsilon})$ over $K$, we use the pair $(M_1, M_1 \circ N_{L/E}) = (M_1, M_1 \cdot M_1^{\tau})$ and the character $\psi_K$ of $K$. 
\end{itemize}
With these splitting characters and additive characters fixed, one can consider the associated theta correspondences for the two dual pairs. Moreover, one has the seesaw identity associated to the above seesaw diagram. For this, one needs to specify the irreducible representations one starts with on $\U(W^{\epsilon})$ and $\U(V)$. 

\vskip 5pt

\begin{itemize}
\item[(i)]   For the dual pair $\U(W^{\epsilon}) \times \U(V_K)$, if one starts with the character $\chi_{M_1^{\tau}M_2}$ of $\U(W^{\epsilon})$ with L-parameter $M_1^{\tau} \cdot M_2$, then its theta lift to $\U(V_K)$ has L-parameter $M =M_1 +M_2$. As $\epsilon$ varies over $\pm$, the two representations so obtained are the elements $\pi^{\epsilon}$ of the L-packet $\Pi_M$.
 \vskip 5pt
 
 \item[(ii)] For the dual pair $\U(V) \times \U(W_E)$, we start with the Weil representation $\omega_{\psi, \mu, V}[\chi_{M_1M_2}]$ of $\U(V)$ whose central character is the character $\chi_{M_1M_2}$ of $E_1$ with L-parameter  $M_1M_2$  and whose L-parameter is $N = \mu + \mu^{-1} M_1M_2$. Its theta lift to $\U(W_E)$, if nonzero,  has the same L-parameter.  
 \end{itemize}
From the seesaw identity, we see that
\[  \Hom_{\U(V)}( \pi^{\epsilon}, \omega_{\psi,\mu, V})
\cong \Hom_{\U(W^{\epsilon})}(\Theta(\omega_{\psi, \mu, V}[\chi_{M_1M_2}]), \chi_{M_1^{\tau}M_2}), \]
so that
\[ \Hom_{\U(V)}(\pi^{\epsilon}, \omega_{\psi,\mu, V}) \ne 0 \Longrightarrow  \Theta(\omega_{\psi, \mu, V}[\chi_{M_1M_2}]) \ne 0. \]
By the theta dichotomy theorem  \cite{HKS, GI}, the latter holds if and only if
\[  \omega_{E/F}(- k^2) \cdot \epsilon = \omega_{E/F}({\rm disc}(W_E^{\epsilon})) = \epsilon(1/2, N \cdot M_1|_{E^{\times}}^{-1}, \psi_{E,e}) \cdot \mu(\det(V)).   \]
The local root number on the RHS is equal to
\begin{align}
 &\epsilon(1/2,   {\rm As}_{L/E}(M_1)^{-1} \cdot \mu, \psi_{E,e}) \cdot \epsilon(1/2, {\rm As}_{L/E}(M_2) \cdot \mu^{-1}, \psi_{E,e}) \notag \\ 
= &\epsilon(1/2, {\rm As}_{L/E} (M_1) \cdot \mu^{-1}, \psi_{E, e}) \cdot \epsilon(1/2, {\rm As}_{L/E}(M_2) \cdot \mu^{-1}, \psi_{E,e}) \cdot \omega_{E/F}(-1) \notag
\end{align} 
On the other hand, by conjecture \ref{conj-local-general}(iii), which we have demonstrated above, we know that
\begin{align}
 \mu(\det(V)) &=  \epsilon(1/2, {\rm As}_{L/E}(M) \cdot \mu^{-1},\psi_E) \cdot \omega_{K/F}(-1) \notag \\
 &=  \epsilon(1/2, {\rm As}_{L/E}(M) \cdot \mu^{-1},\psi_{E,e}) \cdot \omega_{K/F}(e^2). \notag
 \end{align}
Assembling these together, we see that
\begin{eqnarray*} \epsilon & =&  
  \epsilon(1/2, [{\rm As}_{L/E} M_1 + {\rm As}_{L/E}(M_2) +  {\rm As}_{L/E}(M)]
  \cdot \mu^{-1}, \psi_{E, e})\cdot  \omega_{E/F}(k^2) \cdot \omega_{K/F}(e^2), \\
  &= &
\epsilon(1/2, [{\rm As}_{L/E} M_1 + {\rm As}_{L/E}(M_2) +  {\rm As}_{L/E}(M)]
\cdot \mu^{-1}, \psi_{E, e}),
\end{eqnarray*}
as predicted by Conjecture \ref{conj-local-general}(iv),  
 where for the second equality, we have used:
\[  \omega_{K/F}(e^2) = (k^2, e^2)_F = \omega_{E/F}(k^2). \]

Note that by Lemma \ref{L:Asai}(a),  the last epsilon factor  can be simplified as:
\[  \epsilon = \epsilon(1/2, \Ind_L^E (M_1^{\tau} \cdot M_2)\cdot \mu^{-1}, \psi_{E,e}).\]

We have thus completed the proof of Theorem \ref{T:low12}. For concreteness, we highlight the results obtained for the rank 2 case:
\begin{prop} \label{U2K}
Suppose we are given:
\begin{itemize}
\item a  quadratic extension $E/F$ of non-archimedean local fields; 
\item  a quaternion $F$-algebra $B$ with associated skew-Hermitian space $V_B$ of dimension 2 over $E$;
\item a quadratic field extension  $K \ne E$ with associated biquadratic field $L = E \otimes K$;
\item  an L-parameter $M = M_1 + M_2$  of $\U(V_B)(K)$, with $M_1$ and $M_2$ conjugate-symplectic characters of $W_L$, whose  L-packet $\Pi_M$ has 2 representations $\pi^{+}$ and $\pi^{-}$ of $\U(V_B)(K)$, so that $\pi^{+}$ is generic with respect to the Whittaker datum determined by $\psi_K = \psi \circ {\rm Tr}_{K/F}$.
\end{itemize}
Then one has:
\[ \Hom_{\U(V_B)}( \pi^{\epsilon}, \omega_{\psi,\mu, V_B}) \ne 0 \iff 
\begin{cases}   \mu(\det(V_B)) =
  \epsilon(1/2, {\rm As}_{L/E}(M) \cdot \mu^{-1},\psi_E) \cdot \omega_{K/F}(-1),   \\
  \epsilon  =  \epsilon(1/2, \Ind_L^E (M_1^{\tau} \cdot M_2)\cdot \mu^{-1}, \psi_{E,e}),\end{cases} \] 
  where $e \in E^{\times}_0$. 
  \end{prop}

\section{\bf Unitary Principal Series: $E \ne K$}  \label{S:UPS3}
 this section, we shall study the restriction problem for unitary principal series representations  and show the analog of Corollary \ref{C:ps} in the $E \ne K$ setting. 
Recall that we have the diagram of fields and Galois automorphisms as below.

\begin{equation*}
\begin{gathered} 
\xymatrix{ & L=E \otimes K  \ar@{-}[ld]_{\sigma} \ar@{-}[rd]^{\tau}& \\ 
K \ar@{-}[rd]_{\tau} & & E \ar@{-}[ld]^{\sigma} \\   
& F &   } 
\end{gathered} 
\end{equation*}
The biquadratic field $L$ contains a third quadratic subfield $E'$ which is the fixed field of $\sigma \cdot \tau$.
Let  $V$ be a skew-Hermitian space (relative to $E/F$) of dimension $n$ over $E$  and $V_K= V \otimes_F K$, the corresponding skew-Hermitian space (relative to $L/K$) 
over $L=KE$. We also let $\tau$ denote the Galois automorphism acting on $V_K$ and $\U(V_K)$ with fixed points $V$ and $\U(V)$ respectively.
\vskip 5pt

\subsection{\bf Mackey Theory.}
We shall consider the restriction to $\U(V)$ of a parabolically induced representation from a maximal parabolic subgroup of $\U(V_K)$. The following theorem is an analog of 
Theorem \ref{ps1}.
\vskip 5pt

\begin{thm} \label{ps}
Let $V_K$ be the $n$-dimensional skew-Hermitian space relative to $L/K$ which is the base change of any $n$-dimensional skew-Hermitian space relative to $E/F$.
 Let 
 \begin{itemize}
 \item  $P  = MN$ be a maximal parabolic subgroup of $\U(V_K)$ which is the stabilizer of an $a$-dimensional isotropic subspace of $V_K$,  with Levi factor
 \[ M \cong \GL_a(L) \times \U_{n-2a}(K). \]
\item 
 \[  \pi = \pi_1 \rtimes \pi_2 = \Ind_P^{\U(V_K)} (\pi_1 \otimes \pi_2) \]
  be  a tempered principal series representation of $\U(V_K)$, with $\pi_1 \in {\rm Irr}(\GL_a(L))$ and $\pi_2 \in {\rm Irr}(\U_{n-2a}(K))$.  
  \end{itemize}
For any skew-Hermitian $V$ relative to $E/F$ such that $V \otimes_F K \cong V_K$,  let $\omega_{V,\psi, \mu}$ be a Weil representation of $\U(V)$.

\vskip 5pt

Then for all $i \geq 0$,
  
\[   \sum_{V} \Ext^i_{\U(V)}[ \pi, \omega_{V,\psi}] \]
\[ \stackrel{(1)}{ =}    \sum_{i=j+k} \left(\sum_{V'_a} \Ext^j_{\U(V'_a)}[ \pi_1, \omega_{V'_a,\psi, \mu \circ N_{L/E}}] \right) \otimes
  \left(\sum_{V_{n-2a}} \Ext^k_{\U(V_{n-2a})}[ \pi_2, \omega_{V_{n-2a},\psi, \mu}] \right), \]
    where
    \begin{itemize}
    \item the sum over $V$ runs over the two skew-Hermitian spaces relative to $E/F$ of dimension $n$;
    \item     the sum over $V'_a$ runs over the two skew-Hermitian spaces relative to $L/E'$ of dimension $a$;
    \item the sum over $V_{n-2a}$ runs over the two skew-Hermitian spaces relative to $E/F$ of dimension $n-2a$.
   \item  $\omega_{V'_a,\psi, \mu \circ N_{L/E}}$ and  $\omega_{V_{n-2a},\psi, \mu}$ denote the corresponding Weil representations of $\U(V'_a)$ and $\U(V_{n-2a})$. 
   \end{itemize}
      In particular, for $i=0$,

\[  \sum_{V} \Hom_{\U(V)}[ \pi, \omega_{V,\psi, \mu}] \]
\[  \stackrel{(2)}{ =}  \left(\sum_{V'_a} \Hom_{\U(V'_a)}[ \pi_1, \omega_{V'_a,\psi, \mu \circ N_{L/E} }] \right)\otimes
  \left(\sum_{V_{n-2a}}  \Hom_{\U(V_{n-2a})}[ \pi_2, \omega_{V_{n-2a},\psi, \mu}] \right).\]
The isomorphisms in both the equations (1) and (2) come from the  open orbits.
More precisely,
if $[X]$ is a non-open orbit of $\U(V)$ on $\U(V_K)/P$, contributing (by the Mackey theory)
  a certain representation
  $\pi_X$ of $\U(V)$ as a subquotient of  $\pi$, then
  \[\Ext^i_{\U(V)}[ \pi_X, \omega_{V,\psi, \mu}] = 0,\]
  for all $i \geq 0$.
\end{thm}

\begin{proof}
  The proof of this theorem is almost identical to the corresponding theorem for the $E=K$ case, i.e., Theorem
  \ref{ps1}, so we will be  brief.
  It again depends on using the Mackey theory to calculate the  representation
  $\pi_X$ of $\U(V)$ as a subquotient of  $\pi$ supported on each orbit  $[X]$ of $\U(V)$ on the partial flag variety $\U(V_K)/P$.
Hence, we first  investigate the  orbits of $\U(V)$ on $\U(V_K)/P$, and their associated stabilizers in $\U(V)$.
\vskip 5pt

The partial flag variety $\U(V_K)/P$ parameterizes  $a$-dimensional isotropic $L$-subspaces $X$ of $V_K$. Since $\tau$ acts on $V_K$, we have an action $X \mapsto X^{\tau}$ of $\tau$ on $U(V_K)/P$.   For each isotropic $X$, let $P_X \subset \U(V_K)$ be the stabilizer of $X$ in $\U(V_K)$, so that $P_X = M_X N_X$ is a maximal parabolic subgroup with Levi factor 
  \[  M_X = \GL(X) \times \U(X^{\perp}/X). \]
Let   
\[  Q_X = \U(V) \cap P_X \]
 be the stabilizer of $X$ in $\U(V)$, with $N_{Q_X}$ its unipotent radical. Therefore $Q_X$ preserves the flag:
  \[ 0 \subset  X \cap X^\tau \subset  X           \subset    X^\perp \subset (  X \cap X^\tau)^\perp  \subset  V, \]
  and there is a natural map 
  \[  Q_X \rightarrow P_X \rightarrow M_X = \GL(X) \times \U(X^\perp/X).\]
   Hence, a representation $\pi_1 \boxtimes \pi_2$ of $M_X = \GL(X) \times \U(X^\perp/X)$
  gives rise by pull back to a representation of  $Q_X$,    which can then be induced to $\U(V)$  to obtain  the  representation
  $\pi_X$ of $\U(V)$ supported on the $\U(V)$-orbit of $X$.  
\vskip 5pt

After the above generalities, we now  consider  different cases according to the types of $X$.  
\vspace{4mm}

  \noindent {\bf Case 1:} $X \cap X^\tau \not = 0$.
  \vspace{4mm}
  
 The space $X \cap X^\tau$ is defined over $E$, so let $Y\subset V$ be such that  $Y \otimes_E L=  X \cap X^\tau$. The space $Y$ is
isotropic, and it is easy to see that  $Q_X$  has the following properties:

\begin{enumerate}

\item The unipotent radical $N_{Q_X}$ of $Q_X$ contains the  center $Z_Y$ of the
  unipotent radical $N_Y$ of the parabolic subgroup $P_Y \subset \U(V)$ stabilizing the isotropic subspace $Y\subset V$; observe that $Z_Y$ is the subgroup of $Q_X$ acting trivially on $Y^\perp$.
  \vskip 5pt
  
\item  The  natural   map 
  \[  Q_X \rightarrow \GL(X) \times \U(X^\perp/X)   \]
  has kernel  $Z_Y$. Consider  the composite  map
  \[  Q_X \rightarrow Q_X/Z_Y \rightarrow \GL(X) \times \U(X^{\perp}/X) \rightarrow \GL(X). \]
  Its image is contained  in  the parabolic subgroup $P_{d, a-d}$ stabilizing the subspace $X \cap X^\tau \subset X$, and contains the subgroup $\GL(Y) \subset \GL(X \cap X^{\tau})$.
Moreover, the image of $N_{Q_X}/Z_Y$ is the unipotent radical of this parabolic subgroup of $\GL(X)$.
\end{enumerate}

As in the proof of Theorem \ref{ps1},  these properties and Lemma \ref{Weil-Jac} imply that:
\begin{eqnarray*}
& & \Ext^i(\pi_X, \omega_{V,\psi}) \\
  & \cong & \Ext^i_{\U(V)}[  \ind_{Q_X}^{\U(V)}  (\pi_1 \otimes \pi_2 \otimes \delta_{P_X/Q_X}^{1/2}), \omega_{V,\psi}] \\
  & \cong & \Ext^i_{Q_X/Z_Y}[ \delta_{Q_X}^{-1/2}(\omega_{V,\psi^-})_{Z_Y},
    (\pi_1 \otimes \pi_2 \otimes \delta_{P_X/Q_X}^{1/2})^\vee ]\\
  &\cong & \Ext^i_{Q_X/N_{Q_X}}[ (\pi_1)_{d,a-d} \otimes (\pi_2) \otimes \delta_{P_X/Q_X}^{1/2}),  \delta_{Q_X}^{1/2} \cdot |\det|^{-d/2} \cdot \alpha^{-1} \cdot \omega_{n-2a,\psi}] \\
  &\cong &  \Ext^i_{Q_X/N_{Q_X}}(A,B),  
\end{eqnarray*}
where  $ (\pi_1)_{d,a-d}$ denotes the un-normalized Jacquet module of $\pi_1$ with respect to the
parabolic subgroup $P_{d,a-d}$ of  $\GL(X)\cong  \GL_a(L)$ stabilizing $X \cap X^{\perp}$.
Note also that at elements $(g,1) \in  \GL_d(L) \times \GL_{a-d}(L)$. one has
\vskip 5pt
\[ \text{ $\delta_{P/Q}(g,1) = |\det(g)|^a$ and $\delta_Q(g,1)= |\det(g)|^{n-2a}$,} \]
and 
\begin{eqnarray*}
    (\pi_1)_{d,a-d} & = & \delta_{P_{d,a-d}}^{(1+\epsilon)/2} \times {\rm tempered ~~ representation~~ of~~ } \GL_d(L) \times \GL_{a-d}(L) \\
  & = & |\det(g)|^{(d-a+\epsilon)/2 } \times {\rm tempered ~~ representation~~ of~~} \GL_d(L) \times \GL_{a-d}(L).
  \end{eqnarray*}
  Here we have used the fact that the normalized Jacquet module of
a tempered representation is a tempered representation up to multiplication by a character $\delta_{P}^{\epsilon/2}$ for some
 positive real number $\epsilon$.  
\vskip 5pt

To show the vanishing of $\Ext^i(A,B)$, we note that 
$Q_X/N_{Q_X}$ contains $\GL(Y)$ as a direct factor  and the above information allows one to see that
the center of $\GL(Y)$ acts on the two representations $A$ and $B$ via different characters. Therefore, we have shown that
  \[\Ext^i_{\U(V)}[ \pi_X, \omega_{V,\psi}] = 0 \quad \text{  for all $i \geq 0$.} \]

\vspace{4mm}
\noindent {\bf Case 2:} $X \cap X^\tau = 0$, but $Z_0= X\cap X^{\tau \perp} \not = 0$.
\vspace{4mm}

In this case,  $X+X^\tau$ is a degenerate
skew-Hermitian space defined over $E$ whose nullspace is  $Z_0+Z_0^\tau$, i.e.
\[ Z_0+Z_0^\tau = (X+X^\tau) \cap   (X+X^\tau)^\perp.\]
 Let $Z$ be the subspace of $V$ such that $Z\otimes L =Z_0+Z_0^\tau$, so that
$Z$ is an isotropic subspace of $V$.  In this case, it is easy to see that  the subgroup $Q_X$ of
$\U(V)$ preserving $X$ has the following properties:

\begin{enumerate}

\item $Q_X$ contains the unipotent radical of the parabolic subgroup of $\U(V)$ stabilizing the isotropic subspace $Z\subset V$.
\vskip 5pt

\item  The image of the natural map from $Q_X$ to $\GL(X)$
given as the composite:
  \[  Q_X \rightarrow \GL(X) \times \U(X^\perp/X) \rightarrow \GL(X) \]
   lands inside
  the parabolic subgroup defined by the subspace $Z_0= X\cap X^{\tau \perp}  \subset X$, containing the unipotent radical of this parabolic subgroup
  of $\GL(X)$,  as well as
   $\GL(X \cap X^{\tau \perp}) $.

\end{enumerate}
\vskip 5pt

\noindent  A similar analysis as in Case 1 (based on appropriate central character  analysis) allows us to conclude that:
\[\Ext^i_{\U(V)}[ \pi_X, \omega_{V,\psi}] = 0,\]
  for all $i \geq 0$.

\vspace{4mm}
\noindent {\bf Case 3:} Both $X \cap X^\tau = 0$,  and $ X\cap X^{\tau \perp} = 0$.

\vskip 5pt

 In this case,  
the $\U(V)$-orbit of $X$ is open.  Such isotropic spaces $X \subset V_K$, up to $\U(V)$-conjugacy,   are in bijective correspondence with $\U(V)$-conjugacy classes of non-degenerate subspaces $W\subset V$ of dimension $2a$,
since such a subspace $W$ has, up to $\U(W)$ conjugacy, a unique subspace $X\subset W_K$ such that  
$X \cap X^\tau = 0$  and $ X\cap X^{\tau \perp} = 0$ (the proof of this is given in Lemma \ref{Weil-L} below).
\vspace{4mm}

Now let $Q_X$ be the stabilizer of $X$ in $\U(V)$. 
 The following lemma allows us to determine this stabilizer:

 \begin{lemma} \label{Weil-L}
     Let $W$ be a  $2a$-dimensional   nondegenerate skew-Hermitian space over $E$ and let  $X \subset W \otimes_F K =W_K$ be  an isotropic subspace of $W_K$ such that  
  \[  \text{$X \cap X^\tau = 0$  and  $  X + X^\tau = W_K$,} \]
  where we recall that 
     $\Gal(K/F) = \langle \tau \rangle$.  
  Then we have:
  \vskip 5pt
  
  \begin{itemize}
  \item[(i)]  The isotropic subspace $X \subset W_K$ with the above properties is unique up to the action of $\U(W)$ on $W_K$;
  
  \item[(ii)] The stabilizer of $X$ in $\U(W)$ is isomorphic to $\U(W_X)$, 
  where $W_X$ is the skew-Hermitian space on the underlying vector space $X$ relative to $L/E'$ (for
 $E'$ the third  quadratic field contained in the biquadratic extension $L = E \otimes K$), defined by:
  \[  (x_1,x_2) = \langle x_1, \tau x_2\rangle.\]
  \vskip 5pt
  
  \item[(iii)]   The determinant of the  $2a$-dimensional   skew-Hermitian space $W$ for $E/F$ and the $a$-dimensional
  skew-Hermitian space $W_X$ for $L/E'$ are related by (as elements of $F^\times/N_{E/F}( E^\times)$):
  \[ \det (W) = N_{L/E}[ k^a \det (W_X)] = (-k^2)^a  N_{L/E} \det (W_X) ,\]
 where $k$ is any nonzero element of $K$ whose trace to $F$ is zero.
\end{itemize}
Further, the restriction of the Weil representation $\omega_{W,\psi, \mu}$ of $\U(W)$ to $U(W_X)$  is the Weil representation of $\omega_{W_X,   \psi \circ Tr_{E'/F},  \mu \circ N_{L/E}}$.
    \end{lemma}
\begin{proof}
(i) Let $X$ and $X'$ be two $L$-vector subspaces of $W_K$ satisfying the properties in the Lemma. Let 
\[  \phi: X \longrightarrow X' \]
 be a $L$-linear isomorphism of vector spaces. 
Then $\phi$ extends  uniquely to a $L$-linear automorphism (still denoted by $\phi$) of $X+X^{\tau} = W_K$ defined  by 
\[  \phi(x^\tau) = \phi(x)^\tau, \quad \text{ for $x \in X$.} \] 
Since this   satisfies $\phi \circ \tau = \tau \phi$, it follows by Galois descent that $\phi$ is defined over $E$, i.e $\phi$ is obtained by base change from an $E$-linear isomorphism 
\[ \phi_0: W \longrightarrow W. \] 
Moreover, one checks by a direct computation that  $\phi_0$ preserves the given skew-Hermitian structure on $W$ if and only if $\phi$ is compatible with the $L/E'$-skew-Hermitian structure on $X$ and $X'$ defined in (ii), i.e. $\phi$ is an isomorphism of $L/E'$-skew-Hermitian spaces:
\[  \phi: W_X \longrightarrow W_{X'}. \]
Hence, to show that there is an element of $\U(W)$ which carries $X$ to $X'$, it remains to show that $W_X$ and $W_{X'}$ are necessarily isomorphic as $L/E'$-skew-Hermitian spaces. We shall show this and hence complete the proof of (i) only after we demonstrate (iii), using Lemma \ref{disc}  below.
\vskip 5pt

\noindent (ii) Taking $X' = X$ in (i) above, one deduces that the stabilizer of $X$ in $\U(W)$ is precisely $\U(W_X)$. 

\vskip 5pt

\noindent  (iii) For the $L/E'$-skew-Hermitian space $W_X$ defined in (ii), 
there is a natural $E/F$ skew-Hermitian structure on $W_X$ obtained by
taking the same vector space as $W_X$, now treated as an $E$-vector space and denoted by $R_E(W_X)$, with the skew-Hermitian form which
is the $L/E$-trace of the skew-Hermitian form on $W_X$.
Define a map $\phi: X\rightarrow W$ by 
\[  \phi(x)=x+ x^\tau \in W \quad \text{  for $x \in X$.} \]  
  It is easy to check that $\phi$ induces an isomorphism of the $E/F$ skew-Hermitian spaces $R_E(W_X)$ and $W$.  Now we appeal to  the Lemma \ref{disc} below  to complete the proof of (iii).
  \vskip 5pt

\begin{lemma} \label{disc} With the quadratic extensions $E,K,E'$ of $F$ as before, let $\mathcal{W}$ be an  $L/E'$-skew-Hermitian space with a skew-Hermitian form $\langle -, - \rangle $.
  Let $R_{E}(\mathcal{W})$ be the same space  $\mathcal{W}$ regarded as a vector space over $E$, 
  which comes equipped with a natural $E/F$ skew-Hermitian structure $(-,-)$:  
   \[ (w_1,w_2) = \langle w_1, w_2 \rangle  + \langle w_1, w_2 \rangle^\tau.\]

\noindent Fix an element  $k \in K^{\times}$   with ${\rm Tr}_{K/F} (k) = 0$.   Then, with $a = \dim \mathcal{W}$, one has:
\[  k^a \det \mathcal{W} \in E'^\times, \]  
and
\[ N_{L/E} (k^a \det \mathcal{W}) {=} N_{E'/F } (k^a \det \mathcal{W})  = \det R_E(\mathcal{W}),\]
as elements of $F^\times/ N_{E/F}( E^\times)$.
\vskip 5pt

Moreover, if $\mathcal{W}' \ncong \mathcal{W}$ is the other $L/E'$-skew-Hermitian space space of the same dimension, then $R_E(\mathcal{W}') \ncong R_E(\mathcal{W}')$.
\end{lemma}

\vskip 5pt

\begin{proof}
  By writing $\mathcal{W}$ as
  an orthogonal sum of lines over $L$, we are reduced to prove the Lemma for a 1-dimensional  skew-Hermitian space
  for $L/E'$ which we take to be the vector space $L$ with the skew-Hermitian structure: 
  \[ \langle \ell_1, \ell_2 \rangle =  \ell_1 x\ell_2^{\sigma \tau},\]
  with $x \in  L^\times$ with $x+x^{\sigma \tau}=0$.

  This gives rise to an $E/F$ skew-Hermitian structure on $L$ by:
  \[ (\ell_1, \ell_2) = \langle \ell_1, \ell_2 \rangle  + \langle \ell_1, \ell_2 \rangle^\tau =
 \ell_1 x\ell_2^{\sigma \tau} +  \ell_1^\tau x^\tau \ell_2^{\sigma}.\]

 For this $E/F$-skew-Hermitian  space $L$, $\{1,k\}$ is a basis, for which the Gram matrix is given by
$$ A = \left ( \begin{array}{cc} 
  x+ x^\tau & -k(x-x^\tau)    \\
k(x-x^\tau) & -k^2(x+x^\tau)
\end{array}
\right ) ,$$
so that
\[ \det A = -4k^2xx^\tau.\]
  Now  since $(kx)^{\sigma \tau}= kx$, we see that $kx$ belongs to $E'^\times$, as desired.
  \vskip 5pt
  
  For the final statement,  it suffices to show that 
  \[  \det(R_E(\mathcal{W}')) \ne \det(R_E(\mathcal{W})) \in F^{\times}/ N_{E/F}(E^{\times}), \]
   or equivalently that
  \[  \omega_{E/F} (   \det(R_E(\mathcal{W}')) \cdot det(R_E(\mathcal{W}))) = -1. \]
  By the identity proved above,  this is equivalent to showing that
  \[ ( \omega_{E/F} \circ N_{E'/F} ) \left( \det(\mathcal{W}') \cdot \det(\mathcal{W}) \right) = -1. \]
  But this desired identity holds since
  \[  \omega_{E/F} \circ N_{E'/F} = \omega_{L/E'}.  \]
    This completes the proof of Lemma \ref{disc}.
\end{proof}

\vskip 5pt

As we mentioned, Lemma \ref{disc} completes the proof of (iii). The last assertion in Lemma \ref{disc} also allows us to complete the proof of (i).  Indeed, in the proof of (iii), we have shown that $R_E(W_X) \cong  W$ as $E/F$-skew-Hermitian spaces.  Hence, with $X$ and $X'$ as in the proof of (i),  we deduce that $R_E(W_X) \cong R_E(W_{X'})$. In view of the last assertion in Lemma \ref{disc}, one thus deduces that $W_X \cong W_{X'}$ as $L/E'$-skew-Hermitian spaces. 
\vskip 5pt

Finally, we  observe that the restriction of the Weil representations made in  Lemma \ref{Weil-L}  is the
precise version of the well-known assertion  that the restriction of a Weil representation of $\Sp(4n,F)$ to $\Sp(2n,E')$  takes
  a Weil representation of $\Sp(4n,F)$ to a Weil representation of $\Sp(2n,E')$.  This completes the proof of Lemma \ref{Weil-L}.
  \end{proof}

\vskip 5pt

  Applying  Lemma \ref{Weil-L},  we find 
\[  Q_X \cong \U(W_X) \times \U(W_X^{\perp}) \subset \U(V) \]
with $\dim W_X =2a$ and $\dim W_X^{\perp} = n-2a$.
We can now conclude the proof as in Theorem \ref{ps1}. 
\vskip 5pt

The proof of Theorem \ref{ps} is now complete.

   \end{proof}

The following proposition is obtained as a corollary to Theorem \ref{ps}.

\begin{prop} \label{coro}
  Let $V$ be an $n$-dimensional skew-Hermitian space relative to $E/F$,
  and $V_K=V\otimes_F K = V\otimes _EL$ its base change to an
  $n$-dimensional skew-Hermitian space relative to $L/K$.
 
 \begin{itemize}
 \item Let $V_K= X + X^\tau + W'_K $ with $X$ an isotropic subspace of $V_K$ with
   $X\cap X^\tau = 0$. Assume that both $(X+X^\tau)$ and $W'_K = W'\otimes_E L$ are defined over $E$,
   are non-degenerate skew-Hermitian spaces over $E$, and  are perpendicular to each other. Let 
  $ P  = MN$ be a maximal parabolic subgroup of $\U(V_K)$ which is the stabilizer of $X$,  with Levi factor
 \[ M \cong \GL(X) \times \U(W'_K). \]
\item  Let
 \[  \pi = \pi_1 \rtimes \pi_2 = \Ind_P^{\U(V_K)} (\pi_1 \otimes \pi_2) \]
  be  a tempered principal series representation of $\U(V_K)$, with $\pi_1 \in {\rm Irr}(\GL(X))$ and $\pi_2 \in {\rm Irr}(\U(W'_K))$.  
  \end{itemize}

 By Lemma \ref{Weil-L}, the vector space $X$ over $L$ carries a
 natural $L/E'$-skew-Hermitian structure
(where $E'$ is the quadratic extension of $F$ inside $L$ different from $E,K$), that we denote by $W$ (so $W$ as a vector space over $L$ is the same as $X$).
 If Conjecture \ref{conj-local-general}(i)-(iii) holds  for

 \begin{enumerate} \item the representation
   $\pi_1$ of $\GL(X)$ containing the unitary subgroup $\U(W)$,
   of size $a$  for the extension $L/E'$,
   \item  $\pi_2 \in {\rm Irr}(\U(W'_K))$,
\end{enumerate}
   then it holds also  for the representation
$ \pi = \pi_1 \rtimes \pi_2$ of $\U(V_K)$.  
\end{prop}
\begin{proof}
That the Conjecture \ref{conj-local-general}(i) holds for  the representation
$ \pi = \pi_1 \rtimes \pi_2$ of $\U(V_K)$ if and only if it does for both the
representations $\pi_1$ and $\pi_2$ is the content of our previous theorem.

We will next prove the analogous assertion on Conjecture \ref{conj-local-general}(ii). For this, let the representations of the Weil-Deligne group of $L$
associated to $\pi_1,\pi_2$ be $M_1, M_2$. Then the parameter of the representation $\pi$ of $\U(V_K)$
is $M= M_1 + {^\sigma M_1}^\vee + M_2$. We need to prove that if the  equations (1) and (2) below hold, then so does the equation (3). Here is equation (1):

\begin{align} 
& \, \mu(\det(W'))   \notag \\
=&\, \epsilon( 1/2,   {\rm As}_{L/E}(M_2) \otimes \mu^{-1}, \psi_E) \cdot \det ({\rm As}_{L/E}(M_2))(e)
  \cdot  \omega_{K/F}(e^2)^{(n-2a)(n-2a-1)/2}  \notag \\ 
  =& \, \epsilon( 1/2,   {\rm As}_{L/E}(M_2) \otimes \mu^{-1}, \psi_E) \cdot \det (M_2)(e)^{n-2a}
  \cdot  \omega_{K/F}(-1)^{(n-2a)(n-2a-1)/2} \notag  \\
  \stackrel{(1)}{=} & \,  \epsilon( 1/2,   {\rm As}_{L/E}(M_2) \otimes \mu^{-1}, \psi_E) \cdot \det (M_2)(e)^{n}
  \cdot  \omega_{K/F}(-1)^{n(n-1)/2} \cdot \omega_{K/F}(-1)^{a} \notag
\end{align}
where we have used the observation that $\det M_2$ is a character of $L^\times/K^\times$, hence is trivial on $e^2$. Here is the equation (2):

 
\begin{align} 
& \, \mu(N_{L/E} \det(W_X)) \notag \\
  =&\,  \epsilon( 1/2,   M_1 \otimes {}^{\sigma \tau} \!  M_1 {}^{\vee} \otimes \mu^{-1} \circ N_{L/E}, \psi_L) \cdot \det (    M_1 )(-1)^a \cdot  \omega_{L/E'}(-1)^{a(a-1)/2} , \notag \\
\stackrel{(2)}{=} &\,  \epsilon( 1/2,   M_1 \otimes {}^{\sigma \tau} \!  M_1 ^{\vee} \otimes \mu^{-1} \circ N_{L/E}, \psi_L) \cdot \det (    M_1 )(-1)^a , \notag
\end{align}
where  $\mu^{-1} \circ N_{L/E}$ denotes the character of $L^\times$ obtained from the character $\mu^{-1}$ of $E^\times$ by composing with the norm map $N_{L/E}: L^\times \rightarrow E^\times$, and $ \psi_L$ is the character of $L$ obtained
from the character $\psi_E$ of $E$ obtained by composing with the trace map from $L$ to $E$. The  equation (3) is:

 
 \begin{align} 
 & \, \mu(\det(V)) \notag \\
=& \,  \epsilon( 1/2,   {\rm As}_{L/E}(M) \otimes \mu^{-1}, \psi_E) \cdot \det ({\rm As}_{L/E}(M))(e) \cdot  \omega_{K/F}(e^2)^{n(n-1)/2} \notag  \\
=& \,  \epsilon( 1/2,   {\rm As}_{L/E}(M) \otimes \mu^{-1}, \psi_E) \cdot \det(M)(e)^{n} \omega_{L/E}(e)^{n(n-1)/2} \cdot  \omega_{K/F}(e^2)^{n(n-1)/2} \notag  \\
= & \,  \epsilon( 1/2,   {\rm As}_{L/E}(M) \otimes \mu^{-1}, \psi_E) \cdot  \det(M)(e)^{n}   \cdot  \omega_{K/F}(-1)^{n(n-1)/2} \notag \\
 \stackrel{(3)}{=} & \,
  \epsilon( 1/2,   {\rm As}_{L/E}(M) \otimes \mu^{-1}, \psi_E) \cdot  \det(M_1)(-1)^n \cdot \det(M_2)(e)^{n}   \cdot  \omega_{K/F}(-1)^{n(n-1)/2}.   \notag 
  \end{align}

\vskip 5pt

The proof that equations (1) and (2) imply equation (3)
depends essentially on relating $\epsilon( 1/2,   {\rm As}_{L/E}(M) \otimes \mu^{-1}, \psi_E)$
to  $\epsilon( 1/2,   {\rm As}_{L/E}(M_2) \otimes \mu^{-1}, \psi_E)$
and   $\epsilon( 1/2,   M_1 \otimes {}^{\sigma} \!M_1 ^{\vee} \otimes \mu^{-1}\circ N_{L/E}, \psi_L)$,
given that  $M= M_1 + {}^{\sigma} \!M_1^\vee + M_2$ with
 ${}^{\sigma} \!M_2^\vee = M_2$. We begin with the following calculation.

\begin{eqnarray*}
  {\rm As}_{L/E}(M) & =  & {\rm As}_{L/E}(M_1) + {\rm As}_{L/E}({}^{\sigma} \!M_1 ^\vee ) +
  {\rm As}_{L/E}(M_2)  \\
  &{}  & + \Ind_L^E( M_1 \otimes M^\tau_2)
  + \Ind_L^E(  {}^{\sigma} \!M_1^\vee  \otimes M^\tau_2)
  + \Ind_L^E( M_1 \otimes {}^{\sigma \tau}\! M_1^\vee) \\
  & \stackrel{(4)}{=}  & {\rm As}_{L/E}(M_1) + {\rm As}_{L/E}({}^{\sigma} \!M_1 ^\vee ) + \Ind_L^E( M_1 \otimes M^\tau_2)
  + \Ind_L^E(  {}^{\sigma} M_1^\vee  \otimes M^\tau_2) \\
  & {} & +{\rm As}_{L/E}(M_2) + \Ind_L^E(  {}^{\sigma \tau}\! M_1^\vee  \otimes M_1)   
\end{eqnarray*}
Now, for any representation $N$ of $W_E$, one has
\[ \epsilon(N+ {}^{\sigma} \!N ^{\vee}, \psi_E) = \det(N)(-1).\]
Hence, we find (using a calculation on the determinant of the Asai representation ${\rm As}_{L/E}(M_1)  \otimes \mu^{-1}$) that
\begin{eqnarray*}
  \epsilon( [{\rm As}_{L/E}(M_1) + {\rm As}_{L/E}({}^{\sigma} \!M_1 ^\vee )] \otimes \mu^{-1}, \psi_E) & = & \det
          [{\rm As}_{L/E}(M_1)] (-1) \mu^{a^2}(-1) \\
          & = &  \det(M_1)^a(-1) \omega_{L/E}(-1)^{a(a-1)/2} \mu^{a^2}(-1) \\
& \stackrel{(5)}{=} &  \det(M_1)^a(-1) \mu^{a}(-1)  ,
\end{eqnarray*}
where in the last equality, we have used that  $\omega_{L/E}(-1)=1$ since $\omega_{L/E} = \omega_{K/F} \circ N_{E/F}$.  Similarly, using that $M_2 \cong {}^{\sigma} \!M_2 ^ \vee$,
and a calculation on the determinant of the induced representation $\Ind_L^E( M_1 \otimes M^\tau_2)$, we find that:
\begin{align}
&\, \epsilon ( [\Ind_L^E( M_1 \otimes M^\tau_2)    + \Ind_L^E(  {}^{\sigma} \!M_1^\vee  \otimes M^\tau_2)]  \otimes \mu^{-1}, \psi_E)  \notag \\
 = & \, \det [\Ind_L^E( M_1 \otimes M^\tau_2) \otimes \mu^{-1}](-1)  \notag \\
  =  & \,  \det ( M_1 \otimes M^\tau_2) (-1), \notag \\
   \stackrel{(6)}{=} &\,  \det ( M_1)(-1)^n, \notag
\end{align}
where in the second equality, we have used the facts that $\omega_{L/E}(-1)=1$, and  $\det M_2(-1) =1$.

By the inductive nature of the epsilon factors for representations of dimension 0, we have,
\begin{eqnarray*}
  \epsilon( \Ind_L^E(  {}^{\sigma \tau}\! M_1^\vee  \otimes M_1)    \otimes \mu^{-1}, \psi_E)
  & {=} & \epsilon(  {}^{\sigma \tau}\! M_1^\vee  \otimes M_1 \otimes \mu^{-1}\circ  N_{L/E} , \psi_L) \cdot \epsilon(\omega_{L/E}, \psi_E)^{a^2},\\
  & {=} & \epsilon(  {}^{\sigma \tau}\! M_1^\vee  \otimes M_1 \otimes \mu^{-1}\circ  N_{L/E} , \psi_L) \cdot \omega_{L/E}(e)^{a},\\
  & \stackrel{(7)}{=} & \epsilon(  {}^{\sigma \tau}\! M_1^\vee  \otimes M_1 \otimes \mu^{-1}\circ  N_{L/E} , \psi_L)
  \cdot \omega_{K/F}(-e^2)^{a},
\end{eqnarray*}

\noindent From equations (4), (5), (6), and (7), we see that
\begin{eqnarray*}
  \epsilon( {\rm As}_{L/E}(M) \otimes \mu^{-1}, \psi_E) & = &
\epsilon( [{\rm As}_{L/E}(M_2) + \Ind_L^E(  {}^{\sigma \tau}\! M_1^\vee  \otimes M_1)]    \otimes \mu^{-1}, \psi_E)  \\ & &  \cdot  \det(M_1)^a(-1) \mu^{a}(-1) \cdot 
\det ( M_1) (-1)^n,\\
& = & \epsilon( {\rm As}_{L/E}(M_2)\otimes \mu^{-1}, \psi_E)
\epsilon(  {}^{\sigma \tau}\! M_1^\vee  \otimes M_1 \otimes \mu^{-1}\circ  N_{L/E} , \psi_L)   \\
& {}&  \cdot  \det(M_1)^a(-1) \mu^{a}(-1) \cdot 
\det ( M_1) (-1)^n \omega_{K/F}(-e^2)^{a^2}, \\
& \stackrel{(8)}{=} & \epsilon( {\rm As}_{L/E}(M_2)\otimes \mu^{-1}, \psi_E)
\epsilon(  {}^{\sigma \tau}\! M_1^\vee  \otimes M_1 \otimes \mu^{-1}\circ  N_{L/E} , \psi_L)   \\
& {}&   \cdot \det(M_1)^{n+a}(-1) \mu^{a}(-1)  \omega_{K/F}(-e^2)^{a},
\end{eqnarray*}

\noindent By Equation (8), one sees that equations (1) and (2) imply equation (3), using the following identity from Lemma \ref{Weil-L} of elements of $F^\times/N_{E/F}( E^\times)$:
\[ (-k^2)^a \det W' \cdot N_{L/E} \det (W) = \det V,\]
as well as the relation of the characters $\omega_{K/F}$ and $\omega_{E/F}$ to the (quadratic) Hilbert symbol of $F$:
\begin{eqnarray*}
  \omega_{K/F}(x) & = & (k^2,x) \\
  \omega_{E/F}(x) & = & (e^2,x),
  \end{eqnarray*}
therefore,
\[   \omega_{K/F}(e^2)  =  (k^2,e^2) = \omega_{E/F}(k^2) .\]
\vskip 5pt

Finally, under the standard identification of the character of component groups under parabolic induction,
it is easy to see that the recipe in Conjecture \ref{conj-local-general}(iii) holds. This amounts to the identity (for $M_i = {}^{\sigma} \!M_i ^\vee$):
\[ \epsilon(\Ind_L^E(M_i^\tau \otimes [M_1 + {}^{\sigma} \!M_1^\vee]\otimes \mu^{-1}, \psi_{E,e}) = 1,\] 
which is easy to see.
\vskip 5pt

We have thus finished the proof of Proposition \ref{coro}.
 \end{proof}

\vskip 5pt

 \begin{remark}
  The arguments given here also prove that for any tempered representation $\pi$ of $\U(V_K)$ which is a direct summand of a representation  of $\U(V_K)$ parabolically induced from a unitary cuspidal representation of a Levi subgroup of $\U(V_K)$,
  \[\Ext^i_{\U(V)}[ \pi, \omega_{V,\psi}] = 0, ~~{\rm ~for ~all~} i \geq 1.\]
  This vanishing of higher Ext's is as proposed in \cite{P3}, but is not as precise as Theorem \ref{Chen}.
  \end{remark}

  \subsection{\bf $\U(V)$-orbits on the full flag variety}
  Using Theorem \ref{ps}, and  its corollary, Proposition \ref{coro}, one can inductively deduce Conjecture \ref{conj-local-general}(i)-(iii)
  for irreducible  unitary principal series representations of $\U(V_K)$ induced from a Borel subgroup.  
 However,  we shall give an alternative treatment  involving the analysis of the $\U(V)$-orbits on the  full flag variety of $\U(V_K)$, which  has a rather nice structure that
  may be of independent interest.   \vskip 5pt

  \begin{prop}  \label{P:orbits}
  Let
  \begin{itemize}
  \item   $L = E \otimes K$ be a biquadratic extension and let $E'$ be the third quadratic subfield of $L$.
\item  $V$ be a skew-Hermitian space relative to $E/F$ of dimension $n$,  
\end{itemize}
 For a skew-Hermitian space $W$ relative to $L/E'$, let $ {\rm Res}_{L/E}(W)$ be the same space $W$ regarded as a vector space over $E$ (of twice the dimension) together with the associated $E/F$-skew-Hermitian structure (obtained by taking the trace), so that
  \[  \U(W) \subset \U({\rm Res}_{L/E}(W)). \]
 Then we have the following.
  \vskip 5pt
  
  (i) If $\dim V = n  =2d$ is even, there are $2^{d-1}$ open $\U(V)$-orbits on the flag variety of $\U(V_K)$. The open orbits are parameterized by ordered  collection of lines
  \[  \mathcal{L} = \{ L_1, L_2,..., L_d \}, \]
   where each $L_i$ is a rank 1 skew-Hermitian space relative to $L/E'$,
  subject to the condition of $V$-relevance:
  \[  \det(V) = \prod_i \det( {\rm Res}_{L/E}(L_i)). \]
  The stabilizer group for the orbit corresponding to $\mathcal{L}$ is 
  \[  \U(\mathcal{L})  = \prod_i  \U(L_i) \subset \prod_i  \U( {\rm Res}_{L/E}( L_i))
  \subset  \U(V). \]
       \vskip 5pt
    
    (ii) Suppose that $\dim V = n = 2d+1$ is odd. There are $2^d$ open $\U(V)$-orbits on the flag variety of $\U(V_K)$. The open orbits are parameterized by ordered
    collection  
    \[ \mathcal{L} = \{ L_1, L_2,..., L_d ; V_0 \}, \]
where each $L_i$ is a rank 1 skew-Hermitian space relative to $L/E'$, and $V_0$ is a rank $1$ skew-Hermitian space relative to $E/F$, subject to the condition of $V$-relevance:
 \[  \det(V) = \prod_i \det( {\rm Res}_{L/E}(L_i)) \cdot \det(V_0) \]
 In particular, $V_0$ is determined by $\{ L_1,..., L_d \}$. The stabilizer group associated to $\mathcal{L}$ is
 \[  \U(\mathcal{L}) = \prod_i \U(L_i) \times \U(V_0) \subset \prod_i  \U( {\rm Res}_{L/E}( L_i)) \times \U(V_0) \subset   \U(V). \] 
  \end{prop}

\vskip 10pt

\subsection{\bf Unitary principal series}
Using Proposition \ref{P:orbits} and Theorem \ref{ps}, we can  study the restriction of a unitary principal series
\[  \Pi = \Ind_B^{\U(V_K)}  \chi   \]
to $\U(V)$ and  show:

\begin{thm} \label{T:ups2}
  Conjecture   \ref{conj-local-general}(i)-(iii) hold
  for the tempered L-packet 
consisting of the constituents of the unitary principal series representation $\Pi$.
\end{thm}

\begin{proof}
 The argument is similar to that of Corollary \ref{C:ps} and it will be convenient to treat the cases of even or odd $\dim V$ separately. 
We shall only write down the details for the case of even $\dim V$, leaving the odd case as an exercise for the interested reader.

\vskip 5pt

Assume thus that $\dim V = n = 2d$ is even, so that
\[  \Pi = \Ind_B^{\U(V_K)} (\chi_1 \otimes \cdots \otimes \chi_d) \]
for some unitary characters $\chi_i$ of $L^{\times}$.
By Theorem \ref{ps},    we see that
\begin{equation} \label{E:ups}
  \Hom_{\U(V)}(\Pi, \omega_{\psi,\mu, V}) \cong \bigoplus_{\mathcal{L}}  \bigotimes_i \Hom_{\U(L_i)}(\chi_i, \omega_{\psi, \mu, L_i^E}) \end{equation}
where the sum runs over $V$-relevant $\mathcal{L}$'s and we have written $L_i^E$ for $ {\rm Res}_{L/E}(L_i)$.
\vskip 5pt

We thus need to analyze the nonvanishing of $\Hom_{\U(L_i)}(\chi_i, \omega_{\psi, \mu, L_i^E})$.  
As in the $E = K$ case, this comes down to an application of the theorem of Moen and Rogawski, i.e. Theorem \ref{T:moen}. 
Indeed, by the functorial property of the Weil representation, the restriction of $\omega_{\psi,\mu, L_i^E}$ of $\U(L_i)$ is simply the Weil representation 
$\omega_{\psi_K, \mu \circ N_{L/E}, L_i}$ of $\U(L_i)$. Hence
\[ \Hom_{\U(L_i)}(\chi_i, \omega_{\psi, \mu, L_i^E}) \ne 0  \]
if and only if
\begin{equation} \label{E:moen}
\mu( N_{L/E}(\det(L_i))) = \chi_i(-1) \cdot \epsilon(1/2, \chi_i/ \chi_i^{\tau \sigma} \cdot (\mu \circ N_{L/E})^{-1}, \psi_L),\end{equation}
where the local root number is considered over $L$.  
\vskip 5pt

Hence, we see that at most one $\mathcal{L}$ in (\ref{E:ups}) has nonzero contribution, and this $\mathcal{L} = \{L_i \}$ is characterized by having (\ref{E:moen}) holding for all $i$. 
By Lemma \ref{disc},
 \[  \det(L_i^E) =  N_{L/E}( k \cdot \det(L_i))   \in F^{\times}/ N_{E/F}(E^{\times})  \]
where $k \in K_0^{\times}$. Thus, if $\Hom_{\U(V)}(\Pi, \omega_{\psi,\mu, V}) \ne 0$, then
\begin{align}  \label{A:V}
 \mu(\det(V))  &=\prod_i \mu(\det(L_i^E))  \notag \\
 &= \prod_i \mu( N_{L/E}(k \cdot \det(L_i))   \notag \\
&= \omega_{E/F}(-k^2)^d \cdot \prod_i \chi_i(-1) \cdot \prod_i  \epsilon(1/2, \chi_i/ \chi_i^{\tau \sigma} \cdot (\mu \circ N_{L/E})^{-1}, \psi_L), 
\end{align}
with the last equality following by (\ref{E:moen}).
\vskip 5pt

On the other hand, according to Conjecture \ref{conj-local-general}(iii), one should expect that
\[  \mu(\det(V)) =  \epsilon(1/2, {\rm As}_{L/E}(M) \cdot\mu^{-1}, \psi_E) \cdot \det({\rm As}(M))(e) \cdot \omega_{K/F}(e^2)^{n(n-1)/2}  \] 
where 
\[  M = \bigoplus_i M_i =  \bigoplus_i (\chi_i + (\chi_i^{\sigma})^{-1})   \]
is the L-parameter of $\Pi$.  Let us explicate this and compare it with the expression for $\mu(\det(V))$ in (\ref{A:V}).
\vskip 10pt

By Lemma \ref{L:Asai}(a), 
\[  {\rm As}_{L/E}(M) =  \bigoplus_i {\rm As}_{L/E}(M_i)  \oplus \bigoplus_{i<j} \Ind_L^E (M_i^{\tau} \otimes M_j), \]
 with $M_i = \chi_i + (\chi_i^{\sigma})^{-1}$. Likewise, by Lemma \ref{L:Asai}(a) and (c), 
 \[  {\rm As}_{L/E}(M_i) = \chi_i|_{E^{\times}} + (\chi_i^{\sigma})^{-1}|_{E^{\times}}  + \Ind_L^E \chi_i/ \chi_i^{\tau \sigma}, \]
and it follows that 
\begin{align} 
 &\epsilon(1/2,  {\rm As}_{L/E}(M_i)  \cdot \mu^{-1}, \psi_E)  \notag \\
 = &\chi_i(-1) \cdot \omega_{E/F}(-1) \cdot \epsilon(1/2, \Ind_L^E \chi_i/ \chi_i^{\tau \sigma} \cdot \mu^{-1}, \psi_E) \notag \\
 =&\chi_i(-1) \cdot \omega_{E/F}(-1) \cdot  \omega_{K/F}(-e^2)  \cdot \epsilon(1/2, \chi_i/ \chi_i^{\tau \sigma} \cdot (\mu \circ N_{L/E})^{-1}, \psi_L). \notag 
 \end{align} 
 In the above computation, we have used repeatedly the facts:
\begin{enumerate}
\item $  \epsilon(1/2, N + (N^{\sigma})^{\vee}, \psi_E) =\det(N)(-1).$
\item
 $\epsilon(1/2, \Ind_L^E N, \psi_E) = \epsilon(1/2, N, \psi_L) \cdot \epsilon(1/2, \omega_{L/E},\psi_E)^{\dim N}.$
\item$ \epsilon(1/2, \omega_{L/E},\psi_E) = \omega_{L/E}(e) = \omega_{K/F}(-e^2),$
since $\omega_{L/E}$ is a conjugate-orthogonal character of $E^{\times}$.
\end{enumerate}
\vskip 5pt

\noindent For $i< j$, a similar computation using the above facts shows that
\[ 
\epsilon(1/2, \Ind_L^E (M_i^{\tau} \otimes M_j), \psi_E)   =1 \]
Hence, we have:
\begin{align} 
\epsilon(1/2, {\rm As}(M) \cdot\mu^{-1}, \psi_E) = &\prod_i \chi_i(-1) \cdot \omega_{E/F}(-1)^d \cdot  \omega_{K/F}(-e^2)^d    \cdot  \notag \\
&\prod_i \epsilon(1/2, \chi_i/ \chi_i^{\tau \sigma} \cdot (\mu \circ N_{L/E})^{-1}, \psi_L)  \notag
\end{align}
On the other hand, using Lemma \ref{L:Asai}(d), 
 \[ \det({\rm As}(M))(e) \cdot \omega_{K/F}(e^2)^{n(n-1)/2} = \omega_{K/F}(-1)^d. \]
 Hence, Conjecture \ref{conj-local-general}(iii) predicts that
 \[  \mu(\det(V)) = \prod_i \chi_i(-1) \cdot \omega_{E/F}(-1)^d \cdot \omega_{K/F}(e^2)^d \cdot \prod_i \epsilon(1/2, \chi_i/ \chi_i^{\tau \sigma} \cdot (\mu \circ N_{L/E})^{-1}, \psi_L). \]
 Comparing this with (\ref{A:V}) and noting that
 \[  \omega_{E/F}(k^2) = (e^2, k^2) = \omega_{K/F}(e^2), \]
 we see that Conjecture \ref{conj-local-general}(iii) holds for the L-packet defined by unitary principal series representations of $\U(V_K)$.
\end{proof}

The reader will notice that we have not shown Conjecture  \ref{conj-local-general}(iv). For this, one would need to explicate which irreducible summand of the unitary principal series representation $\Pi$ has nonzero contribution to $\Hom_{\U(V)}(\Pi, \omega_{V, \mu, \psi})$. The different summands of $\Pi$ can be distinguished from each other by the effects on the normalized standard intertwining operators (i.e. the so-called local intertwining relations). We do not know how to exploit this to establish Conjecture  \ref{conj-local-general}(iv). However, in a paper \cite{CG} of Rui Chen and the first author, this remaining issue is taken care of by means of theta correspondence.

 \vskip 15pt

\section{\bf When $E \ne K$; Global case}  \label{S:global-general}
In this final section, we will formulate the global conjecture in the general case where $E \ne K$ are two distinct quadratic extensions of a global field  $F$. 
 We will use the notations of \S \ref{SS:global} in this global setting.

\vskip 10pt

Let $\Pi$ be a cuspidal automorphic representation of $\U(V_K)$ with a generic global L-parameter $M_{\Pi}$, so that 
\[ M_{\Pi} = \bigoplus_{i=1}^d M_i, \]
is a sum of conjugate-dual cuspidal representations $M_i$ of $\GL_{m_i}(L \otimes \A_F)$
of sign $(-1)^{n-1}$ where $L= E \otimes_F K$.
Now  we have the global period integral
\[  \mathcal{P}: \Pi \otimes \omega_{V,\psi,\mu} \longrightarrow \C \]
defined as in \S \ref{SS:global}. The global conjecture is:
\vskip 5pt

\begin{conj} \label{conj-global-general}
The global period integral $\mathcal{P}$ is nonzero if and only if (denoting $V_v=V\otimes_F F_v$)
\vskip 5pt

\begin{itemize}
\item[(a)] For all places $v$ of $F$,  $\Hom_{\U(V_v)}(\Pi_v, \omega_{V_v,\psi, \mu_v})  \ne 0$.
\vskip 5pt

\item[(b)]  the twisted Asai  automorphic L-function \cite{F} satisfies:
\[  L(1/2, \Pi, {\rm As}_{L/E} \times \mu^{-1}) \ne 0. \]
\end{itemize}
Further, if $    L(1/2, \Pi, {\rm As}_{L/E} \times \mu^{-1}) \ne 0$,
then there exists
a skew-Hermitian space  $V$ of dimension $n$ over $E$
such that the global period integral $\mathcal{P}$ is nonzero.

\end{conj}

\vskip 10pt

As in \S \ref{SS:global-refined}, after fixing decompositions of Tamagawa measures and Petersson inner products,  one expects a refined conjecture of the following form:
\[  \mathcal{P} \otimes \overline{\mathcal{P}} =  \frac{1}{|S_{\Pi}|} \cdot \frac{L(1, M^{\vee}_{\U(V_K)})}{L(1, M^{\vee}_{\U(V)})} \cdot \frac{L(1/2, \Pi, {\rm As}_{L/E} \times \mu^{-1})}{L(1, \Pi, Ad)} \cdot \prod_v  \mathcal{I}_v^{\#} \]
where 
\vskip 5pt

\begin{itemize}
\item $\mathcal{I}_v^{\#}$ is a normalized local period integral
\[   \mathcal{I}_v^{\#} = 
 \frac {L(1, M^{\vee}_{\U(V_v)})}{L(1, M^{\vee}_{\U(V_{K,v})})} \cdot \frac {L(1, \Pi_v, Ad)}{L(1/2, \Pi_v, {\rm As}_{L_v/E_v} \times \mu_v^{-1})} \cdot \mathcal{I}_v\]
with
\[  \mathcal{I}_v : \Pi_v \otimes \overline{\Pi_v} \otimes \overline{\omega_{\psi_v, \mu_v, V_v} } \otimes \omega_{\psi_v, \mu_v, V_v} \longrightarrow \C \]
defined by the integral of matrix coefficients as in (\ref{E:local-integral}).  
\item $M_{\U(V)}^{\vee}$ and $M_{\U(V_K)}^{\vee}$ are the dual of the motives of $\U(V)$ and $\U(V_K)$ respectively. 
\item $ |S_{\Pi}| = 2^d $, with $M_{\Pi} = \oplus_{i=1}^d M_i$.
\end{itemize}
\vskip 5pt

Observe that for this global conjecture, all the local  possibilities for $(E_v, K_v)$ will occur. 
It is conceivable that one can develop a relative trace formula, as in the case of GGP, to address the global conjectures here. This is being pursued by Danielle Wang, a PhD student of Wei Zhang at MIT.  
\vskip 10pt

Just as for Conjecture \ref{conj-refined-global}, we can verify the above refined conjecture when $\dim_E V =1$. More precisely, suppose one starts with a Hecke character $\chi$ of $\A_{L^{\times}}$, so that $\chi|_{\A_L^1}$ is an automorphic character of $\U(V_K)$.  Then, as in the verification of  Conjecture \ref{conj-refined-global} in \S \ref{SS:global-rank1}, the period $\mathcal{P}$ is the (conjugate of) a global theta lift of $\chi|_{\A_E^1}$ from $\U(V)$ to $\U(W)$, where $W = \langle 1\rangle$ is a rank 1 Hermitian space for $E/F$.  Hence, for $\phi_1, \phi_2 \in \omega_{V,\psi,\mu}$, one has
\[    \mathcal{P}(\phi_1) \cdot \overline{\mathcal{P}(\phi_2)} \cdot \tau(\U(W)) = \langle \Theta(\phi_2, \chi|_{E_1}) ,  \Theta(\phi_1, \chi|_{E_1}) \rangle_{\U(W), {\rm Pet}}. \]
The RHS is computed by the Rallis inner product formula, which gives
\begin{equation} \label{E:refinedlast}   
\mathcal{P}(\phi_1) \cdot \overline{\mathcal{P}(\phi_2)}  = \frac{1}{2} \cdot \frac{L_E(1/2, (\chi^{\sigma} \chi^{-1})|_E \cdot \mu^{-1})}{L(1, \omega_{E/F})} \cdot \prod_v \mathcal{I}_v^{\#}(\chi,\chi, \phi_1,\phi_2), \end{equation}
taking note that the Tamagawa number $\tau(\U(W))$ is equal to $2$. 
This is precisely what the refined conjecture says in this case, since $|S_{\chi|_{L_1}}| =2$ here and 
\[  {\rm As}_{L/E}(\chi^{\sigma} \chi^{-1}) =( \chi^{\sigma} \cdot \chi^{-1})|_E \]
by Lemma \ref{L:Asai}(d).
\vskip 5pt

 It is interesting to note that, as a sesquilinear form on $\omega_{V,\psi,\mu}$,  the LHS of (\ref{E:refinedlast}) is exactly the same as the LHS of (\ref{E:refinedfirst}) (assuming that 
 the character $\chi$ restricts to the same character on $\A_E^1$ in the two cases). Moreover the two ratio of L-values appearing on the RHS of (\ref{E:refinedfirst}) and (\ref{E:refinedlast}) are exactly the same. Hence the reader may wonder why  there is a factor of $1/2$ on the RHS of (\ref{E:refinedlast}) but there is none on the RHS of (\ref{E:refinedfirst}).  The reason is that the adelic periods $\mathcal{I}^{\#} = \prod_v \mathcal{I}_v^{\#}$ on the RHS of both equations are different. Indeed, in view of (\ref{E:peter}), the adelic period $\mathcal{I}^{\#}$ in (\ref{E:refinedfirst})  is defined relative to the Petersson inner product $\langle -, - \rangle_{\GL(V), {\rm Pet}}$ of $\GL(V)$ whereas that in (\ref{E:refinedlast}) is defined using the Petersson inner product $\langle -, -\rangle_{\U(V_K), {\rm Pet}}$ of $\U(V_K)$.  As inner products on the 1-dimensional vector space defined by $\chi$, the latter is twice the former, so that (\ref{E:refinedfirst}) and (\ref{E:refinedlast}) are consistent with each other.

\end{document}